\documentclass[11pt]{report}
\usepackage{amssymb ,amsfonts, amsmath, amscd, graphics, tikz, amsthm,}
\usepackage[a4paper, top=1in, bottom=1.5in]{geometry}
\usepackage{setspace}

\usepackage{amsthm}
\usepackage{color}                    
\usepackage{hyperref}                 
\usepackage{appendix}
\usepackage{tikz-cd}

\input xy
\xyoption{all}

\newtheorem{thm}{Theorem}[section]
\newtheorem{cor}[thm]{Corollary}
\newtheorem{lem}[thm]{Lemma}
\newtheorem{pro}[thm]{Proposition}

\theoremstyle{remark}
\newtheorem{rem}[thm]{Remark}

\theoremstyle{remark}
\newtheorem{exa}[thm]{Example}

\theoremstyle{definition}
\newtheorem{dfn}[thm]{Definition}

\theoremstyle{definition}

\theoremstyle{definition}
\newtheorem{cdn}[thm]{Condition}

\def\kalg{k\text{\textnormal{-alg}}}
\def\Hom{\text{\textnormal{Hom}}}

\def\kvar{k\text{\textnormal{-var}}}
\def\Ring{\text{\textnormal{Ring}}}
\def\Spec{\text{\textnormal{Spec}}}

\def\fsl{\mathfrak{sl}}

\def\R{\mathbb{ R}}

\def\N{\mathbb{ N}}

\def\F{\mathbb{ F}}

\def\Z{\mathbb{ Z}}
\def\Q{\mathbb{ Q}}

\def\C{\mathbb{ C}}

\def\P{\mathbb{ P}}
\def\A{\mathbb{ A}}
\def\H{\mathbb{ H}}
\def\U{\mathbb{ U}}
\def\*{{}^{*}}

\def\cP{\mathcal{P}}
\def\cU{\mathcal{U}}
\def\cL{\mathcal{L}}
\def\cM{\mathcal{M}}
\def\cC{\mathcal{C}}
\def\cH{\mathcal{H}}

\def\cF{\mathcal{F}}

\def\cK{\mathcal{K}}
\def\cN{\mathcal{N}}

\def\cO{\mathcal{O}}
\def\cE{\mathcal{E}}

\def\cS{\mathcal{S}}

\def\Ra{\Rightarrow}

\def\Lra{\Leftrightarrow}
\def\ra{\rightarrow}

\def\Cov{\text{\textnormal{Cov}}}
\def\Can{\text{\textnormal{Cov$^{an}$}}}
\def\lora{\longrightarrow}
\def\gl2q{\text{\textnormal{GL}}_2^+(\Q)}
\def\ggl2{\text{\textnormal{GL}}_2}
\def\GL{\text{\textnormal{GL}}}

\def\SL{\text{\textnormal{SL}}}
\def\sl2{\text{\textnormal{SL}}_2}
\def\psl2{\text{\textnormal{PSL}}_2}
\def\sv2{\text{\textnormal{SZ}}_2}
\def\z2{\text{\textnormal{Z}}_2}
\def\Hom{\text{\textnormal{Hom}}}
\def\pgl2{\text{\textnormal{PGL}}_2}

\def\m2z{\text{M}_2(\Z)}

\def\acl{\text{\textnormal{acl}}}
\def\cl{\text{\textnormal{cl}}}
\def\tp{\text{\textnormal{tp}}}
\def\ex{\text{\textnormal{ex}}}

\def\qftp{\text{\textnormal{qftp}}}
\def\cyc{\text{\textnormal{cyc}}}
\def\tor{\text{\textnormal{Tor}}}
\def\Loc{\text{\textnormal{Loc}}}
\def\Th{\text{\textnormal{Th}}}

\def\End{\text{\textnormal{End}}}
\def\trdeg{\text{\textnormal{trdeg}}}

\def\an{\text{\textnormal{an}}}

\def\Cov{\text{\textnormal{Cov}}}
\def\Spec{\text{\textnormal{Spec}}}
\def\Fib{\text{\textnormal{Fib}}}

\def\Gal{\text{\textnormal{Gal}}}
\def\Sets{\text{\textnormal{Sets}}}
\def\Aut{\text{\textnormal{Aut}}}

\def\Fet{\text{\textnormal{F\'et}}}
\def\Fin{\text{\textnormal{Fin}}}

\def\dim{\text{\textnormal{dim}}}
\def\dcl{\text{\textnormal{dcl}}}
\def\Fet{\text{\textnormal{F\'et}}}
\def\Feg{\text{\textnormal{F\'eG}}}
\def\Xx{\widetilde{X}_x}

\def\g0n{\Gamma_0(N)}

\def\lora{\longrightarrow}

\def\ggl2{\textrm{GL}_2}
\def\GL{\textrm{GL}}
\def\sl2{\textrm{SL}_2}

\def\sv2{\textrm{SZ}_2}
\def\z2{\textrm{Z}_2}
\def\pgl2{\textrm{PGL}_2}

\def\m2z{\textrm{M}_2(\Z)}

\title{Categoricity and covering spaces}
\author{Adam Harris}

\include{head}

\begin{document}

\begin{titlepage}
\newgeometry{top=2in}
\begin{center}
\textsc{\Huge Categoricity and covering spaces \\[2cm]}

\includegraphics[width=0.3\textwidth]{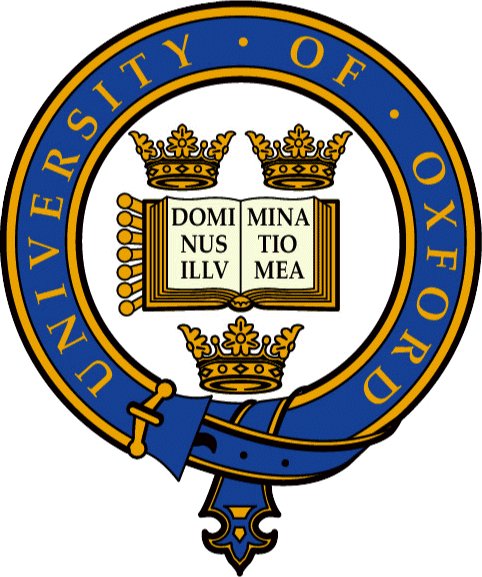}\\[2cm]

\textsc{\Large Adam Harris\\
        Mathematical Institute\\
        Merton College} \\[2cm]

{A thesis submitted for the degree of\\
        \textsc{DPhil in Mathematics} }\\[1.5cm]

{\normalsize University of Oxford, 2014}\\
{\normalsize Funded by EPSRC}

\end{center}

\end{titlepage}
\newgeometry{top=1in, bottom=1.5in}

\begin{center}
{\Large \textrm{Acknowledgements} \\[2cm]}
\end{center}

First and foremost I would like to thank my mum, dad, sister, nanna, and the rest of my family and friends for everything they have done for me over the years. \newline

Secondly I would like to thank my supervisors Jonathan Pila and Boris Zilber, a McCartney-Lennonesque supervisory duo which could only possibly be bettered by the Beatles themselves.\newline

I would like to offer a special thank you to Martin Bays and Chris Daw for patiently explaining many mathematical things to me over the course of my DPhil. I am also very grateful to Misha Gavrilovich for inviting me to Vienna where we had some very valuable discussions. \newline

Next I would like to thank the following people, who have contributed in some way to my mathematical education: Will Anscombe, Daniel Bertrand, Jamshid Derakhshan, Nikolaos Diamantis, Netan Dogra, Bas Edixhoven, Bernhard Elsner, Philipp Habegger, Martin Hils, Keith Hopcraft, Franziska Jahnke, Peter Jossen, Moshe Kamensky, Jochen Koenigsmann, Angus Macintyre, Martin Orr,  Anand Pillay, Kristian Strommen, Marcus Tressl, Jacob Tsimerman, Alex Wilkie, George Wilmers and Andrei Yafaev. I would also like to give a general thank you to everyone involved with the MathOverflow community, which has been a fantastic source of knowledge, and I am sure it will continue to be so. \newline

Finally I would like to thank the EPSRC and Merton College for funding me throughout my studies.

\newgeometry{top=1in, bottom=1.5in}


\tableofcontents

\chapter{Introduction}

One of the main areas of pure model-theoretic study is the abstract classification of theories. In the process of this classification, abstract properties of theories are isolated to create dividing lines between them. For example, these dividing lines can be in the form of stability properties, relating to the number of types of elements realised in models of the theory, or restrictions of the number of models of the theory. Armed with these abstract notions, we can look at natural structures arising in mainstream mathematics, and try to see where they lie in this classification. It may also be illuminating to know what good, abstract model-theoretic behaviour translates into in these more concrete situations.\newline

 In pure mathematics, we try and decide which statements hold in which mathematical structures. However, many of the most interesting of these statements are not easily translated into a reasonable model-theoretic setting. In some cases, showing that a particular problem can be embedded into a model-theoretic setting that is known to have good model-theoretic properties, can lead to spectacular results. Hrushovski's application of the theory of difference fields to the Manin-Mumford conjecture \cite{hrushovski2001manin}, and Pila's embedding of the Andr\'e Oort conjecture into an O-minimal setting \cite{pila2011minimality} are examples of such results. However, in addition to knowing that a statement of mathematics is true, it is also essential to have some conceptual justification of why it holds. The two applications above could be seen as providing proof and conceptual justification simultaneously. \newline

The aim of this thesis is to make a contribution to our understanding of the natural structures related to algebraic geometry. There is a particularly strong connection between model theory and algebraic geometry. One of the reasons for this is that the theory of algebraically closed fields of characteristic 0, i.e. the theory of the structure
$$\langle \C ,  + , \cdot, 0,1 \rangle ,$$
 lies right on the top of the hierarchy given by the model-theoretic classification of structures. It is what Zilber would describe as `logically perfect', that is, it has a unique model in every uncountable cardinality. Such structures are said to be \textit{`categorical in powers'}, and their defining characteristic is that you can recover such a mathematical structure uniquely from the data of its theory and its cardinality. \newline

The notion of categoricity was a guiding light in the early development of stability theory. Initially there was Morley's theorem, which says that if a first- order theory has a unique model in one uncountable cardinality, then it is categorical in all uncountable cardinalities. Then came the Baldwin-Lachlan characterisation of uncountably categorical theories, which says that a first-order theory is categorical in powers if and only if it is $\omega$-stable and has no Vaughtian pairs. Once this was all worked out, categoricity took a back seat as stability theory was developed by model theorists. However, soon Shelah started developing the model theory of more expressive extensions of first-order logic such as $\cL_{\omega_1,\omega}$, where you can form countable conjunctions and disjunctions, and Zilber realised that many interesting natural structures arising in mainstream mathematics can be studied in this infinitary setting. \newline

A complex algebraic variety may be viewed with its complex analytic topology, and any form of `comparison theorem', which allows the passage of information between the algebraic
and analytic categories is extremely useful. In the analytic category, by making use of the notion of path we can construct the universal covering space, which classifies all of the analytic covers of the complex analytic space. Grothendieck's functorial abstraction of the notion of universal cover (via the fibre functor), allows the construction of the analytic universal cover to be mimicked in the algebraic category, where the topology is too weak to talk about paths. This way of thinking is certainly not dissimilar to that of Zilber, whose outlook is that we should be able to capture essential properties of the analytic universal covering space by more algebraic means, and in some situations be able to recover the analytic object from this algebraic description. Zilber's so called `pseudo-analytic' structures lie in this gap between the analytic and algebraic worlds, and these kinds of structures occupy the content of this thesis. \newline

One of the most basic examples of one of these objects is the universal cover of the multiplicative group $\C^{\times}$ considered as a two-sorted structure
$$\langle \C , + ,0 \rangle \overset{\ex}{\lora} \langle \C, +  , \cdot , 0,1 \rangle,$$
where $\C^{\times}$ is considered as an algebraic variety embedded in the algebraically closed field $\langle \C, +  , \cdot , 0,1 \rangle$, and the covering sort is just a divisible, torsion-free, abelian group. That is, we have forgotten about nearly all of the complex analytic topology on the universal cover, leaving only its additive group structure. In \cite{zilber2006covers} and \cite{bays2011covers} it is shown that we can recover the analytic universal cover of $\C^{\times}$, as the unique model of an $\cL_{\omega_1,\omega}$-sentence stating that $\langle \C, +  , \cdot , 0,1 \rangle$ is an algebraically closed field of characteristic zero, and $\ex$ is a surjective homomorphism from the additive group $\langle \C, + ,0 \rangle$ onto the multiplicative group $\langle \C^{\times}, \cdot, 1\rangle$. The above statement may need some clarification for geometers. For example, we do not claim to be able recover a unique Riemann surface structure on either sort, because complex conjugation commutes with the exponential function and therefore lifts to an automorphism of the two sorted structure as a whole. What we mean is that we take the analytic universal cover and forget about all of the analytic structure except the Zariski topology on $\C^{\times}$, and the addition on the covering sort. The categoricity result then says that we may recover this model theoretic structure up to isomorphism from a purely algebraic description (i.e. the  $\cL_{\omega_1,\omega}$-sentence). The analogous result for the universal cover of an elliptic curve without complex multiplication has been exposed by Gavrilovich and Bays in \cite{gavrilovich2006model} and \cite{bays2009categoricity}. \newline

Unlike the first-order setting, there is no equivalent characterisation of categorical $\cL_{\omega_1,\omega}$-theories, and the situation is more delicate. The standard technique used demonstrate the categoricity of an $\cL_{\omega_1,\omega}$-sentence, is to show that it satisfies the conditions of `quasiminimal excellence'. This technique was pioneered by Zilber,  making use of Shelah's abstract notion of excellence. Quasiminimality can roughly be seen as saying that the definable sets in the models are countable or co-countable, or that the structure is homogeneous in the sense that you can extend a partial automorphism to one of the whole structure. The excellence condition is more difficult to describe, and to verify. However, recently there has been a major development in this area, and in \cite{bays2012quasiminimal} it is essentially shown the excellence condition follows from quasiminimality. In this thesis we will make use of this result, and one of the main themes here will be to axiomatise a natural mathematical structure in the logic $\cL_{\omega_1,\omega}$, and then show that we may recover this structure uniquely from this axiomatisation by showing that the theory is quasiminimal.\newline

The modular $j$-function is an analytic function
$$j: \H \lora \C,$$
mapping the upper half plane onto the complex numbers. It is invariant under the modular group $\psl2(\Z)$, and gives a complex analytic isomorphism
$$j: \psl2(\Z) \backslash \H \lora \C,$$
realising $\A^1(\C)$ as the moduli space of elliptic curves. In Chapter \ref{chapj} of this thesis, we will describe a model-theoretic setting for the study of the $j$-function, and verify that in this setting the $j$-function is another example of a `logically perfect' mathematical object. Furthermore, we show that the purely model-theoretic statement of the theory of the $j$-function being categorical is equivalent to deep algebro-geometric results regarding the openness of the images of certain Galois representations in Tate modules of products of elliptic curves. This demonstrates a tight interaction between model theory and geometry. For the case of a single elliptic curve, it is a celebrated result of Serre (\cite{serre1968abelian} and \cite{serre1971prop}) that these images are open, and the fact that categoricity is equivalent to these arithmo-geometric statements, might be seen as providing some model-theoretic justification as to why they hold.\newline

In the process of isolating algebraic conditions from which we may recover the $j$-function uniquely, it becomes clear that at the core of the $j$-function's good model-theoretic behaviour is the simple fact that it is a branched covering map. In Chapter \ref{chapcur} we investigate this further, and a more general model theoretic setting for the study of the universal cover of an arbitrary smooth complex algebraic curve is described. The motivation here was to provide a framework in which all of the previous categoricity results regarding covers of complex algebraic curves (and also that of the $j$-function) could be embedded. \newline

\section{Outline of this thesis}

In Chapter \ref{chapj}, a model-theoretic setting for the study of the $j$-function is described. An axiomatisation for the first-order theory of the $j$-function is given, and then shown to be complete and to have quantifier elimination. An object similar to a pro-\'etale cover is then defined, and shown to be a model of the same first-order theory. We are then able to embed our discussions in this object. 
We show that categoricity can be seen as the statement that this pro-\'etale cover contains the same types of `independent' tuples as the standard model. Categoricity then becomes equivalent to certain statements about Galois representations in the geometric \'etale fundamental group. In particular, it is close to an instance of the adelic Mumford-Tate conjecture regarding the images of Galois representations in the Tate-modules of products of elliptic curves, and is stronger in the sense that it is over fields bigger than number fields, but slightly weaker in another. After proving this equivalence, we then verify that these algebro-geometric conditions hold, mostly by putting together results already proven by geometers in the literature, and then checking the remaining conditions.\newline

In Chapter \ref{chapcur}, we describe a general setting for the model-theoretic study of the universal cover of an arbitrary smooth complex algebraic curve. Guided by Grothendieck's fibre functor, an axiomatisation is given for the general theory of the universal cover of a smooth curve. In this setting, we show that the analytic universal cover and the pro-\'etale cover may be seen as models of the same theory, which is shown to be complete and to have quantifier elimination. We then prove a model-theoretic comparison theorem in this setting, showing that the analytic universal cover embeds in the pro-\'etale cover. At this stage, we then focus on the specific case of the multiplicative group $\C^{\times}$, giving a categorical axiomatisation for the theory of the universal cover of $\C^{\times}$ within this new framework. We then give a new proof of the categoricity result of \cite{zilber2006covers}.
\newline

Finally, a large appendix is included. Since this thesis is aimed at both algebraic geometers and model-theorists, to help the exposition of the results flow more smoothly, many results which might not be assumed common knowledge for some subset of the target audience are included in this appendix. I expect that in terms of the background knowledge required to understand the work of this thesis, the algebraic geometer who is comfortable with the basic notions of model theory is in a better position than the pure model-theorist who knows the basics of algebraic geometry. For this reason, there is a large section of the appendix dedicated to the theory of algebraic curves, which will be nothing new for the geometers, but might be useful for the model-theorist. I would suggest that the best way of reading this thesis would be to refer to the appendix when required, but possibly with the exception of \S \ref{seccov} regarding the general theory of covering spaces. This general theory is at the heart of the thesis, and it will be difficult to get a good feel of what is going on without a basic knowledge of covering spaces.

\chapter{The $j$-function}\label{chapj}

In this chapter, a natural two-sorted structure for the model theoretic study of the modular $j$-function is defined, and a natural axiomatisation is given for the complete first-order theory of the $j$-function in this setting. This first-order theory is then augmented with a non first-order axiom, which fixes the fibres of $j$ to be orbits under the modular group $\psl2(\Z)$, and it is shown that the openness of Galois representations in the automorphism group of a pro-\'etale cover is equivalent to categoricity of this augmented theory. Finally, it is then shown that the non-elementary theory of the $j$-function has a unique model in each infinite cardinality, by demonstrating that certain Galois representations in the Tate modules of products of elliptic curves are open.

\section{Preliminaries}\label{sec prelim}

Consider the upper half plane $\H$ with the group $\gl2q$ acting on it via
$$\begin{pmatrix} a&b\\ c&d \end{pmatrix} \tau = \frac{a \tau + b}{c \tau + d},$$
along with the modular $j$-function mapping $\H$ onto the complex numbers $\C$. 
Some possible intuition behind this set-up is that $\H$ is nearly the universal cover of $\C$, which is complex analytically isomorphic to $\sl2(\Z) \backslash \H$ via $j$. A point in the upper half plane gives rise to a lattice, and taking the quotient of $\C$ by this lattice gives an elliptic curve. Two points in the upper half plane give rise to elliptic curves which are isomorphic over $\C$ if and only if they are conjugated by an element of $\sl2(\Z)$, so the $j$-function realises $\C$ as the moduli space of elliptic curves. The action of $\gl2q$ as a group of analytic automorphisms on the cover $\H$ gives information about isogenies between elliptic curves, however elliptic curves will not need to be mentioned again until \S \ref{secgalrep}. It is exactly the scalar matrices of $\gl2q$ which act trivially on the whole of $\H$, so to ensure we have a faithful action we will consider the group
$$G:=\gl2q^{ad}=\gl2q / Z(\gl2q)$$
i.e. $\gl2q$ modulo its centre.
 \newline

Let $\cL$ be a first-order language for two-sorted structures of the form
$$\cM =  \langle  H ; \{f_g \}_{g \in G}  \rangle \overset{j}{\lora} \langle F,+,\cdot ; C \rangle,$$ where the structure $$\cH:= \langle H ; \{f_g \}_{g \in G}\rangle$$ is a set $H$ with a unary function symbol $f_g:H \ra H$ for each $g \in G$. These unary function symbols will correspond a left action of $G$ on $H$, and we will usually just write $g$ instead of $f_g$, and $G$ instead of $ \{f_g \}_{g \in G}$. $$\cF := \langle F;+,\cdot; C \rangle$$ is an algebraically closed field $F$ of characteristic zero expanded with a set of constant symbols $C$, and the function $j$ goes from $H$ to $F$. We will refer to $H$ as the \textit{`covering sort'} and $F$ as the \textit{`field sort'}. \newline

Let $\Th(j)$ be the complete first-order theory of  standard $j$-function in the language described above i.e. the first-order theory of the \textit{`standard model'} $$\C_j:=  \langle \H, G \rangle \overset{j}{\lora} \langle \C , + , \cdot ; \Q(j(S)) \rangle ,$$
where $S$ is the set of special points (see \ref{sec not}).
Also define the $\cL_{\omega_1, \omega}$ axiom
$$\textrm{StandardFibres}: \ \ \ \forall x \forall y \ \left(  j(x)=j(y)\ra\bigvee_{\gamma \in \sl2(\Z)}x=\gamma(y) \right)$$
which fixes a fibre of $j$ to be an $\sl2(\Z)$-orbit.
We abbreviate $\textrm{StandardFibres}$ by $\textrm{SF}$. \newline

The main theme of this chapter is to consider a `pro-\'etale cover' as a model of $\Th(j)$, and to embed the model theoretic discussion in this algebro-geometric object. We will see that the categoricity of $\Th(j)\wedge SF$ may be interpreted as the statement that in this setting, the analytic universal cover realises the same `independent' types as the pro-\'etale cover (the notion of independence is defined in \ref{defgindep}).

\subsection{Notation}\label{sec not}
\begin{itemize}

\item For a subset $G'$ of $G$ we will write $j(G' \tau)$ for $\{j(g \tau) \ | \ g \in G' \}$. For $\tau \in \H$, $j(G \tau)$ is called a \textit{`Hecke orbit'};
\item For a tuple $x = (x_1,...,x_n)$ and a function $f$, we define $f(x):=(f(x_1),...,f(x_n))$.
\item If $s \in H$ is (the unique element) fixed by some $g \in G$, then $s$ is called {\it`special'}. The set of all special points in $H$ is denoted $S$. In this situation, we also say that $j(s)$ is special;
\item Let $\Gamma: = \psl2(\Z)$, the image of $\sl2(\Z)$ in $G$. $\Gamma$ is a non-normal subgroup of $G$.
\item For a field $K$, let $G_K$ be its absolute Galois group.
\item \textit{Definable} means definable with parameters.
\item $\tp(x)$ is the complete type of $x$ and $\qftp(x)$ the quantifier-free type.
\end{itemize}

\section{Algebro-geometric background}

In this section we recall some of the background results in arithmetic geometry that are needed to set the scene.

\subsection{The $j$-function}\label{sec j}

\begin{pro}\cite[Chapter 5, \S2]{milne2006elliptic}
There exists a unique meromorphic function $J$ on $\H \cup \infty$, invariant under the modular group $\Gamma:=\psl2(\Z)$ which is holomorphic everywhere except for one simple pole at $\infty$, and takes the values
$$J(i)=1, \ \ \  J(e^{2\pi i / 3})=0.$$
\end{pro}
\begin{dfn}
Define the modular $j$-function (or $j$-invariant) as
$$j := 1278J.$$
\end{dfn}
So the $j$-function gives a complex analytic isomorphism from the Riemann surface $\Gamma \backslash \H^*$ (see below) onto the Riemann sphere. The two points $i, e^{2\pi i / 3} \in \H$ above, have non-trivial (finite) stabiliser in $\Gamma$ (and it is only points in the orbits $\Gamma i$ and $\Gamma e^{2\pi i / 3}$ which have non-trivial stabiliser). As a result, the projection map
$$p: \H \ra \Gamma \backslash \H$$
is a branched covering map, with the branching occurring exactly at the points $\Gamma i$ and $\Gamma e^{2\pi i / 3}$. 
 So $j$ is actually a true covering map of $\C \backslash \{0, 1728 \}$, and this will be at the heart of the discussion to follow. In fact, in the setting of this chapter, the theory of the special points is so rigid that we will not be interested in the fibres of special points under certain branched covering maps, and in particular we will not be interested in the branch locus $\{0, 1728 \}$ (which are indeed (very) special points), or its preimage under these maps.

\subsection{Quotients of the upper half plane}\label{sec quh}

In this subsection we recall some of the classical theory of the correspondence between Riemann surfaces and algebraic curves, in the setting of modular curves. See \ref{appmod} for a slightly more more detailed discussion around quotients of the upper half plane, and the algebraicity of such quotients.

\begin{dfn}
Let $\Gamma'$ be a discrete subgroup of $\psl2(\R)$. Then $c \in \P^1(\R)$ is called a cusp of $\Gamma'$ if there is an element $\alpha$ of $\psl2(\R)$ fixing $c$. Define the extended upper half plane to be
$$\H^*:= \H \cup \P^1(\Q).$$
\end{dfn}

The quotient $\Gamma \backslash \H^*$ is a compact, Hausdorff space  \cite[\S 1.4]{shimura1971introduction}, and given any discrete subgroup $\Gamma'$ of $\psl2(\R)$, the quotient $\Gamma' \backslash \H^*$ can be given the structure of a Riemann surface. Therefore if $\Gamma'$ is of finite index in $\Gamma$, the quotient $\Gamma' \backslash \H^*$ is a compact Riemann surface (since the natural map induced by the inclusion of groups is a finite branched covering map, so is proper), and is therefore algebraic by the Riemann existence theorem (\ref{thmrie}). Also, since $\Gamma' \backslash \H^*$ is compact there are finitely many $\Gamma'$-orbits of the cusps.

An arbitrary tuple $g=(g_1,...,g_n) \in G^n$ determines a subgroup of $G$
$$\Gamma_{g}  := g_1^{-1} \Gamma g_1 \cap \cdots \cap g_n^{-1} \Gamma g_n.$$

The quotient $\Gamma_g \backslash \H$ is a complex algebraic curve by the discussion in \cite[\S 1.4 and 1.5]{shimura1971introduction} (or see \ref{appmod}).


\begin{dfn}
With $g$ as above, define ${}_{\C}Z_g$ to be the complex algebraic curve biholomorphic to $\Gamma_g \backslash \H$.
\end{dfn}

Given two discrete subgroups $\Gamma_g$ and $\Gamma_h$ of $\psl2(\R)$ with $\Gamma_g$ a subgroup of finite index in $\Gamma_h$, the natural map
$$p: \Gamma_h \backslash \H^* \ra \Gamma_g \backslash \H^*$$
induced by the inclusion of groups is holomorphic (\cite[\S 1.36]{shimura1971introduction}), so by the basic theory of Riemann surfaces is a branched covering map. The map also induces an algebraic map on the corresponding algebraic curves by the  Riemann existence theorem. \newline

Given $\alpha \in G$, the double coset has the form
$$\Gamma \alpha \Gamma = \sqcup_{i} \Gamma g_i$$
where the $g_i$ form a finite set of coset representatives \cite[5.29]{milnemodular}. For $N \in \N$, we will abbreviate the double coset 
$$\Gamma  \begin{pmatrix} N&0 \\ 0&1 \end{pmatrix} \Gamma$$
by $\Gamma N \Gamma$, and let
$$\Gamma_N:=  \Gamma \cap g_1^{-1} \Gamma g_1 \cdots \cap g_{\psi(N)}^{-1} \Gamma g_{\psi(N)}$$
where the $g_i$ form a set of coset representatives for $\Gamma$ in $\Gamma N \Gamma$ (the number of cosets is denoted $\psi(N)$). We also note that an explicit set of coset representatives for $\Gamma N \Gamma$ is
$$\left\{ \begin{pmatrix} a&b\\ 0&d \end{pmatrix}  \ \ | \ \ 0< a,\ 0 \leq b < d, \ ad=N \right\}$$
(see \cite[5, \S1]{lang1987elliptic}), and an easy computation then gives the following:

\begin{pro}\label{cong}
$$\Gamma_N = \left\{ \begin{pmatrix} a&b\\ c&d \end{pmatrix} \in \Gamma \ \ , \ \ b \equiv c \equiv 0 , \ a \equiv d \mod N \right\}.$$
\end{pro}

\begin{dfn}
Consider the quotient
$$\Gamma_N \backslash \H.$$
Then we denote the corresponding complex affine algebraic curve by ${}_{\C}Z_N$.
\end{dfn}

\begin{pro}\label{etcov}
Given an arbitrary tuple $g=(g_1,...,g_n) \in G^n$,  there is $N \in \N$ and finite, surjective (algebraic) morphism
$$p: {}_{\C}Z_N \ra {}_{\C}Z_g,$$
which is \'etale outside a finite branch locus (i.e. the points of ${}_{\C}Z_g$ lying above $\{0 ,1728 \}$ under the natural map $f: {}_{\C}Z_g \ra \A^1(\C) \cong \Gamma \backslash \H$).
\end{pro}

\begin{proof}
Given the tuple $g=(g_1,...,g_n)$, consider $N_i \in \N$ such that 
$$\Gamma g_i \Gamma = \Gamma N_i \Gamma = \sqcup_j \Gamma g_{i_j},$$
and define
$$g' = (g_{1,1},...,g_{1, \psi(N_1)},...,g_{n,1},...,g_{n, \psi(N_n)}).$$
If we let $N_g$ be the lowest common multiple of the $N_i$, then viewing $\Gamma_{N_g}$ as a congruence subgroup as in \ref{cong}, we see that
$$\Gamma_{N_g} = \bigcap_{N_i | N} \Gamma_{N_i} \leq \Gamma_{g'} \leq \Gamma_g$$
and therefore the inclusion of groups gives us a holomorphic covering map
$$\Gamma_{N_g} \backslash \H \ra \Gamma_g \backslash \H$$
which induces regular map on the corresponding algebraic curves by the Riemann existence theorem.
\end{proof}


\begin{dfn}
Let $\Fin$, be the category of \textit{`branched finite algebraic covers'} of $\A^1(\C)$ i.e. we consider finite, surjective algebraic morphisms preserving the covering maps between varieties.
\end{dfn}

Given a branched Galois cover
$$p: {}_{\C}Z_N \ra {}_{\C}Z_1,$$
and $x \in {}_{\C}Z_1 \cong \A^1(\C)$, there is a right action of $$\Aut_{\Fin}({}_{\C}Z_N / {}_{\C}Z_1) \cong  \Gamma  / \Gamma_N$$ on the fibre $p^{-1}(x)$. Since the cover is Galois, the action is transitive, and if $x$ is outside of the branch locus then the action is free and the fibre $p^{-1}(x)$ is a $  \Gamma / \Gamma_N$-torsor. 

\subsection{Linking geometry and model theory}

In the above discussion we used the fact that compact Riemann surfaces are algebraic to show that certain quotients of the upper half plane were algebraic. To see that an arbitrary compact Riemann surface is algebraic usually requires appealing to Riemann-Roch or something similar, but for the Riemann surfaces we are interested in here the situation is much simpler. The reason is that once you know that the $j$-function gives a complex analytic isomorphism between $ \Gamma  \backslash \H^*$ and $\P^1(\C)$, then we may use the general theory of covering spaces (as in Appendix \ref{app etale}) to explicitly see that the covering maps induced by inclusions of groups may be used to realise these quotients as a definable set in the field sort. This is the key fact linking the algebraic geometry with the model theory, and will eventually mean that types correspond to certain Galois representations. The general algebraicity results of subsection \ref{sec quh} are not actually essential for any of the proofs to follow, but were included to give the reader some background understanding of what is going on geometrically. It is a subtle, but important point, that we need to use the definablilty results of the current subsection, rather than the algebraicity results of the previous subsection to make model theoretic progress. \newline


Fix a basepoint $s_0 \in \H$ outside of the ramification locus of $j$.

\begin{lem}\label{ZN def}
Fix an enumeration $(g_1,...,g_{\psi(N)})$ of a set of coset representatives for $\Gamma$ in $\Gamma N \Gamma$ and consider the map
$$p_N : \tau \mapsto (j(\tau) , j(g_1 \tau),...,j(g_{\psi(N)} \tau )) \subseteq \C^{\psi(N)+1}.$$
Then $p_N$ is biholomorphic onto its image, which is definable in $\langle \C ; + , \cdot ; 0 , 1 \rangle$. 
\end{lem}
\begin{proof}
Since $p_N$ is invariant under $\Gamma_N$, we can view $p_N$ as a holomorphic function from $\Gamma_N \backslash \H$ to $\C^{\psi(N)+1}$. $p_N$ is injective since 
$$j(x)=j(y) \textrm{ iff } y \in \Gamma x,$$ 
so that $$p_N(x)=p_N(y) \textrm{ iff } y \in \Gamma_N x.$$

The group $\Gamma / \Gamma_N$ acts on the image $p_N(\H)$ via
$$(j,j g_1,...,j  g_{\psi(N)}) \mapsto (j  \gamma,j  g_1  \gamma ,...,j  g_{\psi(N)}  \gamma),$$
and each $\gamma \in \Gamma / \Gamma_N$ induces an automorphism of the complex analytic cover $$ p_N(\H) \ra \C,$$
compatible with the canonical projection
$$(j,j g_1,...,j  g_{\psi(N)}) \mapsto j.$$
Making the variable substitution
$$j \mapsto X , \ j \circ g_{i} \mapsto Y_i,$$
$\gamma$ induces a corresponding permutation of the variables giving a map
\begin{align*}
\gamma: \C^{\psi(N)+1} & \ra \C^{\psi(N)+1}\\
(X,Y_1,...,Y_{\psi(N)}) & \mapsto (X, \gamma(Y_1),...,\gamma(Y_{\psi(N)})).
\end{align*}
Denote the finite group of maps arising from $\gamma \in \Gamma / \Gamma_N$ in this manner as $\Aut_{\Fin}(Z_N / Z_1)$. Clearly each element of $\Aut_{\Fin}(Z_N / Z_1)$ is definable in the field sort. By the theory of covering spaces (see Appendix \ref{app etale} for example), the image $p_N(\H)$ is exactly the subset of $\C^{\psi(N)+1}$ containing $p_N(s_0)$ such that fibres of the projection onto the $X$-coordinate are $\Aut_{\Fin}(Z_N / Z_1)$-torsors outside of the branch locus $\{0,1728 \}$, and the projection onto the $X$-coordinate is equal to the canonical projection $ p_N(\H) \ra \C$ on the branch locus. This set is definable in $\langle \C; + \cdot ; 0 , 1 \rangle$. 
\end{proof}

\begin{dfn}
Denote this definable subset of $\langle \C; + \cdot ; 0 , 1 \rangle$ by $Z_N$. The projection
$$(j(\tau) , j(g_1 \tau),..., j(g_{\psi(N)} \tau))  \mapsto  j(\tau)$$
induces a covering map
$$q_N : Z_N \ra Z_1$$
which is also definable in the field sort. The group of covering automorphisms $\Aut_{\Fin}(Z_N / Z_1)$ has already been defined in the above proof. Each element of $\Aut_{\Fin}(Z_N / Z_1)$ is definable in the field sort.
\end{dfn}

\begin{thm}\label{lem_locus_irreducible}
Let $g \in G^n$. Consider the map
$$p_{g} : \tau \mapsto ( j(g_1 \tau),...,j(g_{n} \tau )) \subseteq \C^{n}.$$
Then $p_{g}$ is biholomorphic onto its image, which is definable in $\langle \C ; + , \cdot ; 0 , 1 \rangle$. 
\end{thm}

\begin{proof}
Similarly to the proof of \ref{ZN def}, we may consider $p_g$ as an injection
$$p_g: \Gamma_g \backslash \H \hookrightarrow \C^n,$$
biholomorphic onto its image. Define the tuple
$$\bar{g}:=(1,g_1,...,g_n).$$
Then $\Gamma_{\bar{g}}$ is a subgroup of $\Gamma$, and also of $\Gamma_g$. We have
$$p_{\bar{g}} : \Gamma_{\bar{g}} \backslash \H \ra  \C^{n+1}$$
and $\Gamma$ acts on the image $p_{\bar{g}}(\H)$ via
$$\gamma: (j,j  g_1 ,..., j g_n) \mapsto (j  \gamma ,jg_1  \gamma , jg_n \gamma)$$
with $\Gamma_{\bar{g}}$ acting trivially. 
 As in the proof of \ref{ZN def}, each element of the finite set $\Gamma / \Gamma_{\bar{g}}$ induces a permutation of variables
$$\gamma: (X,Y_1,...,Y_n) \mapsto (X,\gamma (Y_1),..., \gamma( Y_n))$$
and we define $Z_{\bar{g}}$ to be the (unique) subset of $\C^{n+1}$  such that $Z_{\bar{g}}$ contains $p_{\bar{g}}(s_0)$, each of the projections 
$$\gamma q_{\bar{g}}: (X,\gamma(Y_1),..., \gamma(Y_n)) \ra X$$
make $Z_{\bar{g}}$ into an \'etale cover of $Z_1$ outside of the branch locus $\{0,1728 \}$, and $q_{\bar{g}}$ is equal to the canonical analytic projection on the branch locus. The image of the projection
\begin{align*}
Z_{\bar{g}} & \ra \C^n\\
(X,Y_1,...,Y_n) & \mapsto (Y_1,...,Y_n),
\end{align*}
is $p_g(\H)$, and is definable in $\langle \C ; + \cdot , 0, 1 \rangle$. 
\end{proof}

\begin{dfn}\label{def covmap}
Denote this definable subset of $\langle \C; +, \cdot ; 0 , 1 \rangle$ by $Z_g$. We see from the above proof that $Z_g$ comes with a canonical definable projection map
\begin{align*}
q_g: Z_g & \ra Z_1 \\
(Y_1,...,Y_N) & \mapsto X
\end{align*}
induced by the definable projection
$$q_{\bar{g}} : Z_{\bar{g} } \ra Z_1.$$

We now describe a inverse system of maps between the $Z_g$, with each map being definable in the field sort. As above, we may find $N_g$ such that there is an inclusion of groups
$$\Gamma_{N_g} \hookrightarrow \Gamma_g.$$ Define a map
$$q_{N_g,g}: Z_{N_g} \ra Z_g$$
to be the unique map taking $p_{N_g}(s_0)$ to $p_g(s_0)$, \'etale outside of the points above the branch locus $\{0,1728 \} \subset Z_1$, and agreeing with the corresponding analytic covering map on the branch locus, such that the diagram
$$
\xymatrix{
Z_{N_g} \ar@{->}[dr]^{q_{N_g,g}} \ar@{->}[dd]_{q_{N_g}} &  \\
&  Z_g \ar@{->}[dl]^{q_g} \\
 Z_1 &  
}
$$
commutes. So every $Z_g$ is (definably) covered by the corresponding $Z_{N_g}$. This will be important later on in \S \ref{secgalrep} because $\Gamma_{N_g}$ is a normal, congruence subgroup of $\Gamma$, and this endows $Z_{N_g}$ with easily accessible arithmetic and geometric information.
\newline

 Suppose that $M$ divides $N$. Define a map
$$q_{N,M}: Z_{N} \ra Z_M$$
to be the unique map taking $p_{N}(s_0)$ to $p_M(s_0)$, \'etale outside of the points above the branch locus $\{0,1728 \} \subset Z_1$ such that the diagram
$$
\xymatrix{
Z_{N} \ar@{->}[dr]^{q_{N,M}} \ar@{->}[dd]_{q_{N}} &  \\
&  Z_M \ar@{->}[dl]^{q_M} \\
 Z_1 &  
}
$$
commutes. This is definable in $\langle \C ; + , \cdot ; 0,1 \rangle$.

\end{dfn}



Since the covers $q_N:Z_N \ra Z_1$ are Galois, if $M$ divides $N$, then the covering map
\begin{align*}
q_{N,M}: Z_N & \ra Z_M \\
p_N(s_0) & \mapsto p_M(s_0)
\end{align*}
induces a map
$$ \Aut_{\Fin}(Z_N / Z_1) \ra \Aut_{\Fin}(Z_M / Z_1)$$
where $f \in \Aut_{\Fin}(Z_N / Z_1)$ maps to the unique element of $\Aut_{\Fin}(Z_M / Z_1)$ sending $p_M(s_0)$ to $q_{N,M}(f(p_N(s_0)))$. 

\subsection{Special points}\label{sec special}

Suppose that $s \in \H$ is \textit{special} i.e. there exists $g_s \in G$ such that $g_ss=s$. Then everything in the $G$-orbit of $s$ is also special i.e. $gg_sg^{-1}$ fixes $gs$ for $g \in G$. This gives us the notion of a \textit{special orbit} of $G$ in $\H$. The following proposition describing the basic behaviour of the special points:

\begin{pro}\label{pro special}
Special points of $\H$ belong to imaginary quadratic fields. Given a special point $s$ in an imaginary quadratic field $K$ we have $Gs=K \cap \H$, and every special orbit is of this form. Two special points lie in the same $G$-orbit exactly when they lie in the same imaginary quadratic field. If a special point $s$ is fixed by $g_s \in G$, then $s$ is the unique fixpoint of $g_s$.
\end{pro}
\begin{proof}
Just note that for $g \in G$ and $x \in \H$, we have
$$g(x) = x \textrm{ iff } \begin{pmatrix} a&b\\ c&d \end{pmatrix} x = x \textrm{ iff } cx^2 + (d-a)x -b =0$$
for some $a,b,c,d \in \Q$. The rest is then easy to verify.
\end{proof}

\begin{dfn}
Let $g=(g_1,...,g_n) \in G^n$. A point $x$ of $Z_g$ is defined to be \textit{special} if all its coordinates are special i.e. if $x = p_g(s)$ for $s \in S$. If $q_g : Z_g \ra Z_1$ is a cover, then for $x \in Z_1$, we say that the fibre $q_g^{-1}(x)$ is a \textit{special fibre} of the covering map if all the points of the fibre are special. Note that by the above discussion, if one point in a fibre is special, then all points of the fibre are special.
\end{dfn}

\begin{lem}\label{lem inf}
If an irreducible quasi-affine curve $C \hookrightarrow \A^n(\C)$ contains infinitely many $K$-points, then the curve is defined over $K$.
\end{lem}
\begin{proof}
For $\sigma \in \Gal(\C / K)$, $C$ and $C^\sigma$ are strongly minimal definable subsets of $\C^n$. Since $C$ contains infinitely many $K$-points,  $C \cap C^\sigma$ is infinite and therefore $C=C^{\sigma}$.
\end{proof}

\begin{thm}\label{etdef}

$Z_g$ is an irreducible quasi-affine curve defined over $\Q(j(S))$. Given $f \in \Aut_{\Fin}(Z_N / Z_1)$, the graph of $f$ considered as an algebraic variety in $Z_N \times Z_N$ is defined over $\Q(j(S))$.
\end{thm}

\begin{proof}
$Z_g$ is constructible by quantifier elimination in $\langle \C ; + \cdot ; 0,1 \rangle$, and is biholomorphic with $\Gamma_g \backslash \H$, so is an irreducible quasi-affine curve. By \ref{pro special} $Z_g$ contains infinitely many special points so the first statement follows from \ref{lem inf}. The second statement also follows from \ref{lem inf}, since if $y \in Z_N$ is special then $f(y)$ is also special.
\end{proof}

\begin{rem}\label{rem embedzg}
The curve $Z_g$ is $\emptyset$-definable in the structure $\langle \C ; + ; \cdot ; \Q(j(S)) \rangle$. Algebro-geometrically speaking, we have described a model of ${}_{\C}Z_g$ over $\Q(j(S))$, with a fixed embedding into some complex affine space. Similarly for each $f \in \Aut_{\Fin}(Z_N / Z_1)$.
\end{rem}

\section{Description of the types}\label{sec descrip type}

If $\langle H,F \rangle \models \Th(j)$, then we let $\tp_{H}(\tau)$ stand for the type of $\tau$ in the covering sort only, and $\tp_F(z)$ be the type of $z$ in the field sort only. \newline

The algebraic curves $Z_g$ encode information about geometric interactions in the structure $\langle H, G \rangle$ (in particular the trivial pregeometry), and this means that the two-sorted type of a non-special tuple $\tau \subset H$ may be studied down in the field sort alone:

\begin{pro}\label{fielddet}
For a finite tuple $\tau \subset H - S$, $\qftp(\tau )$ is determined by
$$\bigcup_{ g \subset G} \qftp_{\cF}(p_g(\tau) / \dcl(\emptyset) \cap F)$$
where $g$ ranges over all finite tuples $g \subset G^n$.
\end{pro}
\begin{proof}
If $\tau \cap S=\emptyset$ then all that can be said about $\tau$ with quantifier free formulae in the $H$ sort is whether or not any of the coordinates of $\tau$ are are related by elements of $G$. If so, then this is expressible down in the field sort by the corresponding images under $j$ sitting on a particular $\emptyset$-definable algebraic curve i.e.
$$\exists g \in G  \ g \tau_i = \tau_j \ {\rm iff } \ \exists g \in G, \ \ (j(\tau_i) , j(\tau_j  )) \in Z_{\bar{g}}.$$
\end{proof}



\begin{pro}\label{nicesystem2}
For $\tau \in (H - S)^n$ and $L \subseteq F$, $\qftp(\tau )$ is determined by 
$$\bigcup_{N} \qftp_{\cF}(p_N ( \tau ) / \dcl(L) \cap F).$$
\end{pro}
\begin{proof}
Given $g=(g_1,...,g_n) \in G^n$, there is $N_g$ such that there is a definable covering map
$$q_{N_g,g}: Z_N \ra Z_g$$
(see \ref{def covmap}). By \ref{fielddet} the above generalises straightforwardly to give the result.
\end{proof}

\section{Axiomatisation}\label{ax}

\subsection{The covering sort}

Let $\Th \langle \H, G \rangle $ be the complete first order theory of $G$ acting on the upper half plane $\H$ in the language described in \S \ref{sec prelim}. Rather than writing down an axiomatisation, we just note that for any group $G$, the class of faithful $G$-sets is first order axiomatisable, and all that remains is to describe the stabilisers of points. All models of $\Th \langle \H , G \rangle$ are infinite, so the following proof provides a substitute for a more explicit set of axioms since categoricity implies completeness in this case:

\begin{pro}
The isomorphism class of a model of $\Th \langle \H, G \rangle $ is determined by the cardinality of the number of non-special orbits. In particular $\Th \langle \H , G \rangle $ has a unique model in each uncountable cardinality.
\end{pro}

\begin{proof}
We explicitly construct an isomorphism $\sigma$ between two models
$$\sigma : \langle H, G \rangle \lora \langle H' , G \rangle $$
with the same number of non-special orbits: \newline
For each $g \in G$ there is a sentence of $\Th \langle \H , G \rangle$ stating that either $g$ fixes a unique element of $H$ or has no fixpoints. If $g$ has a fixpoint, denote by $s_g$ its unique fixpoint and let $\sigma(s_g)$ be the unique point of $H'$ fixed by $g$. For each pair $(g_1, g_2) \in G^2$ there is a sentence of $\Th \langle \H , G \rangle $ stating that if $g_1$ fixes $s_{g_1}$, then $g_2 g_1 g_2^{-1}$ fixes $g_2 s_{g_1} $. So
$$\sigma(g_2 s_{g_1} ) = \sigma(s_{g_2 g_1 g_2^{-1}} ) =  g_2 \sigma (s_{g_1})$$
and the map $\sigma$ is a partial isomorphism between the substructures generated by the special points.  \newline
To extend $\sigma$ to the non-special orbits fix a bijection between the quotients $G \backslash H $ and $G \backslash H' $. For each non-special orbit $Gh$, choose a point $\sigma(h)$ in the corresponding orbit of $H'$ under this bijection. Now let $\sigma(g h) = g \sigma (h)$ for all $g \in G$. 
\end{proof}

It follows that $\Th \langle \H, G \rangle $ is strongly minimal i.e. definable sets are finite or cofinite, uniformly in parameters. The model-theoretic algebraic closure operator gives a trivial pregeometry and we have a good notion of independence in models of the theory. 

\begin{dfn}
A tuple $( \tau_1,..., \tau_r) \subset H$ is said to be $G$-\textit{independent} if it is independent with respect to this trivial pregeometry i.e. if for all $i$, $\tau_i$ is non-special and for all $i \neq j$ $\tau_i \notin G \tau_j$.  
\end{dfn}


\subsection{The two-sorted theory}

For every finite tuple $g=(g_1,...,g_n) \subseteq G$ let $\Psi_g$ be a sentence saying that the map
$$p_g : \tau \mapsto (j(g_1 \tau),...,j(g_n \tau))$$
maps onto the $F$-points of $Z_g$. Note that this is a first order statement by the discussion leading up to remark \ref{rem embedzg}. \newline


Now let $T$ be the following theory:
$$\Th \langle \H , G \rangle  \bigcup \Th \langle \C , + , \cdot ; \Q(j(S) \rangle \bigcup_{g \subset G} \Psi_{g} \bigcup \tp(s),$$
where by $\tp(s)$, we mean that we include the complete (two-sorted) type of a special point $s$ in the standard model $\C_j$.

\section{The theory of the special points}\label{sec thsp}

In the above axiomatisation includes $\tp(s)$, the complete type of a special point $s \in S$, which includes all formulas in the two-sorted language mentioning finitely many special points. This ultimately means that we work `modulo' the theory of the special points. More technically, we have the following:

\begin{pro}\label{pro isom empty}
Given two models $\cM$ and $\cM'$ of $T$ there is an isomorphism of the substructures generated by $\emptyset$.
\end{pro}
\begin{proof}
In the model $\cM$, the substructure generated by $\emptyset$ is the two sorted structure
$$\langle S , G \rangle \overset{j}{\ra} \langle \Q(j(S)) ; + , \cdot ; \Q(j(S))  \rangle $$
where in the field sort every element is named with a constant symbol. Therefore the theory $T$ gives an identical, full description of every tuple of elements in both models and the result is clear.
\end{proof}

 So we are studying the model theory of the non-special points, and their interaction with the special points as a whole.  We also note that special points are definable, so model theoretically we are always forced to work over the set $S \cup \Q(j(S)) \subseteq \dcl(\emptyset)$.  \newline

The field $\Q(j(S))$ is a large field, which arises naturally in the model theoretic setting we have chosen here. In the model theoretic setting of this chapter, each special point of $\H$ in first order definable (note that the set $S$ of special points is not first order definable as a whole) and this is how the field arises. Including the theory of the special points, in some sense makes the categoricity problem easier by automatically giving us a starting point of building an isomorphism between two models. However, doing this will eventually put stronger restraints on the Galois representations later on (\ref{condition1}), since we have to work over this big field rather than a number field. This field is clearly very interesting, but in terms of its arithmetic, the only property needed for the model theory here is that the field $\Q(j(S))$ is an abelian extension of the compositum of all imaginary quadratic fields (\ref{lem abelian}).

\section{Quantifier elimination and completeness}

 \begin{pro}\label{realise}
Given two models $\cM$ and $\cM'$ of $T$,  non-special $\tau \in H$, $L \subseteq F$ and any finite subset $( g_1,...,g_n ) \subset G$, we may find $\tau' \in H'$ such that
$$\qftp_{\cF}(j(g_1 \tau),...,j(g_n \tau) / L) = \qftp_{\cF}(j'(g_1 \tau'),...,j'(g_n \tau') / L) .$$
\end{pro}

\begin{proof}
$\qftp_{\cF}(j(g_1 \tau),...,j(g_n \tau) / L)$ is determined by the minimal algebraic variety over $L$ containing $(j(g_1 \tau),...,j(g_n \tau))$. This is a subvariety $V$ of $Z_g$, and since $\cM'$ is also a model of $T$, the map
$$p_g : \tau' \mapsto (j'(g_1\tau'),...,j'(g_n \tau'))$$
maps onto $Z_g$ and therefore also maps onto $V$.
\end{proof}

Now we use a standard method for showing quantifier elimination and completeness of a theory:
\begin{pro}\label{QE}
Let
$$\cM = \langle \cH, \cF , j \rangle \textrm{ and } \cM' = \langle \cH', \cF , j' \rangle$$
be $\aleph_0$-saturated models of $T$ and
$$\sigma : H\cup F \ra H' \cup F$$
a partial isomorphism with finitely generated domain $D$. Then given any $z \in H \cup F$, $\sigma$ extends to the substructure
generated by $D \cup \{z \}$.
\end{pro}

\begin{proof}By \ref{fielddet}, and since $j$ is surjective, we may assume that $z \in H$. We may also assume that $z$ is $G$-independent from $D$, or otherwise $z$ is already included in the domain of $\sigma$. So suppose that $z \in H-D$, let $A = D \cap F$, $B = D \cap H$, and consider the field
$$L:=\left( \dcl^{\cM}(\emptyset)\cap  F \right) (A,j(G B))\cong \left( \dcl^{\cM'}(\emptyset)\cap F \right) (\sigma(A),j'(G  \sigma(B))).$$
By \ref{fielddet}, $\qftp(z/ D))$ is determined by the union of all $\qftp_F( p_g(z) /L)$ over all finite tuples $g \subset G$, and by  \ref{realise}, every finite subset of this type is realisable in any model of $T$. Therefore, by compactness the type is consistent, and since $\cM'$ is $\aleph_0$-saturated it has a realisation $z' \in H'$. 
\end{proof}

\begin{cor}\label{cor stab}
$T$ is complete, has quantifier elimination and is superstable. 
\end{cor}
\begin{proof}
That $T$ has quantifier elimination follows from \ref{QE} by basic model theory (see for example Anand Pillay's online model theory notes), and completeness then follows by \ref{pro isom empty}. It is then easy to see that $T$ is stable in cardinalities continuum and above by the description of the quantifier-free types given in \ref{nicesystem2}.
\end{proof}

\begin{rem}
By quantifier elimination we now know that $\dcl(\emptyset)=S \cup \Q(j(S))$ in any model of $T$. This means that if we take any two models $\cM,  \cM' \models T$, then since $T$ includes first-order theory of the special points, we automatically get an isomorphism $\dcl_{\cM}(\emptyset)\cong \dcl_{\cM'}(\emptyset)$, and this is always the first step in showing categoricity.
\end{rem}




\section{The pro-\'etale cover}

Define a `pro-\'etale cover'
$$\hat{\C}:= \varprojlim_{g \subset G} Z_g,$$
with the morphisms of the inverse limit being the definable maps described in \ref{def covmap}. $\hat{\C}$ is a pro-definable set in the structure $ \langle \C , + , \cdot ; \Q(j(S) \rangle $. We denote the projection map by 
$$\hat{j} : \hat{\C} \ra \C.$$
Note that this cover is not universal with respect to all finite branched covers, but we have taken a subsystem of covers corresponding to finite subsets $g \subset G$. The system of $Z_N$'s is cofinal in $\hat{\C}$, as was shown in \ref{def covmap}. $\hat{\C}$ comes with a named `basepoint lift' $(p_g(s_0))_g$, which was used in  \ref{def covmap} to get the explicit descriptions of the images of the maps $p_g$ as definable sets in the field.

\subsubsection{The action of Galois on $\hat{\C}$}\label{secgalpro}
Let $x \in \A^1(\C) - \{ 0 ,1728\}$. Then given a finite, Galois, branched cover
$$q_N : Z_N \ra Z_1$$
there is a natural left action of $$\Aut_{\Fin}(Z_N / Z_1) \cong  \Gamma / \Gamma_N $$ on the fibre $q_N^{-1}(x)$ and as was mentioned earlier, the fibre $q_N^{-1}(x)$ is a left $  \Aut_{\Fin}(Z_N / Z_1) $-torsor. We define $$\pi_1':= \varprojlim_N \Aut_{\Fin}(Z_N / Z_1),$$
so that in the limit, the fibre $\hat{j}^{-1}(x)$ is a $\pi_1'$-torsor. If $M$ divides $N$ then the map between $\Aut_{\Fin}(Z_N , Z_1)$ and $\Aut_{\Fin}(Z_M , Z_1)$ is as described in \ref{def covmap}, and depends on the choice of basepoint $s_0 \in \H$. 
\newline

There is a left action of $\Aut \langle \C ; + . \cdot ; 0,1  \rangle \cong \Gal(\C / \Q)$ on $\C^n$ via
$$(x_1,...,x_n)^\sigma := (x_1^\sigma,...,x_n^\sigma).$$
If $K$ is a field containing $\Q(j(S), x )$ then this induces left actions of $G_K$ on the fibre $q_N^{-1}(x)$, and on $\Aut_{\Fin}(Z_N / Z_1)$. The action of $G_K$ on $\Aut_{\Fin}(Z_N / Z_1)$ is given by conjugation i.e.
$$(y,f(y))^\sigma \mapsto ( y^\sigma , f(y)^\sigma) = (y^\sigma , f(y^{\sigma \sigma^{-1}})^\sigma)$$
for $y \in Z_N$ and $f \in \Aut_{\Fin}(Z_N / Z_1)$.
These actions are compatible with the action of $\Aut_{\Fin}(Z_N / Z_1)$, i.e. for $y \in \hat{p}^{-1}(x)$ we have
$$( \gamma y)^\sigma =  \gamma^\sigma y^\sigma.$$
Since $\hat{j}^{-1}(x)$ is a $\pi_1'$-torsor, given $y \in \hat{j}^{-1}(x)$ and $\sigma \in G_K$, there is a unique $\gamma_{\sigma} \in \pi_1'$, such that $y^\sigma = \gamma_{\sigma} y$. An easy calculation reveals that the pair $(x,y)$ gives rise to a continuous cocycle
\begin{align*}
\rho_{(x,y)}: G_K & \ra \pi_1' \\
 \sigma & \mapsto \gamma_{\sigma}.
\end{align*}

However, by \ref{etdef} the action of $G_K$ is trivial on $\Aut_{\Fin}(Z_N / Z_1)$, so we actually have a homomorphism
$$\rho_{(x,y)}: G_K\lora \pi_1'.$$
Note that we could have obtained this homomorphism more swiftly by noting that the fibre is a torsor for $\pi_1'$ and that the actions commute, but it is good to note that we are in a nice, specific case of a more general construction. \newline

Another point $ \gamma y$ in the fibre $\hat{j}^{-1}(x)$ gives rise to a conjugate representation $$\rho_{(x,\gamma y)} = \gamma \rho_{(x,y)} \gamma^{-1}.$$
\begin{pro}\label{progal}
In the situation above, for any $y \in \hat{j}^{-1}(x)$, the orbits of $G_K$ on the fibre $\hat{j}^{-1}(x)$ are in bijection with the index of the image $\rho_{(x,y)}(G_K)$.
\end{pro}
\begin{proof}
Consider the bijection
\begin{align*}
f:  \pi_1' & \ra \hat{j}^{-1}(x) \\
 \gamma & \mapsto \gamma y.
\end{align*}
Then $f$ induces a well defined map on the quotients
$$  \pi_1' / \rho_{(x,y)}(G_K) \lora \hat{j}^{-1}(x) / G_K $$
since 
$$ \gamma_1 \rho_{(x,y)}(\sigma) = \gamma_2 \textrm{  iff  } (\gamma_1 y )^\sigma = \gamma_2 y .$$
This map is clearly a bijection.
\end{proof}

These representations will ultimately be used to count the number of orbits of $G_K$ on the fibre, so we will drop the $y$ and just write $\rho_x$.
\begin{rem}
For a tuple of elements $x \in \left(Z_1 - \{ 0 ,1728\} \right)^n$, the fibre $\hat{j}^{-1}(x)$ is a $\pi_1'^n$-torsor and clearly all of the above applies, giving a representation
$$\rho_x : G_{K} \lora \pi_1'^n.$$
\end{rem}

\subsection{The pro-\'etale cover as a model of $T$}\label{proet mod}
Now we would like to consider the pro-\'etale cover $\hat{\C}$
as model of $T$. \newline

Firstly, $\hat{\C}$ has an action of $G$ induced by the action of $G$ on $\H$. For $\alpha \in G$, this action is given as follows:
\begin{align*}
\alpha: \Gamma_g  \backslash \H & \ra \alpha \Gamma_g \alpha^{-1} \backslash \H \\
\Gamma_g x & \mapsto \alpha \Gamma_g \alpha^{-1} \alpha x.
\end{align*}
As you can see, the action is slightly unusual in the sense that it mixes components around. Now given an element $(x_{\Gamma_g})_{\Gamma_g} \in \hat{\C}$ we define 
\begin{align*}
\hat{j}: \hat{\C} & \ra \C \\
(x_{\Gamma_g})_{\Gamma_g}  & \mapsto j(x_{\Gamma})
\end{align*}
i.e. $\hat{j}$ takes the $j$ value of the $\Gamma$-entry. \newline

Even though $\hat{\C}$ has an action of $G$, if we want to view the pro-\'etale cover as a model of the first order theory $T$ we cannot just define the universe of the covering sort to be $\hat{\C}$ because special elements of $G$ can fix many elements of $\hat{\C}$. It would be interesting to look at a theory which relaxes the condition that special matrices have unique fixpoints, but sticking with the current theory we make the following definition:
\begin{dfn}
Define a set $\hat{U}$ to be the union of the special $G$-orbits of $\H$, and the non-special $G$-orbits of $\hat{\C}$. Now define the (model-theoretic) pro-\'etale cover to be the two-sorted structure
$$\hat{\cU}:= \langle \hat{U} , G \rangle \overset{\hat{j}}{\lora} \langle \C , + , \cdot ; \Q(j(S) \rangle $$
where $\hat{j}$ is defined to be $j$ on $\H$.
\end{dfn}

Now we need to check that $\hat{\cU} \models T$. For this we note that if 
$$\hat{x} = (\Gamma_g x_{\Gamma_g})_g \in \hat{\C} = \varprojlim_g Z_g$$
then for $\alpha \in G$
$$\hat{j}( \alpha \hat{x}) = j( \alpha (\alpha^{-1} \Gamma \alpha x_{\alpha^{-1} \Gamma \alpha })) = j(\Gamma \alpha x_{\alpha^{-1} \Gamma \alpha}) = j(\alpha x_{\alpha^{-1} \Gamma \alpha}).$$
So for $\alpha = (\alpha_1 ,...,\alpha_n) \in G^n$ we see that the map
$$\hat{x} \mapsto (\hat{j}( \alpha_1 \hat{x}),...,\hat{j}( \alpha_n \hat{x})) = (j(\alpha_1 x_{\alpha_1^{-1} \Gamma \alpha_1}),...,j(\alpha_n x_{\alpha_n^{-1} \Gamma \alpha_n}))$$
is onto the $\C$-points of the algebraic curve $Z_{\alpha}$, so $\hat{\cU}$ satisfies the axiom $\Psi_{\alpha}$, and therefore $\hat{\cU} \models T$. \newline

\begin{rem}
The pro-\'etale cover $\hat{\cU}$ may be seen as a model of $T$ which is $\omega$-saturated with respect to non-special types.
\end{rem}

We can embed the standard model $\C_j$ in the pro-\'etale cover $\hat{\cU}$ as follows: We embed $\H$ in $\hat{\cU}$ (as a set) via 
$$\i: x \mapsto (\Gamma_g x)_g$$
and we have the following commutative diagram
$$
\xymatrix{
\H \ar@{^{(}->}[r]^\i \ar@{->}[dr]^j & \hat{U} \ar@{->}[d]^{\hat{j}}  \\
 & \C 
}
$$
To see that the map is injective, note that if $\i(x) = \i(y)$, then $\Gamma_g x = \Gamma_g y$ for all $\Gamma_g$, and in particular $\Gamma_p x = \Gamma_p y$ for all primes $p$. Clearly nothing can be in $\Gamma_p$ for infinitely many $p$, and now it is clear since the stabilizer of any $x \in \H$ in $\Gamma$, is finite. If we let $\i$ act as the identity on the field sort, then it is easy to check that $\i$ is a model theoretic embedding of $\cL$-structures, and therefore is elementary by quantifier elimination.

\section{Necessary conditions}


\begin{pro}\label{type in etale realise SF}
Given an element $x \in \hat{U}$, there is a model of $T\wedge SF$ realising $\tp(x)$. 
\end{pro}
\begin{proof}
The idea is to use the embedding
$$\i: \H \hookrightarrow \hat{U},$$
and replace a $G$-orbit in $\H$ with $Gx$.
So let 
$$H':= Gx \cup \{ y \in \i(\H)  \ \ | \ \ \ \hat{j}(y) \notin \hat{j}(G x) \},$$
and define a new model
$$\cM :=  \langle \langle H' , G \rangle , \langle \C , + ,\cdot , \Q(j(S)) \rangle , \hat{j}|_{H'} \rangle.$$
Clearly $\cM \models SF$, and it is also easy to see that $\cM \models \Psi_g$ for all $g \subset G$ and thus $\cM \models T$.

\end{proof}
It is important in the above that $x$ is a singleton and not an arbitrary tuple. To extend the above to tuples, we need an independence notion:

\begin{dfn}\label{defgindep}
Define a pair of elements $(x,y) \in F^2$ to be \textit{strongly $G$-independent}, if $\Phi_N(x, y ) \neq 0$ for all $N$. For $(x,y) \in \hat{U}^2$ we say that $(x,y)$ is \textit{strongly $G$-independent} if $(\hat{j}(x),\hat{j}(y))$ is strongly $G$-independent. We say a tuple is strongly $G$-independent if the coordinates are pairwise strongly $G$-independent. Also define two sets $X$ and $Y$ to be \textit{relatively strongly $G$-independent}, if for all $(x,y) \in X \times Y$, $(x,y)$ is strongly $G$-independent. Note that in models of $T \wedge  SF$, being $G$-independent is the same as being strongly $G$-independent.
\end{dfn}

Proposition \ref{type in etale realise SF} easily generalises to strongly $G$-independent tuples:
\begin{pro}\label{type in etale realise SF2}
Given a strongly $G$-independent tuple $x \in \hat{U}^n$, and $L \subset \C$ relatively strongly $G$-independent to $\hat{j}(x)$, there is a model of $T\wedge SF$ realising $\tp(x / L)$. 
\end{pro}

\begin{rem}\label{rem1}  Note the stronger statement that if $x \in \hat{U}^n$ is a $G$-independent tuple, then we can realise $\tp(\tau)$ in a model of $T \wedge  SF$, is not true. To see this consider the type of $x = ( x_1, x_2 )$ such that $\hat{j}(x_1)=\hat{j}(x_2)$ but $x_1 \notin \Gamma x_2$. \newline
\end{rem}

\begin{pro}\label{number types}
Let $z \in \C$ be non-special. Then the number of distinct types of elements in the fibre $\hat{j}^{-1}(z)$ is either finite or $2^{\aleph_0}$.
\end{pro}

\begin{proof} 
For $(x_1,...,x_n) \in \C^n$ and $K$ a subfield of $\C$, we let $\Loc( (x_1,...,x_n) / K)$ be the minimal (irreducible) algebraic variety over $K$ containing $(x_1,...,x_n)$ (i.e. the `Weil locus' of $(x_1,...,x_n)$ over $K$). Denote the field $\Q(j(S))$ by $\cS$, let $x \in \hat{j}^{-1}(z)$, and suppose that 
$$\Loc( \hat{p}_N(x) / \cS) \neq \Loc(\hat{p}_N(x') ) / \cS)$$
where $$\hat{p}_N: x \mapsto (\hat{j}(x), \hat{j}(g_1 x),...,\hat{j}(g_{\psi(N)} x)) $$
under the enumeration $(g_1,...,g_{\psi(N)})$ of cosets of $\Gamma$ in $\Gamma N \Gamma$ from earlier. Imagine a tree where branches are types of $x' \in \hat{j}^{-1}(z)$.
 Now also suppose there is $x''$ such that $$\Loc( \hat{p}_N(x'') / \cS) = \Loc(\hat{p}_N(x) / \cS) 
\textrm{ but } \Loc( \hat{p}_{NM}(x'') / \cS) \neq \Loc(\hat{p}_{NM}(x) / \cS) .$$ I claim that there exists $x'''$ such that $$\Loc( \hat{p}_{N}(x''') / \cS) = \Loc(\hat{p}_{N}(x') / \cS)  \textrm{ and }\Loc( \hat{p}_{NM}(x''') / \cS) \neq \Loc(\hat{p}_{NM}(x') / \cS):$$ Since the fibre above $z$ in $Z_N$ is an $\Aut_{\Fin}(Z_N / Z_1)$-torsor,
 and by \ref{etdef}, elements of $\Aut_{\Fin}(Z_N / Z_1)$ are defined over $\Q(j(S))$, there is a $\emptyset$-definable function sending $\Loc( \hat{p}_{N}(x) / \cS)$ to  $\Loc( \hat{p}_{N}(x') / \cS)$, and $\Loc( \hat{p}_{N}(x'') / \cS)$ follows $\Loc( \hat{p}_{N}(x) / \cS)$ across onto the other branch of the tree, proving the claim. Now we are done since our tree of types branches homogeneously.
\end{proof}

The following theorem of Keisler allows us to relate categoricity to the number of types:
\begin{thm}[Keisler, \cite{keisler1970logic}]\label{keislerthm}
If an $\cL_{\omega_1 , \omega}$-sentence has less than the maximum number of models of cardinality $\aleph_1$ (e.g. is $\aleph_1$-categorical) then there are only countably many $\cL_{\omega_1 , \omega}$-types realisable over $\emptyset$.
\end{thm}

We saw in \ref{type in etale realise SF2} that strongly $G$-independent types of tuples in the pro-\'etale cover $\hat{\cU}$ may be realised in models of $T \wedge SF$. 
By \ref{number types} there are finitely many, or uncountably many such types, so by Keisler's theorem, if $T\wedge  SF$ is to be $\aleph_1$-categorical, then there must be finitely many strongly $G$-independent types of tuples in $\hat{\cU}$. \newline

The fibre $\hat{j}^{-1}( x) \subset \hat{\C}$ above the point $x$ is a $\pi_1'$-torsor and given a field $K$ containing $\Q(j(S), x)$ we have a homomorphism
$$\rho_x: G_K \lora \pi_1'.$$
By \ref{progal}, the cosets of the representation are in bijection with the orbits of $G_K$ on the fibre, and by homogeneity in algebraically closed fields of characteristic 0, if $G_K$ does not conjugate two elements in the fibre, then they correspond to two distinct types. This extends to tuples, in summary:
\begin{pro}\label{pro imgtyp}
Given a non-special tuple $x \in F^n$, the set of types of tuples $u \in \hat{U}$ such that $\hat{j}(u)=x$ is in bijection with the cosets of the image of the representation
$$\rho_x : G_K \lora \pi_1'^n.$$
\end{pro}

The following arithmetic condition, along with another more geometric condition \ref{condition2}, will turn out to be equivalent to categoricity. Furthermore, we will verify that these conditions do actually hold in \S \ref{secgalrep}.
\begin{cdn}\label{condition1}
Let $x \in \C^n$ be a strongly $G$-independent tuple, $L$ be a finitely generated extension of $\Q$, and let $$K := L(j(S) , x).$$ Then the image of the homomorphism
$$\rho_{x}: G_K \lora \pi_1'^n$$
is of finite index.
\end{cdn}

As a consequence of the preceding discussions, we have the following:

\begin{thm}\label{thm n1}
Suppose that $T\wedge SF$ is $\aleph_1$-categorical, then Condition \ref{condition1} holds.
\end{thm}

\begin{proof}
Assume $T\wedge SF$ is $\aleph_1$-categorical. Then by Keisler's Theorem (\ref{keislerthm}) there are only finitely many types of tuples over $\emptyset$ realisable in a model of $T \wedge SF$. However, in this setting, expanding the language by finitely many constant symbols of the field sort only increases the number of types by a finite amount, and we may apply Keisler's theorem to types of tuples over a finite set of constant symbols from the field sort. So consider $L / \Q$ a finitely generated field extension, and a strongly independent tuple $x \in \C^n$. By \ref{type in etale realise SF2}, for all $u$ in the fibre $\hat{j}^{-1}(x)$, the type of $u$ over $L(j(S),x)$ is realisable some model of $T+SF$, and so by categoricity all of these types are realisable in the unique model of cardinality $\aleph_1$. By \ref{number types}, the number of types of such tuples in the fibre is either finite or uncountable, so the number of types of such tuples must be finite. By \ref{pro imgtyp} these types are in bijection with the index of the image of the corresponding Galois representation, and this forces the image of the representation to have finite index.
\end{proof}

If an $\cL_{\omega_1,\omega}$-sentence also has the amalgamation property for countable models, then a stronger form of Keisler's theorem holds.

\begin{dfn}
A theory has the \textit{amalgamation property} if given three models of the theory $\cM_0,\cM_1,\cM_2$ and elementary embeddings $\cM_0 \preccurlyeq_{\pi_i} \cM_i$, $i\in \{ 1,2\}$, there is a model $\cM$ and elementary embeddings $\cM_i \preccurlyeq_{\phi_i} \cM$ such that $\phi_1 \circ \pi_1=\phi_2 \circ \pi_2$.
\end{dfn}

\begin{pro}
$T\wedge SF$ has the amalgamation property.
\end{pro}
\begin{proof}

First notice that by quantifier elimination, all embeddings are elementary. Suppose that we have three models of $T\wedge SF$ 
$$\cM_i:= \langle H_i , F_i , j_i \rangle,      \ \ \ i \in \{0,1,2 \} $$
and embeddings $\cM_0 \preccurlyeq \cM_i$. Then there are corresponding embeddings of algebraically closed fields $F_0 \preccurlyeq F_1,F_2$. Let $F$ be an algebraically closed field which is the free amalgam of $F_1$ and $F_2$ over $F_0$ i.e. we may think of $F$ as $\acl(F_1 F_2)$ with $F_I \subset F$ and $F_1 \cap F_2 = F_0$. 

Define
$$H_1+H_2:=H_0 \sqcup (H_1 \backslash H_0) \sqcup (H_2 \backslash H_0).$$

There is a natural function
$$ j' : H_1+H_2 \ra F_1 \cup F_2, \ \ \ j'|_{H_i}:=j_i.$$

Now think of $F$ as embedded in $\C$ in the standard model of $T \wedge SF$, and let $$B:=j^{-1}\left( F \backslash (F_1 \cup F_2) \right).$$ 
Now the amalgam $\cM$ is defined as $(H_1+H_2) \sqcup B$ as a covering sort above the $F$-sort, and let the new function $J$ be the union of $j'$ and $j$ restricted to $B$
i.e.
$$J:H \ra F, \ \ \ J:= j' \cup j|_B.$$
To show that we do indeed have an amalgamation, we note that clearly the amalgamation works well with respect to the field sort so we just consider what is happening up in the covering sort. Now the special points are in $H_0$ already, and $G(H_1) \cap G(H_2) = G(H_0)$ since the $F_i$ are algebraically closed and if there was any intersection of $G$-orbits outside of $H_0$, this would be expressed in a modular relation in the field sort.
\end{proof}

As noted in \cite[\S4]{zilber2003model}, we have:

\begin{thm}\label{str kei}
Suppose that an $\aleph_1$-categorical $\cL_{\omega_1,\omega}$-sentence has the amalgamation property for countable models. Then the set of complete $n$-types over a countable model, realisable in a model of the sentence is at most countable.
\end{thm}

As was mentioned previously, the following condition is a more geometric analogue of \ref{condition1}.

\begin{cdn}\label{condition2}
Let $L \subseteq \C$ be a countable algebraically closed field, $x \in \C^n$ a strongly $G$-independent tuple with $x \cap L=\emptyset$, and let $K := L(x)$. Then the image of the homomorphism
$$\rho_{x}: G_K \lora \pi_1'^n$$
is of finite index.
\end{cdn}

Since $T+SF$ has the amalgamation property, we may apply the strong form of Keisler's theorem \ref{str kei}, and by the same proof as in \ref{thm n1} we get:

\begin{thm}\label{thm n2}
Suppose that $T+SF$ is $\aleph_1$-categorical, then condition \ref{condition2} holds.
\end{thm}

\section{Sufficient conditions}\label{secsuf}

In this section we apply the theory of quasiminimal excellent classes to show that the conditions of the last section are also sufficient for categoricity. The main technical condition of quasiminimality is the following:





\begin{lem}[$\aleph_0$-homogeneity over $\dcl(\emptyset)$]\label{homo} Suppose that Condition \ref{condition1} holds. 
Let
$$\cM = \langle F,H,j  \rangle \textrm{ and } \cM' = \langle F,H',j'  \rangle$$
be models of $T\wedge SF$. Let $L \subset \C$ be a finitely generated extension of $\Q$, and let $h=(h_1,...,h_n ) \subset H$ and $h' \subset H'$ be such that $\qftp(h/ L)=\qftp(h'/L)$. Then for all non-special $\tau \in H$ there is $\tau' \in H'$ such that $\qftp(h\tau /L)=\qftp(h'\tau' /L)$.
\end{lem}
\begin{proof}
The idea is that by Condition \ref{condition1}, all of the information in $\qftp(h \tau /L)$ is contained in the field type of a finite subset of the Hecke orbit, and we can find $\tau' \in H'$ which contains this `finite' amount of information by \ref{realise} (so the theory is atomic up to the failure of atomicity for $ACF_0$). Note that in general $\tp(ab)$ contains the same information as $\tp(a)\cup \tp(b / a)$. \newline

We may assume that the $h_i$ and $\tau$ are $G$-independent. Let
$$K:= L(j(S),j(G h_1),...,j(G h_n) ,j(\tau) ).$$
By Condition \ref{condition1} there exists an $N_0$, and a cover
$$Z_{N_0} \ra Z_1$$
 such that every element of $ \pi_1'$ fixing $Z_{N_0}$ 
belongs to the image of $G_K$ under the representation
$$\rho_{j(\tau)}: G_K \ra \pi_1'.$$
By \ref{realise} we can find $\tau' \in H'$ such that $p_{N_0} (\tau) = p_{N_0}'( \tau')$ and the result follows from \ref{nicesystem2}.

\end{proof}

The same argument gives the following:
\begin{lem}[$\aleph_0$-homogeneity over countable models]\label{homo2}
Suppose that Condition \ref{condition2} holds. Let
$$\cM = \langle F,H,j  \rangle \textrm{ and } \cM' = \langle F,H',j'  \rangle$$
be models of $T\wedge SF$. Let $L \subset \C$ be a finitely generated extension of a countable algebraically closed field and let $h=(h_1,...,h_n )\subset H$ and $h' \subset H'$ be such that $\qftp(h/L)=\qftp(h'/L)$. Then for all non-special $\tau \in H$ there is $\tau' \in H'$ such that $\qftp(h\tau /L)=\qftp( h' \tau' /L)$.
\end{lem}

\begin{thm}\label{thmcat}
Suppose that Conditions \ref{condition1} and \ref{condition2} hold. Then the theory of the $j$-function $T\wedge SF \wedge \trdeg(F) \geq \aleph_0$ has a unique model (up to isomorphism) in each infinite cardinality. 
\end{thm}
\begin{proof}
We appeal to the theory of quasiminimal excellence as in \cite[$\S1$]{kirby2010quasiminimal}. We define a closure operator $\cl:=j^{-1}\circ \acl \circ j$ which (by properties of $\acl$) is clearly a pregeometry with the countable closure property such that closed sets are models of $T\wedge  SF$. Lemma \ref{homo} gives us Condition II.2 (of \cite[$\S1$]{kirby2010quasiminimal}), and all other conditions of quasiminimality (i.e. $0$, I.1, I.2, I.3 and II.1) are then easily seen to be satisfied. $\aleph_1$-categoricity of $T \wedge SF$ follows (by \cite[Corollary 2.2]{kirby2010quasiminimal} for example). \newline

By \ref{condition1} and \ref{condition2}, the standard model $\C_j$ with the pregeometry $\cl$ is a quasiminimal pregeometry structure (as in \cite[\S2]{bays2012quasiminimal}) and  $\cK(\C_{j})$ is a quasiminimal class, so by the main result of \cite{bays2012quasiminimal}  (Theorem 2.2), $\cK(\C_j)$ has a unique model (up to isomorphism) in each infinite cardinality, and in particular  $\cK(\C_{j})$ contains a unique structure $\cC_{j}$ of cardinality $\aleph_0$. Let $\cK$ be the class of models of $T\wedge  SF \wedge  \trdeg \geq \aleph_0$. It is clear that $\cK(\C_j) \subseteq \cK$ since $\C_j \in \cK$, and to prove the theorem we want to show that $\cK = \cK(\C_{j})$. \newline

By \ref{condition1} and \ref{condition2}, $\cK$ has a unique model $\cM$ of cardinality $\aleph_0$. Since $\cK$ is the class of models of an $\cL_{\omega_1,\omega}$-sentence, $\cK$ together with closed embeddings is an abstract elementary class with Lowenheim-Skolem number $\aleph_0$, so by downward Lowenheim-Skolem (in $\cK$), everything in $\cK$ is a direct limit (with elementary embeddings as morphisms) of copies of the unique model of cardinality $\aleph_0$. Finally, all the embeddings in $\cK$ are closed with respect to the pregeometry, so $\cK =\cK(\cM)=\cK(\cC_j)= \cK(\C_{j})$.
\end{proof}

Define the \textit{arithmetic standard model} to be
$$\bar{\Q}_j:= \langle \langle j^{-1}(\bar{\Q}) , G \rangle , \langle \bar{\Q} , + , \cdot , 0,1 \rangle , j|_{j^{-1}(\bar{\Q})} \rangle.$$
We also note the following which just follows from the back and forth argument in \cite[$\S1$]{kirby2010quasiminimal} (the argument is also given in \ref{app quasi}).
\begin{thm}
Suppose that Condition \ref{condition1} holds. Then $\bar{\Q}_j$ is the unique (countable) model of the theory $T\wedge SF\wedge \trdeg(F) =0$.
\end{thm}

\section{Equivalence results}\label{secequiv}

We may summarise everything so far with the following theorems, which demonstrate a tight interaction between model theory and artithmetic geometry:

\begin{thm}
The theory $T\wedge SF \wedge \trdeg(F) \geq \aleph_0$ is $\aleph_1$-categorical iff conditions \ref{condition1} and \ref{condition2} hold.
\end{thm}

\begin{thm}
Assuming the continuum hypothesis, the standard model $\C_j$ is the unique model of cardinality continuum of the theory $T\wedge SF\wedge \trdeg(F) \geq \aleph_0$ iff conditions \ref{condition1} and \ref{condition2} hold.
\end{thm}

If we do not assume CH, then (using the notion of quasiminimal pregeometry structure as in \cite{bays2012quasiminimal}) we may say the following:
\begin{thm}
The standard model $\C_j$ is a quasiminimal pregeometry structure iff conditions \ref{condition1} and \ref{condition2} hold.
\end{thm}

\section{Galois representations on Tate modules}\label{secgalrep}

In this section we show that conditions \ref{condition1} and \ref{condition2} hold, thus giving the categoricity result. So far, we have neglected the fact that $j$ realises $\A^1$ as the moduli space of elliptic curves (see for example \cite[8.4]{milnemodular}), and I have been careful not even to mention an elliptic curve in the previous sections. However, in this section we will use this extra arithmetic and geometric information, and this is where things become more interesting. \newline

Given an elliptic curve $E$ defined over a field $K$, there is a continuous Galois representation on the Tate module $T(E)$ (the inverse limit of the automorphisms of the $N$-torsion of $E$, see \ref{app tate}):
$$\rho : G_K \ra T(E) \cong \GL_2(\hat{\Z}).$$
 Taking determinants gives an exact sequence
$$\SL_2(\hat{\Z}) \hookrightarrow \GL_2(\hat{\Z}) \overset{\det}{\ra} \hat{\Z}^{\times},$$
and via the Weil pairing, we may identify $\hat{\Z}^{\times}$ with the `cyclotomic character' i.e. the action of $G_K$ on $K^{\textrm{cyc}}$ (the field obtained by adjoining the roots of unity to $K$). So we have a representation
$$\rho : G_{K^{\textrm{cyc}}} \ra \sl2(\hat{\Z}),$$
and for a product of elliptic curves over $K$ this gives a representation
$$\rho : G_{K^{\textrm{cyc}}} \ra \sl2(\hat{\Z}) \times \cdots \times \sl2(\hat{\Z}).$$

There is a tight relationship between the images of these representations and our conditions for categoricity \ref{condition1} and \ref{condition2}. By \ref{cong},  $\Gamma_N$ is a congruence subgroup of the modular group $\Gamma$, and therefore the curve $Z_N$ is the moduli space of elliptic curves, with some additional data regarding the $N$-torsion.

\subsection{The curve $Z_N$ as a moduli space}

To see the curve $Z_N$ as a moduli space, it is useful to first look at the modular curve $$Y(N):= \Gamma(N) \backslash \H $$ where 
$$\Gamma(N) = \left\{ \begin{pmatrix} a&b\\ c&d \end{pmatrix} \in \sl2(\Z)\ \ , \ \ b \equiv c \equiv 0 , \ a \equiv d \equiv 1 \mod N \right\}$$
 is the principal congruence subgroup of level $N$. Since $-1$ acts trivially on $\H$ we have 
$$\Aut_{\Fin}(Y(N) / Y(1)) \cong  \SL_2(\Z / N \Z) / \pm.$$
$Y(N)$ parametrises equivalence classes of triples $(E, b_1,b_2)$ where $E$ is an elliptic curve, $(b_1,b_2)$ is a basis for $E[N]$ with Weil pairing a  fixed primitive $N^{th}$ root of unity $\zeta_N$, and two triples $(E, b_1, b_2)$ and $(E' , b_1', b_2')$ are equivalent if there is an isomorphism $E \cong E'$ sending $(b_1 , b_2)$ to $(b_1', b_2')$. Note that $-1 \in \Aut(E)$ so $(E , b_1, b_2)$ is always equivalent to $(E, -b_1 , -b_2)$. 
All modular curves have models over $\Q^{\cyc}$, and the $\Q(\zeta_N)$-model $Y(N)_{\Q(\zeta_N)}$.
It is enlightening to note that if you want to remove the restriction on the Weil pairing, then you need to take a disjoint union of $\varphi(N)$ copies of $Y(N)$, or look at the disconnected Shimura curve
$$\SL_2(\Z) \backslash (\H \times \GL_2(\Z / N \Z))$$
which has model $Y(N)_{\Q}$ over $\Q$, parametrising elliptic curves with `full level $N$ structure'. A detailed discussion of  modular curves as moduli spaces is given in \cite[\S 8]{milnemodular} for example.
 \newline

Let $F$ be an algebraically closed field of characteristic zero. For $\tau \in H$ let $E_{\tau}$ be an elliptic curve over $\Q(j(\tau))$ and fix a bijection $\tilde{\alpha}$ between a set of coset representatives for $\Gamma$ in $\Gamma N \Gamma$ and the set of cyclic subgroups of order $N$ in $E_{\tau}[N](F)$ (recall the cardinality of these sets is denoted $\psi(N)$).  Given a pair of disjoint  (except for the identity) cyclic subgroups of order $N$ in $E[N](F)$, choosing a generator of each subgroup gives a basis of $E[N](F)$. 

\begin{dfn}
Let $\cE_N(F)$ be the set of triples $(E , C_1 , C_2)$, where $E$ is an elliptic curve over $F$ and $C_1 , C_2$ are disjoint (except for the identity) cyclic subgroups of $E[N](F)$ of order $N$ containing a basis with Weil pairing $\zeta_N$. Let $\sim$ be the equivalence relation where $(E, C_1 , C_2)$ is equivalent to $(E' ,  C_1' , C_2')$ if there is an isomorphism $E \cong E'$ over $F$, taking $(C_1 , C_2)$ to $(C_1' , C_2')$.
\end{dfn}

\begin{pro}\label{pro moduli}
The natural map
\begin{align*}
\alpha : \cE_N(F)   & \ra Z_N(F) \\
(E_{\tau} , C_1,...,C_{\psi(N)}) & \mapsto (j(\tau) , j(g_1 \tau) ,... j(g_{\psi(N)} \tau) )
\end{align*}
induced by $\tilde{\alpha}$ gives a bijection
$$ \cE_N(F) / \sim   \ra Z_N(F) .$$
\end{pro}

Before giving the proof we remind the reader of some groups: \newline
$\ggl2 (\Z / N \Z)$ is the set of $2 \times 2$ matrices with coefficients in $\Z / N \Z$ and invertible determinant. Define
\begin{align*}
\z2(\Z / N \Z) & := \left\{ \begin{pmatrix} e&0 \\ 0&e \end{pmatrix} \ \ | \ \ e \in (\Z / N \Z)^{\times}  \right\} \\
\sv2(\Z / N \Z) & :=  \left\{ \begin{pmatrix} e&0 \\ 0&e \end{pmatrix} \ \ | \ \ e \in (\Z / N \Z)^{\times}, e^2\equiv 1 \mod N \right\}\\
\psl2(\Z / N \Z) & := \sl2(\Z / N \Z) / \sv2(\Z / N \Z).
\end{align*}

Then we have the following diagram of exact sequences
$$
\xymatrix{
\sv2(\Z / N \Z) \ar@{^{(}->}[r] \ar@{^{(}->}[d] & \z2(\Z / N \Z) \ar@{->>}[r]^{\det} \ar@{^{(}->}[d]  & (\Z / N \Z)^{\times  2} \ar@{^{(}->}[d] \\
\sl2(\Z / N \Z) \ar@{^{(}->}[r] \ar@{->>}[d] & \ggl2 (\Z / N \Z) \ar@{->>}[r]^{\det} \ar@{->>}[d] & (\Z / N \Z)^{\times} \ar@{->>}[d] \\
\psl2(\Z / N \Z) \ar@{^{(}->}[r]  &       \pgl2(\Z / N \Z) \ar@{->>}[r] & (\Z / N \Z)^{\times}/(\Z / N \Z)^{\times  2}
}
$$
where $(\Z / N \Z)^{\times  2}$ is the multiplicative group of quadratic residues mod $N$.

\begin{proof}[Proof of \ref{pro moduli}]
The inclusion of groups
$$\Gamma(N) \hookrightarrow \Gamma_N$$
induces a covering map
$$p: Y(N)_{\Q(\zeta_N)}(F) \ra Z_N(F),$$
where $Y(N)_{\Q(\zeta_N)}(F)$ is considered as a definable set in $\langle F ; \cdot, + ; ,0,1 \rangle$. Every pair of disjoint cyclic subgroups of order $N$ in $E[N](F)$ contains a basis, and elements of $\Aut(E[N]) \cong \GL_2(\Z / N \Z)$ fixing such a pair of cyclic subgroups are exactly the scalars. We have
$$\Aut_{\Fin}(Z_N / Z_1) \cong \psl2(\Z / N \Z)$$
and the result follows.
\end{proof}

\begin{rem}
If there are non-trivial square roots of unity mod $N$ then $\sv2(\Z / N \Z)$ is strictly bigger than $\pm$ and the curve $Y(N)$ is a proper cover of $Z_N$.
\end{rem}

The curve $Z_N$ is defined over $\Q(j(S)) \cap \Q^{\cyc}$, and therefore carries an action of $G_K$ which induces an action of $G_K$ on $\cE_N(F) $ via $\alpha$. 
\begin{pro}\label{pro functorialmod}
Let $K$ be a field containing $\Q(j(S)) \cap \Q^{\cyc}$. Then the diagram
$$
\xymatrix{
\cE_N(F) \ar@{->}[r]^\sigma \ar@{->}[d]^{\alpha} & \cE_N(F) \ar@{->}[d]^{\alpha} \\
Z_N(F)  \ar@{->}[r]^{\sigma} & Z_N(F).
}
$$
commutes.
\end{pro}

\begin{proof}
If an elliptic curve $E$ is considered as the set defined by a Weierstrass polynomial in the structure $\langle F , \cdot , + , 0,1 \rangle$, then $j(E)$ is a rational function of the coefficients of this polynomial. So for a cyclic subgroup $C$ of $E$ and $\sigma \in G_{\Q}$ we have
$$j(E / C)^\sigma = j(E^\sigma / C^\sigma)$$
and the result follows.
\end{proof}

\subsection{Relating Galois representations in $\pi_1'$ with representations on Tate modules}

In the last subsection, we saw that $Z_N$ can be seen as the moduli space of elliptic curves with some additional data involving cyclic subgroups. The functoriality with respect to the action of Galois (\ref{pro functorialmod}) is the key to comparing the Galois representations in $\pi_1'$ arising from model theory, with Galois representations on Tate modules of elliptic curves. \ref{pro functorialmod} says that when Galois moves points on $Z_N$ it moves the corresponding cyclic subgroups in an identical manner, and vice versa. From \ref{pro functorialmod} we immediately get:

\begin{pro}
Let $(\tau_1,...,\tau_r) \in \H^r$ and  $$K := \Q^{\textrm{cyc}}(j(\tau_1),...,j(\tau_r)).$$ Let $E_i$ be an elliptic curve, defined over $\Q(j(E_i))$ such that $j(E_i) = j(\tau_i)$, and let $L$ be a field containing $K$. Then if the image of the homomorphism
$$\rho: G_{L} \lora \SL_2(\hat{\Z})^r$$
in the product of the Tate modules has finite index, then so does the representation
$$\rho': G_L \lora \pi_1'^r.$$
\end{pro}

Now we would like to go the other way, and see that the assumption of categoricity has arithmetic, and geometric consequences with respect to Galois representations on elliptic curves. Given a representation $\rho$ into $\sl2(\hat{\Z}) \cong \varprojlim_N \sl2(\Z / N \Z)$ corresponding to an elliptic curve as above, we get an induced representation $\bar{\rho}$ and a commutative diagram
$$
\xymatrix{
G_K \ar@{->}^{\bar{\rho}}[dr] \ar@{->}^{\rho}[r] & \sl2(\hat{\Z})  \ar@{->}[d] \\
& \psl2(\hat{\Z})  
}$$
via the projections $\sl2(\Z / N \Z) \ra \psl2(\Z / N \Z)$. 
By \ref{pro functorialmod}, \ref{thm n1} and \ref{thm n2} we have the following:
\begin{thm}
Suppose the $T\wedge SF$ is $\aleph_1$-categorical. Then the images of the induced representations $\bar{\rho}$ corresponding to the representations of Theorems \ref{amtapp} and \ref{mooo} are open.
\end{thm}

\subsection{Images of Galois representations}

Now we go about proving that the images of certain Galois representations in the product of Tate modules of elliptic curves are of finite index. Note that for a profinite group, open is equivalent to being closed and of finite index.

\begin{thm}\label{amt}
Let $A$ be an abelian variety defined over $K$, a finitely generated extension of $\Q$, such that $A$ is a product of $r$ non-isogenous elliptic curves, all without complex multiplication. Then the image of the Galois representation
$$\rho: G_{K^{\rm{cyc}}} \ra \SL_2(\hat{\Z})^r$$
is open.
\end{thm}

Before embarking on the proof we need a definition.
\begin{dfn}\label{csc }
As in \cite[3.4]{ribet1975}, we say that a profinite group satisfies the \textit{`commutator subgroup condition'} if for every open subgroup $U$, the closure of the commutator subgroup $[U:U]$ of $U$, is open in $U$.
\end{dfn}

\begin{proof} [Proof of \ref{amt}] For $K$ a number field and $r \leq 2$, this was done by Serre (\cite[\S 6]{serre1971prop}), and Ribet reduced the problem to the case $r=2$ by noting that $\SL_2(\hat{\Z}$) satisfies the commutator subgroup condition, see \cite[3.4]{ribet1975}. The case where we have a product $A=E_1 \times \cdots \times E_r$ of $r$ non-isogenous elliptic curves with transcendental $j$-invariants, follows from the number field case and a standard specialisation argument \cite[Remark 6.12]{pink2005combination}. 
\end{proof}

Now we want to increase the base field of the representations to the compositum with $\Q(j(S))$. This will be possible due to the non-abelianness of $\SL_2(\hat{\Z})$ (i.e. it satisfies the commutator subgroup condition), and the following:

\begin{lem} \label{lem abelian}
The extension
$$\Q^{\textrm{cyc}}(j(S)) / \Q^{\textrm{cyc}}$$
is abelian.
\end{lem}
\begin{proof}
Let $\tau \in \H$ lie in an imaginary quadratic field $K$ (all elements of $S$ lie in imaginary quadratic fields), and $E$ be an elliptic curve defined over $\Q (j(E))$ with complex multiplication by $K$. Then $\Q(j(G \tau))$ contains $K$, and is a subset of the field obtained by adjoining $j(\tau)$ and the $x$-coordinates of the torsion of $E$ to $\Q$ (which is an abelian extension of $K$ by the theory of complex multiplication - see for example \cite[p135]{silverman1994advanced}). The result now follows since $\Q^{\textrm{cyc}}$ contains all imaginary quadratic fields and the compositum of abelian extensions is abelian.
\end{proof}

\begin{cor}
Given an abelian variety $A \cong E_1 \times \cdots \times E_r$ defined over $K$, a finitely generated extension of $\Q$, where the $E_i$ are non-isogenous elliptic curves with $\End(E_i)\cong \Z$ we have
$$[\Q^{\textrm{cyc}}(j(S))\cap K(\tor(A)) : K^{\textrm{cyc}} ] \textrm{ is finite.}$$
\end{cor}
\begin{proof}
$(\SL_2 \times \SL_2)(\hat{\Z})$ satisfies the commutator subgroup condition, since the Lie algebra $\fsl \times \fsl$ is its own derived algebra. The result follows since the intersection $K^{\textrm{cyc}}(j(S))\cap K(\tor(A))$ has to be an abelian extension of $K^{\textrm{cyc}}$ by \ref{lem abelian}.
\end{proof}

The following is now immediate, and implies condition \ref{condition1}:

\begin{thm}\label{amtapp}
Let $A$ be an abelian variety as in \ref{amt}, and let $L:=K^{\textrm{cyc}}(j(S))$. Then the image of the Galois representation
$$\rho: G_{L} \ra \SL_2(\hat{\Z})^r$$
is open.
\end{thm}

Now we want to verify the more geometric condition \ref{condition2}, which is implied by the following:

\begin{thm}\label{mooo}
Let $L \subset \C$ be an algebraically closed field. Let $A$ be a product of $r$ non-isogenous elliptic curves $E_1\times \cdots \times E_r$ defined over $\Q(j(E_1),...,j(E_r))$, and let $K:=L(j(E_1),...,j(E_r))$. Then the image of the Galois representation on the Tate module of $A$ is open in $\SL_2(\hat{\Z})^r$.
\end{thm}

The following proof is joint work with Chris Daw, a student of Andrei Yafaev, who works on the geometry of Shimura varieties at UCL. Recently, we have been able to generalise the equivalence results of \S \ref{secequiv} to an arbitrary Shimura Curve, and the results of this section to an arbitrary modular curve.
\begin{proof} [Proof of \ref{mooo}]
We may assume that the transcendence degree of $K / L$ is one, since if the $j(E_i)$ are algebraically independent then the result follows immediately from the case $r=1$, which follows from the theory of modular function fields (see for example \cite[Chapter 6, \S3]{lang1987elliptic}) and is much easier than Serre's result for an elliptic curve over a number field. We make the following deductions:

\begin{enumerate}

\item Note that $K/\Q(j(E_1),...,j(E_r))$ is a normal extension, since its Galois group is isomorphic to that of the extension $L/L\cap\Q(j(E_1),...,j(E_r))$, which is an intermediate extension of $L/\Q$. Therefore, $G_K$ is a normal subgroup of $\Gal(\bar{K} /\Q^{\rm cyc}(j(E_1),...,j(E_r)))$.

\item Recall that the image of the projection to any of the individual factors of the homomorphism
\begin{align*}
G_K \rightarrow\SL_2(\hat{\Z})^r
\end{align*}
has finite index (see \cite[Chapter 6, \S3, Corollary 2]{lang1987elliptic}). It is closed since it is a compact subset of a Hausdorff topological space. Therefore, it is also open.

\item We claim that the projection to the product of any two factors of the reduced homomorphism
\begin{align*}
G_K \rightarrow\SL_2(\F_p)^r
\end{align*}
is surjective for almost all primes $p$. To see this note that the projection to any single factor is surjective for almost all primes (see 2). Therefore, let
\begin{align*}
G_K \rightarrow\SL_2(\F_p)^2
\end{align*}
be a projection as in the claim. If $p$ is outside a finite set of primes, Goursat's Lemma states that the image is either full or its image in
\begin{align*}
\SL_2(\F_p)/\ker(\pi_2)\times\SL_2(\F_p)/\ker(\pi_1)
\end{align*}
is the graph of an isomorphism
\begin{align*}
\SL_2(\F_p)/\ker(\pi_2)\cong\SL_2(\F_p)/\ker(\pi_1),
\end{align*}
where $\pi_1$ and $\pi_2$ denote the first and second projections respectively. However, if $p>3$ then $\SL_2(\F_p)/\{\pm\}$ is simple. In which case, $\ker(\pi_1)$ and $\ker(\pi_2)$ are equal to $\SL_2(\F_p)$, $\{\pm\}$ or $\{1\}$. However, the latter two possibilities can occur for only finitely many primes since for almost all primes the projection of the homomorphism
\begin{align*}
\Gal(\bar{K}/\Q^{\rm cyc}(j(E_1),...,j(E_r))\rightarrow\SL_2(\F_p)^r
\end{align*}
to the product of any two factors is surjective by \ref{amt}, and the fact that $G_K$ is a normal subgroup of $\Gal(\bar{K}/\Q^{\rm cyc}(j(E_1),...,j(E_r))$ (see 1).

\item For almost all primes $p$, the projection of the homomorphism
\begin{align*}
G_K \rightarrow\SL_2(\Z_p)^r,
\end{align*}
to the product of any two factors is surjective by \cite[ \S6, Lemma 10]{serre1971prop}, (see 3).

\item We claim that for all primes $p$, the projection of the homomorphism
\begin{align*}
G_K \rightarrow\SL_2(\Z_p)^r,
\end{align*}
to the product of any two factors has open image. Indeed, for any such projection, let $\mathfrak{a}_p$ denote the $\Q_p$-Lie algebra of the image of $G_K$ in $\SL_2(\Z_p)\times\SL_2(\Z_p)$. It is a $\Q_p$-Lie subalgebra of $\mathfrak{sl}_2\times\mathfrak{sl}_2$ that surjects on to each factor (see 2). Therefore, as in \cite[\S6, Lemma 7]{serre1971prop}, either $\mathfrak{a}_p=\mathfrak{sl}_2\times\mathfrak{sl}_2$ or it is the graph of a $\Q_p$-algebra isomorphism $\mathfrak{sl}_2\cong\mathfrak{sl}_2$. However, the latter is impossible by \ref{amt}, and the fact that $G_K$ is a normal subgroup of $\Gal(\bar{K} / \Q^{\rm cyc}(j(E_1),...,j(E_r)))$.   
\end{enumerate}

Therefore, by 4 and 5, the image of the projection, to the product of any two factors, of the homomorphism
\begin{align*}
G_K \rightarrow\SL_2(\hat{\Z})^r
\end{align*}
is open. Therefore, by \cite[Lemma 3.4 and Remark 3]{ribet1975} the image of this homomorphism itself is open. 
\end{proof}
So we have the categoricity result:
\begin{thm}
Conditions \ref{condition1} and \ref{condition2} hold, and the theory of the $j$-function $T\wedge SF\wedge \trdeg(F) \geq \aleph_0$ has a unique model (up to isomorphism) in each infinite cardinality. The standard model $\C_j$ is the unique model of cardinality continuum.
\end{thm}

On one hand, the results of this chapter are very satisfying, in that we started with the purely model theoretic assumption of categoricity, and ended up making deep geometric and arithmetic conjectures (which turned out to be true) regarding Galois representations on the Tate modules of certain abelian varieties. On the other hand, the fact that the conjectures do not quite correspond to the natural ones coming from geometry, indicates that it may be beneficial to construct a model-theoretic setting where there is a more direct correspondence. It is also desirable to have a setting for the $j$-function where the Galois action on Tate modules of CM elliptic curves is seen (i.e. special points would have to not be $\emptyset$-definable), and I try to address these issues in the next chapter.

\begin{rem}
There is another issue, in that I am not sure whether categoricity has anything to say about Galois action on special points in this setting. That is, for special $\tau$ in an imaginary quadratic field $K$, is it essential that the extension $K(j(G \tau)) / K$ is abelian for categoricity to hold?
\end{rem}

\section{Extensions of this work}
There are three main directions in which this work can be extended:

\begin{itemize}
\item Firstly, we can stick with the same language, and consider covers of a Shimura variety. As was previously mentioned, in joint work with Chris Daw we have extended the equivalence results of \S \ref{secequiv} to arbitrary Shimura curves, and the results of \S \ref{secgalrep} to arbitrary modular curves. The fact that the required results regarding Galois representations are not known in the setting of a Shimura curve, realises the techniques of the chapter as a tool for making conjectures. I am unable to give a reference for this work at the moment, since we are in the final stages of writing up the results.
\item Secondly, we can go one-sorted, and try and construct a `pseudo$j$-function' in analogy with Zilber's pseudo-exponential function. The categoricity result of this chapter is the first step of this construction, in analogy with Zilber's categoricity result for the universal cover of the multiplicative group $\C^{\times}$ \cite{zilber2006covers} (with a mistake fixed by Bays \cite{bays2011covers}).
\item Lastly, we can try and look for more general model theoretic setting for the study the analytic universal cover of an arbitrary variety. We pursue this direction for a smooth curve in the next chapter. 
\end{itemize}

Finally, I should note that I originally had very clunky exposition of the material this chapter, where everything was viewed in terms of elliptic curves. I ended up talking about images of Galois representations in what I was calling `projective Tate modules' of elliptic curves (i.e. in $\psl2(\hat{\Z})$). It was through trying to extend this work with Chris Daw to the general Shimura curve case, that the presentation of this work given here came about. I am very thankful to Chris for the many hours of conversations we have had regarding this work, and for making me begin to realise how powerful, and clean the Shimura way of looking at things is. Hopefully, as a result, the exposition given here is much more accessible to geometers and number theorists than my original one. I would also like to thank Martin Bays for explaining the connection between quasiminimality and quasiminimal excellence provided by the recent result of the five author paper \cite{bays2012quasiminimal}, as in the proof of Theorem \ref{thmcat}.

\chapter{The universal cover of a smooth complex curve}\label{chapcur}

The aim of this chapter is to describe a general model theoretic framework for the study of the universal cover of a smooth complex algebraic curve. In particular my aim was to describe a setting in which all previous results regarding the categoricity of `universal covers' of smooth curves, i.e. the multiplicative group (\cite{zilber2006covers}), an elliptic curve (\cite{gavrilovich2006model} and \cite{bays2009categoricity}), and the $j$-function (Chapter \ref{chapj}) may in some sense be embedded. The term `universal covers' is used loosely here, since we want it to include the work of the previous chapter regarding the $j$-function also. This would require not taking the full category of finite, Galois \'etale covers of $\A^1(\C) \backslash \{ 0,1728 \} $, but taking the subsystem of Galois covers as in the last chapter (corresponding to tuples $g \subset \gl2q$). The general construction of this chapter also works for suitable subsystems of covers, however to ease notation we write things out for the full category of finite \'etale Galois covers. \newline

In the case of the $j$-function I was guided by trying to construct a setting where to have a categorical infinitary theory the openness of Galois representations on Tate modules of arbitrary elliptic curves would be essential (i.e. including curves with complex multiplication). This means that the action of $\gl2q$ cannot be in the language but must be seen somehow in the infinitary theory, and this points to putting very little structure on the covering sort. This makes the setting of this chapter very different to that of Gavrilovich (\cite{gavrilovich2006model}) where he defines a topology on the covering sort, however the work here is still strongly influenced by Gavrilovich's. \newline

After describing a general framework where the analytic universal cover and the pro-\'etale cover may be seen as models of the same first-order theory (which is shown to have quantifier elimination), we prove a model-theoretic `comparison theorem' stating that the analytic universal cover may be elementarily embedded in the pro-\'etale cover. From \S \ref{multgpet} onwards the material should be viewed as work in progress because we have to break away from the general situation and restrict our focus on the multiplicative group. Here, aiming for categoricity we restrict the models of the general theory of covering spaces with an $\cL_{\omega_1,\omega}$-sentence stating that the models have a divisible, torsion-free, abelian group on the covering sort. I am hopeful that there is an analogous version of the infinitary axiomatisation of the additive group given in \S \ref{multgpet} for (something like) $\gl2q$, and that the framework described in this chapter will ultimately give a more general, and less restrictive framework for the study of the model theory of $j$-function. I am also hopeful that this axiomatisation is a special case of a more general `independence notion' which will apply to a larger class of curves.


\section{Background and model-theoretic framework}

\subsection{Comparing $\pi_1^{an}$ and $\pi_1^{et}$}

Given a smooth complex algebraic variety $V$, let $V(\C)^{an}$ be the complex analytic variety obtained by considering the $\C$-points of $V$ with its complex analytic topology. The main `comparison theorem' linking the categories of finite \'etale covers of $V(\C)$ and finite analytic covers of $V(\C)^{an}$, allowing us to compare the analytic and geometric \'etale fundamental groups is the Riemann existence theorem:

\begin{thm}[Riemann existence theorem]
Let $V$ be a smooth variety over $\C$. Then the functor sending a finite \'etale covering $f:Y \ra V$ to the finite covering space $f^{an}: Y(\C)^{an} \ra X(\C)^{an}$ is an equivalence of categories.
\end{thm}
The interesting part of the above is clearly the fact that the finite analytic coverings $f^{an}: Y(\C)^{an} \ra X(\C)^{an}$ have a natural algebraic structure.

\begin{dfn}
A group is said to be \textit{residually finite}, if the intersection of all normal subgroups of finite index is trivial. 
\end{dfn}

A group is residually finite if and only if it embeds in its profinite completion via the diagonal embedding. $\Z$ and $\sl2(\Z)$ are residually finite, also free groups, and the direct product of a set of residually finite groups are residually finite. Importantly for the discussion here, we have the following:


\begin{thm}
The analytic fundamental group of a smooth, complex algebraic curve $V$ (considered as a Riemann surface) is residually finite.
\end{thm}

\begin{proof}
Topologically, the Riemann surface corresponding to the curve looks like a $g$-holed torus minus $n$-points. Unless $(g,n) \in \{ (1,0),(0,0),(0,1),(0,2) \}$, such a surface has negative curvature and therefore by the uniformisation theorem has the upper half plane $\H$ as its universal cover. So $\pi_1^{an}(V(\C)^{an})$ embeds into $\psl2(\R) = \Aut(\H)$ as a finitely generated subgroup and it is well known that these groups are residually finite. The other cases are easy to check.
\end{proof}

Fix a smooth, connected affine algebraic curve $V$, defined over a subfield $k$ of $\C$, and let $\Can (V)$ be the category of Galois covers of $V \times_k \C$ considered with their analytic topology. Let $\Feg(V)$ be the category of finite, Galois \'etale covers of $V \times_k \C$.  $\Feg(V)$ is equivalent to the subcategory of finite Galois covers in $\Cov^{an}(V)$ by the Riemann existence theorem. We just consider affine curves here for convenience, but by elimination of imaginaries in the theory of algebraically closed fields of characteristic 0, projective curves may be considered also. \newline

Let $K$ be a field containing the minimal fields of definition of all finite, Galois, \'etale covers $q: Y \ra V \times_k \C$ (i.e. this includes the fields of definition of all the finite \'etale covering automorphisms also).

\begin{dfn}
Fix an embedding $K \hookrightarrow \C$, and define the $K$-enriched algebraically closed field
$$\C_K:= \langle \C; +,\cdot; 0 ,1, \{c \}_{c \in K} \rangle$$
where we have named every element of $K$ with a constant symbol. For each finite \'etale Galois cover $q : Y \ra V \times_k \C$, fix an embedding of the $\C$-points of $Y$ (and its covering maps) as definable sets in the structure $\C_K$. Under this embedding, choose a directed system, $$(Y_i , q_{j,i} : Y_j \ra Y_i)_{i < j \in I}$$
with indexing set $I$, of all covers in $\Feg(V)$.
\end{dfn}

From now on, for a finite, \'etale, Galois cover $q:Y \ra V \times_k \C$, we will just think of $Y$ as a definable set in $\C_K$, i.e. we have identified $Y$ with its $\C$-points, with some fixed embedding into affine space. Instead of $Y(\C)$ I will just write $Y$.

\subsection{The analytic universal cover}

Now take any $x \in V$ and consider the analytic universal cover $p: \U_x \ra V$  i.e. an object representing the fibre functor 
$$\Fib_x : \Cov^{an}(V) \ra \Sets.$$
We identify the fundamental group $\pi_1^{an}(V)$ with the covering automorphisms of $\U_x$ i.e. define
$$\pi_1^{an} := \Aut_{\Cov^{an}(V)}(\U_x / V),$$
 where we let $\pi_1^{an}$ act on $\U_x$ from the left.
Consider the structure
$$\cU : = \langle \langle \U_x ,  \{ \gamma \}_{\gamma \in \pi_1^{an}} \rangle , \C_K , \{p_Y : U \ra Y \}_{Y \in \Feg(V)} \rangle $$
where $p_Y \in \Hom_{\Cov^{an}(V)}(\U_x , Y)$, and we have a unary function symbol for every element of $\pi_1^{an}$ acting on $\U_x$, which now is just considered as a set (i.e. we now forget about the complex topology). Since all analytic universal covers are isomorphic in the complex topology, we will just denote the set $\U_x$ by $\U$. There are as many covering maps $p_Y : \U \ra Y$ as the degree of the cover $q: Y \ra V$, but we just choose one map for each $Y$. \newline

To get some a idea of where we are headed with this, it is good to note the following:
\begin{pro}
Let $Y$ and $Z$ be either finite, Galois covers in $\Cov^{an}(V)$, or $\U$. Then all elements of $\Hom_{Cov^{an}(V)}(Y,Z)$ are definable in $\cU$.
\end{pro}

\begin{proof}
The fibre functor
\begin{align*}
\Fib_x : \Cov^{an}(V) & \ra \Sets \\
(q:Y \ra X ) & \mapsto q^{-1}(x)
\end{align*}
gives an equivalence of categories between $\Cov^{an}(V)$, and $\pi_1^{an}$-sets.
The fundamental group $\pi_1^{an}$ acts on $\Hom_{\Cov^{an}(V)}(\U , Y)$ on the right via composition of maps i.e.
$$p_Y \mapsto p_Y \circ \gamma$$
for $\gamma \in \pi_1^{an}$. By definition, since $\U$ represents the functor $\Fib_x$, for all $q: Y \ra X$ there is a functorial isomorphism
$$\Hom_{\Cov^{an}(V)}(\U , Y) \cong \Fib_x(Y).$$
Combining the two facts above gives a transitive action of $\pi_1^{an}$ on $\Hom_{\Cov^{an}(V)}(\U , Y)$. Finally, all elements of $\Aut_{\Feg(V)}(Y / V)$ are definable in the structure $\C_K$, and the result follows by the Riemann existence theorem.
\end{proof}


\section{The language $\cL_{\Cov(V)}$}

Consider a language $\cL_{\Cov(V)}$ for two-sorted structures of the form 
$$\cM=\langle \langle U ,  \{ \gamma \}_{\gamma \in \pi_1^{an}} \rangle ,  \C_K , \{q_Y : U \ra Y \}_{Y \in \Feg(V)} \rangle.$$
We will usually just denote the universal covering map $p_V$ by $p$. \newline



We extend all of the maps above to tuples component-wise, e.g. for a tuple $u=(u_1,..,u_n)$, define
$$p_{Y}(u):=(p_{Y}(u_1),...,p_{Y}(u_n)).$$

\subsection{The theory of universal covers}

The aim now is to describe a complete first-order theory in the language $\cL_{\Cov(V)}$, of which both the analytic universal cover and the pro-\'etale cover of $V$ are models. The first-order theory will then be augmented with an $\cL_{\Cov(V), \omega_1, \omega}$-sentence from which we would ideally like to be able to recover the analytic universal cover uniquely up to isomorphism.

\subsubsection{The first-order theory}

Consider the following schemes of axioms in the language $\cL_{\Cov(V)}$:
\begin{itemize}
\item [Cov1] (The action of $\pi_1^{an}$): The unary function symbols $\gamma$ give a left action of $\pi_1^{an}$ on $U$ which is free. \\
\item [Cov2] (Compatibility of the covering maps): For all $\gamma \in \pi_1^{an}$ and $i<j \in I$ $$p_{Y_i} \circ \gamma =  q_{Y_{j,i}} \circ p_{Y_j} \circ \gamma.$$
\item [Cov3] (The covering maps are surjective): For all $\gamma \in \pi_1^{an}$ and every $i \in I$,
$$p_{Y_i} \circ \gamma (U)=Y_i.$$
\item [Cov4] (Representability of the fibre functors): Let $G$ be a normal subgroup of finite index in $\pi_1^{an}$  corresponding to a finite Galois cover $Y \in \Feg(V)$. Then for all $v \in V$, and $u \in U$ such that $p(u) = v$, the set $\{p_Y( \gamma_i u)\}$ is in bijection with the fibre $q_Y^{-1}(v)$, as $\gamma_i$ runs through any set of coset representatives for $G$ in $\pi_1^{an}$. Here we include $V$ as a cover of itself with the identity as a covering map
\end{itemize}

\begin{dfn}
Let $\Th(\C_K)$ be the complete first-order theory of $\C_K$ and define $T_{\Cov(V)}$ to be the union of $\Cov1$ to $\Cov4$, and $\Th(\C_K)$.
\end{dfn}
\subsubsection{The $\cL_{\omega_1,\omega}$-theory}

Let SF be an $\cL_{\Cov(V), \omega_1,\omega}$-sentence stating that fibres are $\pi_1^{an}$-torsors 
i.e. 
$$SF:=p(x)=p(y) \ra \bigvee_{\gamma \in \pi_1^{an} } \gamma x =y.$$
Clearly $\cU$ is a model of $T_{\Cov(V)}\wedge SF$, and we will refer to $\cU$ as the \textit{`standard model'}.


\subsection{Analysis of the types}

Since we are looking at structures with a covering sort $U$ which maps onto the structure $\C_K$ and we included the complete first-order theory of $\C_K$ in the theory $T_{\Cov(V)}$, the situation is fairly rigid. In particular, the theory of all finite Galois \'etale covers
$$q_{j,i}: Y_j \ra Y_i$$
in $\Feg(V)$ is already included in the theory. 
This makes the types easy to study.

\begin{dfn}
For a $\cM \models T_{\Cov(V)}$, and a tuple $u = (u_1,...,u_n) \in U$, define $u$ to be $\pi_1^{an}$-\textit{independent} if $$u_i \notin \pi_1^{an}u_j$$ for $i\neq j$, and define $u$ to be $\pi_1$-\textit{independent} if $$p(u_i) \neq p(u_j)$$ for $i\neq j$.
\end{dfn}

\begin{pro}\label{typesgeneral} For a $\pi_1^{an}$-independent tuple $u \in U^n$ and $L \subset \C$, $\qftp_{\cM}(u / L)$ is determined by the formulae expressing the $\pi_1^{an}$-independence of the tuple, and the formulae
$$p_{Y}(u) \in W \cup \{ p_{Y}(u) \notin W' \ | \ W' \subset W, \ \dim W'< \dim W\}$$
where $W$ is the minimal variety over $L$ containing $p_{Y}(u)$ and $Y$ ranges over $\Feg(V)$. 
\end{pro}
\begin{proof}
First note that by $\Cov1$ the action of $\pi_1^{an}$ on $U$ is free, so there are no interesting definable sets in the sort $\langle U, \pi_1^{an} \rangle$ considered as a structure on its own. By $\Cov2$, if $i<j$ then $q_{j,i}(p_{Y_j}( \gamma u_k))=p_{Y_i}( \gamma u_k)$ for $u_k \in u$, and these maps $q_{i,j}$ are $\emptyset$-definable so we only have to consider quantifier free formulae of the form $\varphi(p_Y(\gamma_1 u_1),...,p_Y(\gamma_n u_n))$ where $\varphi$ is some formula of $\C_K$. But elements of $\Aut_{\Feg(V)}(Y / V)$ are definable also, so by $\Cov4$ we have $p_Y(\gamma u_k) \in \dcl( p_Y(u_k))$ for all $\gamma \in \pi_1^{\an}$, and the result follows by the description of the types in $\C_K$.

\end{proof}

So we may identify $\qftp(u)$ with this countable collection of varieties $W$.


\subsection{Quantifier elimination and completeness of the first-order theory}

\begin{dfn}
For an $\cL_{\Cov(V)}$-structure $\cM$, a subset $S$ of the sorts of $\cM$ and a tuple belonging to the sorts of $S$, denote by $\tp_S(x)$ ($\qftp_S(x)$) the formulae (quantifier free formulae) satisfied by $x$ in the sorts in $S$. 
\end{dfn}

\begin{pro}\label{QEgen}
Let
$$\cM  \textrm{ and } \cM' $$
be $\aleph_0$-saturated models of the first-order theory $T_{\Cov(V)}$ and
$$\sigma : \cM \ra \cM'$$
a partial isomorphism with finitely generated domain $D$. Then given any $z \in \cM$, $\sigma$ extends to the substructure
generated by $D \cup \{z \}$.
\end{pro}

\begin{proof}
Let $u \in U^n$ generate $D\cap U$. We may assume that $z \in U$ and that the tuple $(u,z)$ is $\pi_1^{an}$-independent. Let $C$ be a finite subset of $\C_K$ such that $C \cup u$ generates $D$. We need to show that there is $z' \in U'$ such that for all $Y \in \Feg(V)$ we have
$$\Loc(p_{Y}'(z') / \sigma(L))=\Loc(p_{Y}(z) / L)$$
where
$$L:=K \left( C, \bigcup_{i \in I, \gamma \in \pi_1^{an}}p_{Y_i}( \gamma u) \right)$$
and $\Loc(x /L)$ denotes the minimal variety over $L$ containing $x$. Since $\cM'$ satisfies \Cov3, for all $Y$, $p_Y'$ surjects onto $Y$ and therefore onto $\Loc(p_Y(z) /L)$ and there is $z' \in U'$ such that $p_Y'(z')  = p_Y(z) $. If $i \leq j$ in the directed system $I$, then by $\Cov2$, $\qftp_{\C_K}(p_{Y_j}(u) / L)$ determines $\qftp_{\C_K}(p_{Y_i}(u) /L)$, so $\qftp(z / D)$ is finitely satisfiable, and therefore by the  compactness theorem is satisfiable and is realised in $\cM'$ by $\aleph_0$-saturation. The case where $z$ is in $\C_K$ is covered by the above.
\end{proof}

\begin{cor}
$T_{\Cov(V)}$ is complete, has quantifier elimination and is superstable.
\end{cor}

\section{The pro-\'etale cover}

Keeping the same directed system $ (Y_i , q_{j,i} : Y_j \ra Y_i)_{i < j \in I}  $ as earlier, consider a pro-\'etale cover of $V$
$$\hat{\U}: = \varprojlim_{I} Y_i$$
and the \textit{geometric \'etale fundamental group}
$$\pi_1^{et}(V) := \varprojlim_{I} \Aut_{\Feg(V)} Y_i,$$
which comes equipped with its natural topology as a projective limit of finite discrete groups. Denote the pro-universal covering map by $\hat{p}$, i.e.
$$\hat{p}: \hat{\U} \ra V$$
and the projections to the $Y_i$ by $\hat{p}_{Y_i}$. Note that the construction of $\hat{\U}$ is dependent on the choice of embedding $K \hookrightarrow \C$, how we embedded $Y$ in each sort $\C_K$ as a definable set, and the choice of system of maps $(q_{j,i})_I$. \newline

Since $\pi_1^{an}$ is residually finite, we may (diagonally) embed $\pi_1^{an}$ in $\widehat{\pi_1^{an}}$, which by the Riemann existence theorem is isomorphic to the geometric \'etale fundamental group. Using this embedding we have an action of $\pi_1^{an}$ on $\hat{\U}$.

\begin{dfn}
Using the above embedding of $\pi_1^{an}$ in $\pi_1^{et}$ we may define the pro-\'etale cover (as an $\cL_{\Cov(V)}$-structure)
$$\hat{\cU}:=\langle \langle \hat{\U} ,  \{ \gamma \}_{\gamma \in \pi_1^{an}} \rangle , \C_K , \{\hat{p}_Y : U \ra Y \}_{Y \in \Feg(V)} \rangle.$$
\end{dfn}
The following is then immediate:
\begin{pro}
$\hat{\cU}$ is a model of $T_{\Cov(V)}$.
\end{pro}




So we have a setting where the analytic universal and pro-\'etale covers are models of the same first-order theory $T_{\Cov(V)}$. The following could be described as a `comparison theorem' in our model theoretic setting (in a similar fashion to the interesting direction to the Riemann existence theorem being called a comparison theorem by geometers), and it is here that we make crucial use of the residual finiteness of  $\pi_1^{an}$.

\begin{pro} Let $\cM$ be a model of $T_{\Cov(V)}\wedge SF$. Then there is an elementary embedding of $\cL_{\Cov(V)}$-structures
$$\i: \cM \hookrightarrow \hat{\cU}.$$
\end{pro}

\begin{proof}
Define the map
\begin{align*}
\i: \cM & \hookrightarrow \hat{\cU} \\
u \in U & \mapsto (p_Y(u))_{Y \in \Feg(V)},
\end{align*}
letting $\i$ be the identity on $\cF$. Given a finite, Galois \'etale cover
$$f: Y \ra V,$$
 by the Galois correspondence for covers and the Riemann existence theorem, $Y$ is isomorphic to $G \backslash \U$ for a normal subgroup $G$ of finite index in $\pi_1^{an}$. 
For all $i \in I$ choose an isomorphism $Y_i \cong G_i \backslash \U$ such that we have 
$$\i(u) = (G_i u)_I.$$
Now we see that the map $\i$ is injective, since $\pi_1^{an}$ acts freely, and the intersection of all the $G$'s is trivial (since $\pi_1^{an}$ is residually finite). The left action of $\pi_1^{an}$ on $U$ induces a natural left action on the fibres $f^{-1}(p(u))$ via $Gu \mapsto G \gamma u$, and under this action $\i(\gamma u) = (G \gamma u)_G = \gamma \i(u)$. By definition $\i(p_Y(u)) = p_Y(u) =\hat{q}_Y( \i ( u))$ so $\i$ is an embedding. The embedding is elementary by quantifier elimination.
\end{proof}

We have the following analogue of \ref{type in etale realise SF}:
\begin{pro}\label{realisegenet}
Given a $\pi_1$-independent tuple $u \subset \hat{\U} $, there is a model of $T_{\Cov(V)}\wedge SF$ realising $\tp(u)$. 
\end{pro}
\begin{proof}
Similarly to the proof of  \ref{type in etale realise SF}, let $$U':=\pi_1^{an} u \cup \{x \in \i(\U) \ | \ \hat{p}(x) \neq \hat{p}(u) \}$$
and restrict the covering map $\hat{p}$ to $U'$.
\end{proof}

\subsection{The action of Galois on fibres in the pro-\'etale cover}
Let $v \in V$ and let $L / K$ be the minimal subfield of $\C$ containing the coordinates of $v$. Since a finite \'etale cover $f:Y \ra V$ is defined over $K$, there is a concrete left action of $G_{L}: = \Gal(\bar{L} / L)$ on the fibre $f^{-1}(v)$ via its action on coordinates in the structure $\C_K$. There is also a left action of $\Aut_{\Feg(V)}(Y / V)$ on the fibre, and in the limit this gives a left action of the geometric \'etale fundamental group $\pi_1^{et}(V)$ on the fibre $\hat{p}^{-1}(v)$ in $\hat{\U}$. The fibre $\hat{p}^{-1}(v)$ is a left $\pi_1^{et}$-torsor, so for $\sigma \in G_L$ and $y$ in the fibre, there is $g_{\sigma} \in \pi_1^{et}(V)$ such that $y\sigma = g_{\sigma}(y)$. Since the elements of $\Aut_{\Feg(V)}(Y / V)$ are defined over $L$, the actions commute and we get a continuous homomorphism
\begin{align*}
\rho_v : G_L & \ra \pi_1^{et}(V) \\
\sigma & \mapsto g_{\sigma}.
\end{align*}
Clearly the above can be extended to tuples $v \in V^n$. By exactly the same argument as in \S \ref{secgalpro}, we have the following:

\begin{pro}
Let $v \in V^n$ be a tuple of distinct elements and let $L: = K(v)$. Then the cosets of the image of the representation
$$\rho_v : G_L \ra \pi_1^{et}(V)^n$$
are in bijection with the types of tuples $u \in \hat{\U}^n$ such that $\hat{p}(u_i)=v_i$.
\end{pro}
\begin{proof}
By the same argument as \ref{progal}, the cosets of $\rho_v(G_L)$ are in bijection with the number of orbits of $G_L$ on the fibre $\hat{p}^{-1}(v)$. But by the description of the types (\ref{typesgeneral}) and quantifier elimination, we can just look at the corresponding pro-definable set in the algebraically closed field $\C_L$ which determines the type. $G_L$ acts as the group of automorphisms of $\C_L$, and the result now follows by the (model-theoretic) homogeneity of the field $\C_L$ (i.e. if $y$ and $y'$ in the fibre are not conjugated by $G_L$, then their field types differ).
\end{proof}

\section{The multiplicative group}\label{multgpet}

It is here where the general theory finishes. It turns out that the theory $T\wedge SF$ is too weak on its own to give categoricity for the most basic case I had in mind, the universal cover of the multiplicative group $\C^{\times}$. In this section we will see why this is the case, and how to fix this by including the information of the additive group on the covering sort with an additional $\cL_{\omega_1,\omega}$-sentence. \newline

So consider the multiplicative group $\C^{\times}$, which may be seen in this situation as $\A^1(\C) \backslash \{0 \}$, the $\C$-points of the $\Q$-variety  $$\Spec(\Q[X,Y] / (XY-1)).$$ It is well known (for example by an application of Riemann-Hurwitz) that the finite \'etale covers of $\C^{\times}$ are of the form
$$\C^{\times} \overset{x \mapsto x^n}{\lora} \C^{\times},$$
and that the \'etale automorphisms of these covers are given by multiplications by $n$'th roots of unity. So in this situation the field of definition of all finite \'etale covers and their automorphisms is $\Q^{\rm{cyc}}$, and we consider $\C^{\times}$ as the subset of $\C^2$ defined by the formula $xy=1$ in the structure $$\C_{\Q^{\textrm{cyc}}}= \langle \C, + , \cdot,  \{\Q^{\textrm{cyc}} \} \rangle.$$ 

 In this example, we may emerge from the abstraction of the previous discussion and see concretely what is going on: Given $x \in \C^{\times}$, the fibre $\hat{p}^{-1}(x)$ is just a compatible sequence $(x^{1/n})_n$ of $n$'th roots of $x$, and $$\pi_1^{et}(\C^{\times}) \cong \hat{\Z}$$ acts on the fibre via multiplication by a compatible sequence of roots of unity. The analytic universal cover is the exponential function
$$\ex: \C \ra \C^{\times}$$
and we may take the maps
$$\ex_n: x \mapsto \ex(x/n)$$
as $n$ ranges over the natural numbers, for the set of maps $(p_{Y})_{Y \in \Feg(\C^{\times})}$. \newline

We denote the map in an arbitrary model corresponding to $\ex_n$ by $p_n$, and the covering maps in the pro-\'etale cover by $\hat{\ex}_n$.

\begin{pro}
There are $2^{\aleph_0}$-types of $\pi_1$-independent tuples realisable in $\hat{\cU}$.
\end{pro}
\begin{proof}
Consider for example $(x,y) \in \hat{\U}^2$ such that $\hat{\ex}_n (x)=\hat{\ex}(y)$ and $\hat{\ex}(y) \neq \hat{\ex}(x)$. Let $\hat{\ex}(x) = a$ $\hat{\ex}(y) = b$. Then the image of the Galois representation
$$\rho_{(a , b)} : \Gal(\C / \Q^{\textrm{cyc}}(a,b)) \lora \hat{\Z} \times \hat{\Z}$$
has index $2^{\aleph_0}$.
\end{proof}

So by Keisler's theorem (\ref{keislerthm}) and \ref{realisegenet}, in the case of $\C^{\times}$ the theory $T_{\Cov(\C^{\times})}\wedge SF$ is not $\aleph_1$-categorical. However we can fix the situation by axiomatising the action of a divisible, torsion-free abelian group on the covering sort with an $\cL_{\Cov(\C^{\times}) , \omega_1,\omega}$-sentence.

\subsection{Axiomatisation of the additive group}\label{additive}

Consider the following $\cL_{\Cov(\C^{\times}) , \omega_1,\omega}$-sentences:
\begin{itemize}
\item [Identity] $\exists ! x \ \ \bigwedge_{n \in \N}  p_n(x)=1$;
\item [Closure] $\forall x \forall y \exists ! z \ \ \bigwedge_{n \in \N} p_n (x) . p_n(y) = p_n(z)$;
\item [Inverses] $\forall x \exists ! y \ \ \bigwedge_{n \in \N} p_n(x).p_n(y)=1$;
\item [Associative] $\forall x_1,...,x_7 \ \ (\bigwedge_{n \in \N} p_n(x_1).p_n(x_2)=p_n(x_4) \wedge p_n(x_2).p_n(x_3)=p_n(x_5) \wedge$ \\
$ p_n(x_4).p_n(x_3)=p_n(x_6)  \wedge p_n(x_1).p_n(x_5)=p_n(x_7))  \ra x_6 = x_7$;
\item [Divisible] $\bigwedge_{n \in \N} \forall x \exists!y \ \ \bigwedge_{m \in \N} p_{nm}(x)=p_m(y)$;
\item [Torsion-free] $ \forall x  \ \ \left( \bigvee_{n \in \N} \exists y \bigwedge_{m \in \N} p_{nm}(x)=p_m(y) \right) \ra \bigwedge_{n \in \N}  p_n(x)=1$
\end{itemize}
Let $GP$ be the conjunction of these sentences. We will now study the theory $T_{\Cov(\C^{\times})}\wedge SF\wedge GP$. The following is immediate from the axioms:
\begin{pro}
Let $\cM \models T_{\Cov(\C^{\times})}\wedge SF\wedge GP$. Then there is a divisible, torsion-free, abelian group $\langle U,+ ,0 \rangle$, $\cL_{\Cov(\C^{\times})}$-definable on the sort $U$. The covering map $p$ is a homomorphism from $\langle U,+,0 \rangle$ onto the multiplicative group $\langle \C^{\times}, \cdot ,1 \rangle$. 
\end{pro}
Clearly the standard model is a model of $T_{\Cov(\C^{\times})}\wedge SF \wedge GP$ (consider the structure $\langle \C , +, 0 \rangle$ on the covering sort), where the $+$ operation and the identity 0 are the ones we expect. As a result, the pro-\'etale cover $\hat{\cU}$ is a model of $T_{\Cov(\C^{\times})} \wedge GP$ with the induced component-wise addition.


\subsection{Categoricity}

Now we go about showing that this theory is categorical if and only if certain openness conditions on Galois representations in the geometric \'etale fundamental group hold. There is nothing new here in two ways. Firstly, the proofs are direct generalisations of the analogous theorems for the $j$-function in the previous chapter. Secondly, now we have seriously restricted the class of models of the $\cL_{\omega_1,\omega}$-theory to have a $\Q$-vector space on the covering sort and the covering map to be a homomorphism, we are in the exact situation considered in \cite{zilber2003model} and \cite{zilber2006covers}, where the categoricity result for this class of structures has been proven already. That is, in \cite{zilber2006covers} Zilber considers the two-sorted structure
$$\langle \C,+,0 \rangle \overset{\ex}{\lora} \langle \C,+,\cdot,0,1 \rangle,$$
along with an $\cL_{\omega_1,\omega}$-sentence saying that $\langle \C,+,0 \rangle $ is a divisible, torsion-free, abelian group, and $\ex$ is a surjective homomorphism onto the multiplicative group $\langle \C^{\times}, \cdot ,1 \rangle$, with $\ker{\ex}\cong \Z$. He shows that this theory has a unique model in every uncountable cardinality (modulo an error which was fixed by Bays in \cite{bays2011covers}). In \cite{zilber2003model}, Zilber considers the universal cover of an arbitrary semi-abelian variety with the suitable generalisation of the $\cL_{\omega_1,\omega}$-sentence of \cite{zilber2006covers}, and finds necessary and sufficient conditions for categoricity of this sentence (which are presented in a different form, but are equivalent to the ones given here in the case of $\C^{\times}$). \newline

The interesting thing is the fact that the situation there can in some sense be embedded in the situation here (I would like to be more precise about the sense in which the `situation can be embedded', but I am unsure at the moment). So we state the results below with just the main steps in the proof outlined, because we are just looking at things from a slightly different perspective.

\begin{dfn}
Define a dependence relation on tuples $u \in U^n$ in models of $T_{\Cov(\C^{\times})}\wedge SF\wedge GP$, to be the $\Q$-linear closure with respect to the group $\langle U ,+,0 \rangle$. We say that a tuple $u \in U^n$ is \textit{independent} iff it is independent with respect to this linear closure.
\end{dfn}
Note that a tuple $u \in U^n$ is (in)dependent with respect to this closure operator iff $p(u)$ is multiplicatively (in)dependent in $\langle \C^{\times}, \cdot,1 \rangle$.\newline

We now give the three conditions which will be equivalent to categoricity, the first of which is an elementary result of algebra and is of a slightly different nature to the other two:

\begin{cdn}\label{condtorsion}
$\Gal(\Q^{\textrm{cyc}} / \Q) \cong \hat{\Z}^{\times}$.
\end{cdn}

\begin{cdn}[Arithmetic homogeneity]\label{AH} Let $x \in \C^{\times n}$ be a multiplicatively independent tuple, and $L \subset \C$ a finitely generated extension of $\Q^{\textrm{cyc}}$ containing the coordinates of $x$. Then the image of the Galois representation $$\rho_{x}: G_L \lora \pi_1^{et}  ( \C^{\times})^n$$
is open.
\end{cdn}

\begin{cdn}[Geometric homogeneity]\label{GH}
Let $x$ and $L$ be as above, with $A \subset \C$ be a countable algebraically closed field such that $x_i \notin A$, and let $B$ be the compositum $LA$. Then the image of the representation
$$\rho_{x} : G_B \lora  \pi_1^{et}  ( \C^{\times})^n$$ is open.
\end{cdn}

\begin{rem}
Let $\cM$ and $\cM'$ be models of $T_{\Cov(\C^{\times})} \cup SF \wedge  GP$. Then there is a unique element $0$ in the covering sort $U$ of $\cM$ provided by the axiom $GP$. Any isomorphism $\sigma: \cM \ra\cM'$ must map $0$ to the corresponding element $0' \in U'$. Since there is a symbol for every element of $\pi_1^{an} \cong \Z$ in the language, for all $\gamma \in \pi_1^{an}$ we must have $\sigma(\gamma 0) = \gamma \sigma(0) = \gamma 0'$. So $\sigma$ does not have any choice as to where it sends elements of $\pi_1^{\an} 0$. As a result,  $\sigma$ must make up for this with a field automorphism i.e. Condition \ref{condtorsion} must hold. As noted above, this is a well known result and below we will assume that our partial isomorphisms include an isomorphism on the substructure generated by $0$ i.e. $\pi_1^{\an}0 \cup \Q^{\textrm{cyc}}$.
\end{rem}

\begin{lem}[$\aleph_0$-homogeneity over $\emptyset$]
Suppose Conditions \ref{condtorsion} and \ref{AH} hold. Let $\cM, \cM'$ be models of $T_{\Cov(\C^{\times})} \cup SF \wedge  GP$ and
$$\sigma : \cM \ra \cM'$$
a partial isomorphism with finitely generated domain $D$. Then given any $z \in \cM$, $\sigma$ extends to the substructure generated by $D \cup \{z \}$.
\end{lem}

\begin{proof} Let $u \in U^n$ generate $D\cap U$. We may assume that $z \in U$, and since $\cM \models GP$ we may assume that the tuple $(u,z)$ is independent (it is crucial that we can make an assumption of this form, or the corresponding Galois representations may not be open). The rest of the proof is now almost identical to the proof of \ref{homo}.
\end{proof}

In the same way, \ref{GH} implies $\aleph_0$-homogeneity over countable models. We may also define a pregeometry with the countable closure property on models of $T_{\Cov(V)}\wedge SF \wedge GP$ as 
$$\cl:= p^{-1} \circ \acl_{\C_{\Q^{\textrm{cyc}}}} \circ p,$$
 and by the same argument as in \ref{thmcat} we have the following:
\begin{thm} Suppose that Conditions \ref{condtorsion}, \ref{AH} and \ref{GH} hold. Then the standard model  $\cU$ is the unique model of $T_{\Cov(V)}\wedge SF\wedge GP$ of cardinality continuum.
\end{thm}

Note that I could have set things up to involve models of greater cardinality as in Chapter \ref{chapj}, but for me the interesting part of the theory of quasiminimal excellence is the passage from a unique model of cardinality $\aleph_1$ to the standard model being unique, so I chose to state things in less generality in this chapter.

\subsubsection{Necessary conditions for categoricity}

We also have the following analogue of \ref{number types}, and \cite[Proposition 1]{zilber2003model}:
\begin{pro}\label{number types2}
Let $x \in \C^{\times n}$ be multiplicatively independent. Then the number of types of (independent) tuples realisable in $\hat{\cU}$ such that $\hat{\ex}(u)=x$ is either finite or $2^{\aleph_0}$.
\end{pro}

Again, we also have the analogue of \ref{type in etale realise SF} and \cite[Lemma 3.5]{zilber2003model}:
\begin{pro}\label{realisegenet}
Given an independent tuple $u \subset \hat{\U} $, there is a model of $T_{\Cov(V)}\wedge SF \wedge GP$ realising $\tp(u / L)$ for any $L \subseteq \C$ such that $\hat{\ex}(u)$ is multiplicatively independent over $L$. 
\end{pro}


Again, in direct analogy with Chapter \ref{chapj} and \cite{zilber2003model}, by Keisler's theorem and the above two propositions, for the theory $T_{\Cov(V)}\wedge SF\wedge GP$ to be $\aleph_1$-categorical there must be finitely many types of independent tuples over a finite set. Since independent tuples in the pro-\'etale cover are realised in models of $T_{\Cov(\C^{\times})}\wedge SF \wedge GP$, we are able to focus our attention there. To a fibre in the pro-\'etale cover and some parameters, we may assign a Galois representation, and certain sets of types are in bijection with the cosets of the corresponding Galois representations. So the number of models of the theory is related to its stability properties, which is related to the images of Galois representations in the \'etale fundamental group. As in Chapter \ref{chapj}, to have categoricity implying the geometric homogeneity condition, we need to show that the theory has the amalgamation property. However since models of $T\wedge SF \wedge GP$ have a $\Q$-vector space in the covering sort, this follows from \cite[Proposition 3]{zilber2003model}. So by exactly the same arguments of the previous chapter, we have the following:

\begin{thm}
If $T_{\Cov(\C^{\times})}\wedge SF \wedge GP$ is $\aleph_1$-categorical, then Conditions \ref{condtorsion}, \ref{AH} and \ref{GH} hold.
\end{thm}



\section{Images of Galois representations}

In this section, we verify that arithmetic and geometric homogeneity hold in this setting. The reason for reproving the result here is that the proof is considerably shorter than the previous ones in \cite{zilber2006covers} and \cite{bays2009categoricity}. The proof is based on \cite[V \S 4]{lang1979elliptic}, and Bays used the same ideas in his thesis (where he gives a proof which works for elliptic curves also), but since he did not have the main result of \cite{bays2012quasiminimal} at hand he had to consider representations over independent systems of algebraically closed fields which complicates things. The only new thing here really, is that we observe that $\Z_l$ is a principal ideal domain and use the structure theorem for modules over a PID at the end of the proof. We will use the notation $\mu$ ($\mu_N$) for the multiplicative group of ($N^{th}$) roots of unity. We are going to prove the following two theorems:

\begin{thm}[Arithmetic homogeneity]\label{thumb0}Let $K$ be a number field and $\bar{a} \subset K^{\times}$, a multiplicatively independent tuple. Then the image of the continuous homomorphism
$$\rho_{\bar{a}} : \Gal(K(\mu, \hat{p}^{-1}(\bar{a})) / K(\mu )) \hookrightarrow \pi_1^{et}(\C^{\times})^r$$
 is open.
\end{thm}

\begin{thm}[Geometric homogeneity]\label{thumb1}Let $K \subset \C$ be an algebraically closed field, and $\bar{a} \subset \C-K$ multiplicatively independent. Then the image of the continuous homomorphism
$$\rho_{\bar{a}} : \Gal(K( \hat{p}^{-1}(\bar{a})) / K(\bar{a})) \hookrightarrow \pi_1^{et}(\C^{\times})^r$$
 is open.
\end{thm}

In fact, geometric homogeneity falls out of the proof of arithmetic homogeneity, and is easier since no cohomology is needed to deal with the roots of unity.

\begin{rem}
The arithmetic homogeneity statement above only covers Galois representations over a number field, which is enough to do the back and forth argument needed for quasiminimality (and therefore categoricity). However, assuming that the theory $T_{\Cov(\C^{\times})}\wedge SF\wedge GP$ in this situation is $\aleph_1$-categorical will still give the full arithmetic homogeneity condition.
\end{rem}

A subgroup of a profinite group is open iff it is closed and of finite index. To prove that the image of the Galois representations are open in the profinite group 
$$\pi_1^{et}(\C^{\times}) \cong \hat{\Z} \cong \prod_{l} \Z_l$$ we split the proof into a `horizontal' and a `vertical' result:

\begin{lem}\label{horiz}[Horizontal openness]
The image of the continuous homomorphism
$$\rho_{l^{\infty}} : \Gal(K(\mu_{l^{\infty}},\hat{p}_l^{-1}(\bar{a})) / K(\mu_{l^{\infty}}))  \hookrightarrow \hat{\Z}_l^r$$
is surjective for almost all primes $l$.
\end{lem}

\begin{lem}\label{vert}[Vertical openness]
The image of the continuous homomorphism
$$\rho_{l^{\infty}} : \Gal(K(\mu_{l^{\infty}},\hat{p}_l ^{-1}(\bar{a})) / K(\mu_{l^{\infty}}))  \hookrightarrow \hat{\Z}_l^r$$
is open for all primes $l$.
\end{lem}
For the definition of the $l$-adic fibre $ \hat{p}_l^{-1}(\bar{a})$, just take all compatible sequences of $l^{n}$th roots of $\bar{a}$ for all $n$ and for the definition of $\rho_{l^{\infty} , \bar{a}}$ just choose one of them and proceed as before.

\begin{lem}\label{freeab} \cite[2.1]{zilber2006covers}
Let $K$ be a finitely generated extension of $\Q$. 
Then
$$K^{\times} \cong A \times \left( K \cap \mu \right)^{\times} $$
where $A$ is a free abelian group. If $L$ is algebraically closed, then
$$ (KL)^{\times} \cong B \times L^{\times} $$
where $B$ is a free abelian group.
\end{lem}

Let $K$ be field which is either a number field or the compositum of a finitely generated extension of $\Q$ with an algebraically closed field and $\bar{a}$ a multiplicatively independent tuple coming from the free abelian part of $K^{\times}$ (as in \ref{freeab}). Let $\Gamma$ be the multiplicative subgroup of $K^{\times}$ generated by $\bar{a}$, and $\Gamma'$ the division group of $\Gamma$ in $K^{\times}$ i.e.
$$\Gamma':=\{x \in K^{\times} \ | \ x^n \in \Gamma \textrm{ for some } n \in \N \}.$$

\begin{cor}\label{corfingen}
Let $K$ be a finitely generated extension of $\Q$ and $\Gamma$ a finitely generated subgroup of $K^{\times}$. Then $[\Gamma':\Gamma]$ is finite.
\end{cor}

At this point, we make the observation that if $n$ is prime to the index $[\Gamma':\Gamma]$, then
$$\Gamma^n = \Gamma \cap K^{\times^n}.$$

\begin{thm}\label{kummer}[Kummer theory]
Let $L$ be a field containing the $n^{th}$ roots of unity $\mu_n$ and $\Gamma$ a finitely generated subgroup of $L^{\times}$. Then
$$\Gal(L( \Gamma^{1 / n}) / L) \cong \Gamma /  \Gamma \cap L^{\times^n} .$$
\end{thm}

\begin{proof}
We start with the `Kummer exact sequence' of $G:=\Gal(L(\Gamma^{1/n}) / L)$-modules
$$\mu_n \lora L(\Gamma^{1/n})^{\times} \overset{x\mapsto x^n}{\lora} L(\Gamma^{1/n})^{\times^n},$$
and we take cohomology to get a long exact sequence
$$H^0(G,\mu_n)\ra H^0(G,L(\Gamma^{1/n})^{\times}) \overset{x\mapsto x^n}{\ra} H^0(G,L(\Gamma^{1/n})^{\times^n})  \overset{\delta}{\ra}  H^1(G, \mu_n) \ra H^1(G, L(\Gamma^{1/n})^{\times}). $$
$G$ acts trivially on $L$ and therefore on $\mu_n$, and $H^1(G, L(\Gamma^{1/n})^{\times})=0$ by Hilbert's theorem 90, so we get an exact sequence
$$\mu_n \lora L^{\times}  \overset{x\mapsto x^n}{\lora} L^{\times}\cap L(\Gamma^{1/n})^{\times^n}   \overset{\delta}{\lora}  \Hom(G,\mu_n) \ra 0.$$
Now we note that $G$ is a finite abelian group of exponent $n$, and so is isomorphic to its character group $\Hom(G,\mu_n)$, giving
$$\Gal(L(\Gamma^{1/n}) / L) \cong \left( L^{\times} \cap L(\Gamma^{1/n})^{\times^n} \right) /L^{\times^n} =\Gamma L^{\times^n} / L^{\times^n} \cong \Gamma /  \Gamma \cap L^{\times^n} .$$
\end{proof}

\subsection{Horizontally}

\begin{pro}\label{arith2}\cite[V 4.2]{lang1979elliptic}
Let $K$ be a number field, $n$ be coprime to $2[\Gamma':\Gamma]$ and
$$\Gal(K(\mu_n) / K) \cong (\Z / N \Z)^{\times}.$$
Then 
$$\Gamma \cap K(\mu_n)^{\times^{l^n}} = \Gamma^n.$$
\end{pro}

\begin{proof}
Suppose not. Then for some prime $p|n$ there is $\alpha \in \Gamma$ such that
$$\alpha = \beta^p \ \ , \ \ \beta \in K(\mu_n) - K$$
(where $\beta \notin K$ since $n$ is coprime to $2[\Gamma':\Gamma]$). Now since $\beta \notin K$, the polynomial $X^p-\alpha$ is irreducible over $K$ and so $\beta$ has degree $p$ over $K$. But we have
$$[K(\mu_p) :K]=p-1$$
so $\beta$ has degree $p$ over $K(\mu_p)$ also. The Galois extension
$$K(\mu_p, \beta) / K$$
is non-abelian and therefore cannot be contained in $K(\mu_n)$, since the extension $K(\mu_n) / K$ is abelian.
\end{proof}

\begin{pro}
Let $K$ be a finitely generated extension of $\Q$. Then 
$$\Gal(K(\mu) / K)$$
is isomorphic to an open subgroup of $\hat{\Z}^{\times}$.
\end{pro}
\begin{proof}
It is an elementary result that
$$\Gal(\Q(\mu) / \Q) \cong  \hat{\Z}^{\times}$$
and the result follows almost immediately, remembering that open is equivalent to closed and finite index.
\end{proof}

\begin{proof} [Proof of Lemma \ref{horiz}]
Since $\Gal(K(\mu) / K)$
is isomorphic to an open subgroup of $\hat{\Z}^{\times}$, there is $l_0$ such that if $l \geq l_0$ then
$$\Gal(K(\mu_{l^n}) / K) \cong (\Z / l^n \Z)^{\times}$$
for all $n$. By Kummer theory (\ref{kummer})
$$\Gal(K(\mu_{l^n}, \Gamma^{1 / l^n}) / K(\mu_{l^n})) \cong \Gamma /  \Gamma \cap K(\mu_{l^n})^{\times^{l^n}} $$
but by \ref{corfingen}, the index $[\Gamma':\Gamma]$ is finite and if $l \geq l_0$ is coprime to $2[\Gamma':\Gamma]$ then by \ref{arith2} we have
$$\Gamma \cap K(\mu_{l^n})^{\times^{l^n}} = \Gamma^{l^n}$$
and so 
$$\Gal(K(\mu_{l^n}, \Gamma^{1 / l^n}) / K(\mu_{l^n})) \cong \Gamma /  \Gamma^{l^n} \cong (\Z / l^n \Z)^r. $$
\end{proof}

\subsection{Vertically}

\begin{lem}[Sah's Lemma]
Let $M$ ba a $G$-module and let $\alpha$ be in the centre $Z(G)$. Then $H^1(G,M)$ is killed by the endomorphism
$$x \mapsto \alpha x-x$$
of $M$. In particular, if $f$ is a 1-cocycle and $g \in G$ then
$$(\alpha -1) f(g) = (g-1)f(\alpha).$$
\end{lem}

Note that we presented Sah's lemma in its familiar additive notation, but we will now apply it in multiplicative notation.

\begin{pro} \label{sahcor}
Let $a \in K^{\times}$ and suppose that $\rho_{l^m,a}(h)=0$ for all  $h \in H:= \linebreak \Gal(K(\mu_{l^m} ,a^{1/l^m}) / K(\mu_{l^m}))$.  Then there is a constant $\lambda (K,l) \in \N$ (independent of $m$) such that $a^{\lambda} \in K^{\times^{l^m}}$.
\end{pro}

\begin{proof}
If $b$ is an $l^m$-th root of $a$, and $\rho_{l^m,a}(h)=0$ for all $h \in H$, then $b$ is in $K(\mu_{l^m})$ since it is fixed by everything in $H$ and the map
$$g \mapsto \frac{g b}{b}$$
is a cocycle from $G:=\Gal(K(\mu) / K)$ to $\mu_{l^m}$. $G$ is isomorphic to an open subgroup of $\hat{\Z}^{\times}$, and units in $\Z_l$ are those which aren't divisible by $l$, so there is $\alpha \in G$ such that $\alpha$ acts on $\mu$ as raising to the power $\lambda:=l^{m_0}+1$ for some $m_0$. Now by Sah's lemma there exists $\zeta \in \mu_{\lambda}$ such that
$$ \left( \frac{g b}{b} \right)^{\lambda} = \frac{g \zeta}{\zeta}.$$
Now if we let $\eta^\lambda = \zeta$, then we have
$$ \left( \frac{g b}{b} \right)^\lambda = \frac{g \zeta}{\zeta} = \frac{g \eta^\lambda}{\eta^\lambda} = \left(\frac{g \eta}{\eta}\right)^\lambda,$$
so $c:=(b. \eta^{-1})^\lambda $ is fixed under all $g \in G$ and is therefore in $K^{\times}$, and $a^\lambda=c^{l^m}$.
\end{proof}



\begin{proof}[Proof of \ref{vert}]
The image of $\rho_{\bar{a},l^ \infty }$ is a closed subgroup of $ \Z_l^r$ and is therefore a $\Z_l$-submodule. $\Z_l$ is a PID, so a $\Z_l$-submodule of the free module $\Z_l^r$ is free of dimension $s \leq r$. If $s \neq r$ then there are $\eta_1 ,...,\eta_r \in \Z_l$, not all zero such that
$$\eta_1 \rho_{a_1}(\sigma)+\cdots + \eta_r \rho_{a_r}(\sigma)=0$$
for all $\sigma$. Let $\eta_{j,m} \in \Z$ such that $\eta_{j,m}\equiv \eta_j \mod l^m$, and let
$$a=\eta_{1,m}a_1 + \cdots + \eta_{r,m}a_r.$$
Then $\rho_{a,l^m} = 0$ on $\Gal(K(\mu_{l^m},a^{1/l^m}) / K(\mu_{l^m} ))$ (since $\rho_{-,l^m}(-)$ as a function of two variables is bilinear) 
so by \ref{sahcor} $a^{\lambda} \in K^{\times^{l^m}}$, and this cannot happen for arbitrarily large $m$ since $K$ is finitely generated. So the image of $\rho_{\bar{a},l^ \infty }$ is an $r$-dimensional $\Z_l$-submodule of $\Z_l^r$. A $\Z_l$-submodule of $\Z_l$ is either 0, or is of the form $l^k \Z_l$ for some $k$, in which case it has index $l^k$ in $\Z_l$ and is isomorphic to $\Z_l$ and the result follows.
\end{proof}

Geometric homogeneity follows by an identical argument, except we use the second statement of \ref{freeab} and there is no need to use Sah's lemma. 

\begin{rem}
In the setting of this chapter, the union of Conditions \ref{AH} and \ref{GH} are equivalent to a weak version of the main algebraic lemma of \cite{zilber2006covers} (sometimes called the `thumbtack lemma'). The weakening is in that we only consider Galois representations over one algebraically closed field, and not over independent systems of them. However, along with the main result of \cite{bays2012quasiminimal} (which implies that just considering representations over one algebraically closed field is enough),
and the elementary fact that
$$\Gal(\Q(\mu) / \Q) \cong \hat{\Z}^{\times},$$
Theorems \ref{thumb0} and \ref{thumb1} below imply the main categoricity result (Theorem 1) of \cite{zilber2006covers}. 
\end{rem}

\section{Directions for extending this work}

In the above we started with a very general setting, with a covering sort with no structure on it at all, and then restricted the number of models through an appropriate $\cL_{\omega_1,\omega}$-sentence. Hopefully this can be done in some other situations:

\subsection{Elliptic curves}

The case of complex elliptic curves without complex multiplication is very similar to the case of the multiplicative group, since it is also a 1-dimensional algebraic group, and again, we end up looking at Galois representations in the Tate module. Here, the finite \'etale covers of elliptic curves are well know to be the multiplication by $N$ maps from $E$ to itself
$$[N] : E \ra E.$$
The \'etale covering automorphisms are given by addition by an $N$-torsion point, so we have
$$\pi_1^{et}(E) \cong \varprojlim_{N} \Aut(E[N]) \cong \varprojlim_N \GL_2 (\Z / N \Z) \cong \GL_2 (\hat{\Z})$$
and the etale fundamental group is just the Tate module of the elliptic curve. So if $E$ is a complex elliptic curve, then we may view $E$ as the subset of $\C^2$ defined by
$$y^2 = x^3 + Ax + B$$
in the structure $\C_K$, where $K = \Q(A,B, \tor(E))$ (and $\tor(E)$ is the set of coordinates of all torsion points for all $N$).\newline

If the elliptic curve has complex multiplication by an imaginary quadratic field $K$, then as Misha Gavrilovich notes in his thesis, the information of the Mumford-Tate group of $E$ must somehow be included. That is, the Galois representations on the Tate-module are not open in $\GL_2(\hat{\Z})$, but in the $\hat{\Z}$-points of the Mumford-Tate group i.e. $\hat{\cO}_K^{\times}$. However it should also be possible to encode this information as an $\cL_{\Cov(\C^{\times}),\omega_1,\omega}$-sentence.

\subsection{The $j$-function}

One of my initial reasons for approaching the model theory of covering spaces in this way, was to construct a setting for the $j$-function where Galois representations on Tate modules of CM elliptic curves may be seen. In the setting of Chapter \ref{chapj}, this is not the case since the action of each element of $\gl2q$ is definable on the covering sort, so the special points are definable. It was clear in the previous chapter that few analytic properties of the $j$-function were used, and it was just the behaviour of $j$ as a covering map that was at the heart of things (outside of the special points). This points towards considering the complex affine algebraic curve $$V:= \P_1(\C) \backslash \{ 0 , 1728 , \infty \}.$$
The analytic fundamental group $\pi_1^{an}(V)$ is equal to the free group on two generators, so the geometric \'etale fundamental group is the free profinite group on two generators.
However, we take a subsystem of $\Feg(V)$ corresponding to the action of $\gl2q$ on $\H$ (as in Chapter \ref{chapj}).\newline

Consider the (unramified) $j$-function
$$j': \H \backslash \{i , e^{2\pi i / 3} \} \ra V,$$
i.e. the restriction of $j$ to its unramified points. We consider the subcategory of $\Feg(V)$ given by Galois covers $f: Z_g \ra V$ arising from tuples $g=(g_1,...,g_n) \subset \gl2q$ and restricting the domain of the cover to the unramified points. Denote this restriction by $f: Z_g' \ra V$. 

We define the resticted pro-\'etale cover  to be the projective limit
$$\hat{\U}' := \varprojlim_{g} Z_g'$$
and the restricted \'etale fundamental group
$$\pi_1^{et'}(V):= \varprojlim_{g} \Aut_{\Feg(V)} (Z_g' / V) \cong \psl2(\hat{\Z}).$$
Similarly, define the restricted analytic fundamental group to be $$\pi_1^{an'}(V) : = \psl2 ( {\Z})=: \Gamma.$$
Note that $\Gamma$ acts freely on $ \H \backslash \{i , e^{2\pi i / 3} \}$. Also note that all of the general theory of model theory of covering spaces in this chapter applies in the situation of $j'$ with this restricted system. For example we still have an embedding of the restricted analytic fundamental group $\psl2(\Z)$ into the restricted \'etale fundamental group $\psl2(\hat{\Z})$. \newline

The question now is whether an $\cL_{\omega_1,\omega}$ axiomatisation similar to that of \S \ref{additive} for the additive group can be found for (something like) $\gl2q$ acting on the upper half plane, but I am hopeful that this is the case. If this is the case, then the results regarding the openness of Galois representations in the Mumford-Tate groups of products of arbitrary non-isogenous elliptic curves over number fields are known \cite{hindry2008torsion}, so if the information of the Mumford-Tate groups of the products can also be encoded in the infinitary theory, then it should be categorical.

\subsection{The modular $\lambda$-function}
Similarly to $j$, we could consider the $\lambda$-function. Consider the punctured plane $\P^1(\C) \backslash \{ 0,1,\infty \}$. Since the corresponding Riemann surface has negative curvature, it is universally covered by $\H$, and the analytic fundamental group is easily seen to be the free group on two generators $F_2$. Consider the exact sequence
$$\Gamma(2) \hookrightarrow \Gamma \ra \psl2(\Z / 2 \Z)$$
induced by reduction mod $2$. Then
$$\Gamma(2) = \left\{ \begin{pmatrix} a&b\\ c&d \end{pmatrix} \in \Gamma \ \ , \ \ b \equiv c \equiv 0 , \ a \equiv d \equiv 1 \mod 2 \right\},$$
is of index 6 in the modular group $\Gamma$, and isomorphic to $F_2$. The $\lambda$-function is invariant under $\Gamma(2)$, and gives a complex analytic isomorphism
$$\lambda : \Gamma(2) \backslash \H \cong \C \backslash \{0,1 \},$$
i.e. it is the universal covering map. As a result, the $\lambda$-function is in many ways a more natural one to consider than $j$ in this setting. 

\begin{rem}
It would be interesting to know the relationship between the theories of $\lambda$-function and $j$-function in this situation, and the correspondence between categoricity and Grothendieck's anabelian program. For instance, does categoricity imply the injectivity statement of the section conjecture (which is known)?
\end{rem}

\appendix


\chapter{Algebro-geometric background}\label{app etale}

\section{Covering spaces}\label{seccov}

This section is a fairly direct summary of chapter $2$ of \cite{szamuely2009galois}. We assume all topological spaces to be locally simply connected (i.e. each point has a basis of neighbourhoods consisting of connected open subsets).  Let  $\Cov(X)$ be  the category of covers of $X$. 

\begin{dfn}
Let $X$ be a topological space. Then a \textit{covering space} (or \textit{cover}) of $X$ is a topological space $Y$, and a continuous map 
$$p:Y \ra X$$
such that each point of $X$ has an open neighbourhood $U$ such that $p^{-1}(U)$ decomposes into a disjoint union of open subsets $U_i$ of $Y$, and $p$ is a homeomorphism of $U_i$ onto $U$. Given two covers $p_i : Y_i \ra X$, $i \in \{1,2 \}$, a morphism of covers (of $X$) is a continuous map $f: Y_1 \ra Y_2$ such that the diagram
$$
\xymatrix{
Y_1  \ar@{->}[dr]_{p_1} \ar@{->}[rr]^{f} & & Y_2 \ar@{->}[dl]^{p_2} \\
& X & 
}
$$
commutes. We denote the category of covers of $X$ by $\Cov(X)$.
\end{dfn}

An immediate consequence of the definition is that covering maps are always surjective. \newline

Take a non-empty discrete topological space $I$ and form the topological product $X \times I$. The projection onto the first coordinate $$p_1 : X \times I \ra X$$ turns $X \times I$ into a cover of $X$ called a \textit{trivial cover}. It turns out that every cover is locally a trivial cover \cite[2.1.3]{szamuely2009galois} (given an open set $U \subseteq X$, if $p^{-1}(U)$ is a disjoint union of open subsets of $Y$, homeomorphic to $U$, and indexed by a set $I$, then the discrete space is $I$),  and the points of $X$ over which the fibre of $p$ equals a particular discrete space $I$, form an open subset of $X$. From this we deduce:

\begin{pro}
 If $p:Y \ra X$ is a cover, and X is connected, then the fibres of $p$ are all homeomorphic to the same discrete space.
\end{pro}

\begin{dfn}\label{defdeg}
A cover $p:Y \ra X$ is called finite if it has finite fibres. For connected $X$ all fibres have the same cardinality, called the \textit{degree} of the cover.
\end{dfn}

\subsection{Quotients as covers}

\begin{dfn}
Let $G$ be a group acting continuously from the left on a topological space $Y$. The action of $G$ is \textit{even} if each point $y \in Y$ has a neighbourhood $U$ such that the open sets $gU$ are pairwise disjoint for all $g \in G$.
\end{dfn}

Looking at the definitions, we see that even group actions give covering spaces:
\begin{pro}
If G is a group acting evenly on a connected space $Y$ , the projection $p : Y \ra G \backslash Y$ turns $Y$ into a cover of $G \backslash Y$.
\end{pro}

Now if $p:Y\ra X$ is a connected cover, then around $x \in X$ there is a neighbourhood $U$ such that $p^{-1}(U)$ is a disjoint union of open subsets of $Y$ isomorphic to $U$. Automorphisms of the cover have  to permute these open sets, but an automorphism of a connected cover having a fixed point must be trivial (\cite[2.2.1]{szamuely2009galois}), so we get the following:

\begin{pro} If $p : Y \ra X$ is a connected cover, then the action of $ \Aut_{\Cov(X)}(Y/X)$ on $Y$ is even.
\end{pro}

Conversely, if a group $G$ acts evenly on a connected space $Y$, then elements of $g$ are automorphisms of the cover $p : Y \ra G \backslash Y$. Since fibres of the quotient map are $G$-orbits, and covering automorphisms act freely on fibres, these are all of the covering automorphisms:
\begin{pro} If G is a group acting evenly on a connected space $Y$, the automorphism group of the cover $p : Y \ra G \backslash Y$ is precisely $G$. 
\end{pro}

So we have seen that automorphism groups of connected covering spaces are even group actions, and even group actions give rise to connected covering spaces.

\begin{exa}
Consider $\C^{\times}:=\C \backslash \{0 \}$ with the complex analytic topology. Multiplication by the $N^{th}$ roots of unity $\mu_N$ defines an even action on $\C^{\times}$ and we get a cover
$$\C^{\times} \lora \C^{\times} / \mu_N \overset{z \mapsto z^n}{\lora} \C^{\times}$$
where the first map is the quotient map, and the last map is a homeomorphism.
\end{exa}

\begin{dfn}
A connected cover $p:Y \ra X$ is said to be \textit{Galois} if  \linebreak $\Aut_{\Cov(X)}(Y /X)$ acts transitively on fibres.
\end{dfn}


So given a Galois covering $p:Y\ra X$ and $x \in X$, the fibre $p^{-1}(x)$ is an $\Aut_{\Cov(X)}(Y /X)$-torsor (see \ref{app torsors} below for a quick reminder on torsors). We have now arrived at the main theorem of this subsection. It is a Galois correspondence for covering spaces, and is analogous to the Galois correspondence for fields.

\begin{thm}[The Galois correspondence] Let $p : Y \ra X$ be a Galois cover. For each subgroup $H$ of $G =  \Aut_{\Cov(X)}(Y/X)$, the projection $p$ induces a natural map $p_H : H \backslash Y \ra X$ which turns $H\backslash Y$ into a cover of $X$.
Conversely, if $q:Z \ra X$ is a connected cover fitting into a commutative diagram 
$$
\xymatrix{
Y \ar@{->}[dr]^f \ar@{->}[dd]^p &  \\
&  Z \ar@{->}[dl]^q \\
 X &  
}
$$
then $f:Y \ra Z$ is a Galois cover and $Z \cong  H \backslash Y$ for the subgroup $H = Aut_{\Cov(X)}(Y / Z) $ of $G$. The maps 
\begin{align*}
H & \mapsto H \backslash Y \\
Z & \mapsto \Aut_{\Cov(X)}(Y / Z)
\end{align*}
induce a bijection between subgroups of $G$ and intermediate covers $Z$ as above. The cover $q : Z \ra X$ is Galois if and only if $H$ is a normal subgroup of $G$, in which case $\Aut_{\Cov(X)}(Z / X) \cong G/H$.
\end{thm}

\subsection{Torsors}\label{app torsors}

Given a group $G$, a set $X$ is said to be a $G$\textit{-torsor} if $G$ acts freely and transitively on $X$. This is equivalent to saying that for every $x,y \in X$ there is a unique $g \in G$ such that $g(x)=y$. Similar to an affine space being a vector space which has forgotten where 0 is, a $G$-torsor is $G$ but after forgetting where the identity is. So to recover $G$ from a $G$-torsor $X$, you just have to choose $x_0 \in X$ to be the identity of $G$. After choosing $x_0 \in X$, you get an isomorphism
$$f: X \cong G$$
where $x \in X$ is sent to the unique $g \in G$ sending $x_0$ to $x$. So every $G$-torsor is isomorphic (as a set) to $G$, however the important thing is that the isomorphism is non-canonical since it is dependent on the choice of identity $x_0$.\newline

\subsection{The fibre functor and the universal cover}

For the purpose of understanding this thesis, the reader need not know the definitions of  path, or homotopy equivalence. The notion of path is used in the construction of the universal cover, and in the definition of the fundamental group. However, if the reader is willing to take for granted the existence of an object representing the fibre functor, then this is actually more in the spirit of this thesis, since in the model-theoretic discussions of the previous chapters we are always forgetting about the complex topology (and therefore about paths), and trying to recover the analytic universal cover from algebraic information alone. However, for completeness we give a brief definition of the fundamental group. 
\begin{dfn}
Given a topological space, $X$ as above, define the \textit{fundamental group} of $X$, with base point $x$, denoted $\pi_1(X;x)$ to be the set of all homotopy classes of paths from $x$ to $x$ (the group operation is given by composition of paths).
\end{dfn}




Consider the \textit{fibre functor}
\begin{align*}
\Fib_x: \Cov(X) & \lora \Sets \\
(p:Y \ra X) & \longmapsto p^{-1}(x).
\end{align*}


\begin{thm}\cite[2.3.4]{szamuely2009galois}
Let $X$ be  connected, locally simply connected topological space, and $x \in X$. The functor $\Fib_x$ induces an equivalence of categories of $\Cov(X)$, with the category of left $\pi_1(X,x)$-sets. Connected covers correspond to $\pi_1(X,x)$-sets with transitive action, and Galois covers to coset spaces of normal subgroups.
\end{thm}
So $\pi_1(X,x)$ classifies the covering spaces of $X$. Crucial in proving the above we have:

\begin{thm}\cite[2.3.5]{szamuely2009galois}\label{fibrep}
Let $X$ connected, locally simply connected topological space, and $x \in X$. Then the functor $\Fib_x$ is representable by a cover $\tilde{p}: \Xx \ra X$, i.e.
$$\Fib_x \cong \Hom_{\Cov(X)}(\Xx,-).$$
\end{thm}
The representing object is called the \textit{universal cover} (it is unique up to isomorphism in $\Cov(X)$ by its universality), and it classifies the covering spaces of $X$.

\begin{rem} To construct the universal cover $\Xx$, we take the space of homotopy classes of all paths in $X$
$$\Xx := \cup_{y \in X} \pi_1  (X ; x,y),$$
however, as was mentioned earlier, the reader may just take its existence for granted.
The universal cover has the property that
$$\Aut_{\Cov(X)}(\Xx / X) \cong \pi_1(X,x),$$
and since we are forgetting about the notion of path, we will identify the two groups.
\end{rem}

\subsection{The action of $\Aut_{\Cov(X)}(\Xx)$ on fibres}\label{sec indact}

 Now let $\gamma \in \Aut_{\Cov(X)}(\Xx)$ and $(p:Y\ra X)$ a cover. Then $\sigma$ induces a bijection
\begin{align*}
\sigma: \Hom_{\Cov(X)}(\Xx,Y) & \lora \Hom_{\Cov(X)}(\Xx,Y)  \\
f & \longmapsto  f \circ \gamma
\end{align*}
so that the left action of $ \Aut_{\Cov(X)}(\Xx)$ on $\Xx$ induces a right action on the fibre $p^{-1}(x)=\Fib_x(Y) \cong \Hom_{\Cov(X)}(\Xx,Y)$.




\section{Algebraic geometry (for model-theorists)}
In this section is a summary of some background algebraic geometry, mostly aimed at the model-theorist. 
Throughout this thesis, the main focus will be on smooth algebraic curves over an algebraically closed field of characteristic 0, so this is the main focus of the start of this section. The main objective here is to give a simple definition of an \'etale morphism for a smooth complex curve in \S \ref{appetalemor}. \newline

For a variety $V$ over a field $k$ (not necessarily algebraically closed), it is common for model-theorists to consider the $\bar{k}$-points of the variety, as embedded in the `$k$-enriched algebraically closed field $\bar{k}$'. This basically means that we consider $V$ as a $k$-variety, and as a $\bar{k}$-variety at the same time. This can be confusing for geometers, so the purpose of sections \ref{kvar} and \ref{natlang} is to make precise this distinction, and to describe how to pass between the two points of view.

\subsection{Riemann surfaces}\label{secrie}

Riemann surfaces are just Hausdorff topological spaces that look locally like the complex numbers. It is an easy fact that a complex projective algebraic curve can be given the structure of an algebraic curve, however, compact Riemann surfaces can be given the structure of a complex algebraic curve, and this allows us to pass between the algebraic and analytic categories.  However, we do not just want study the category of compact Riemann surfaces, since we want to look at the analytic universal cover. For this reason we look at the category of Riemann surfaces in general, which does not really complicate our discussions much at all here. For a more detailed discussion of Riemann surfaces in general see \cite{miranda1995algebraic}. In this section we will assume that all maps are non-constant on connected components i.e. they do not map a whole component to a point. \newline

At the heart of all discussions of this section is the following local description of holomorphic maps between Riemann surfaces:

\begin{pro}\cite[II, 4.1]{miranda1995algebraic}\label{local}
Let $p:Y \ra X$ be a holomorphic map of Riemann surfaces, let $y \in Y$ and $x=p(y) \in X$. Then there exist open neighbourhoods $V_y$ of $y$, and $U_x$ of $x$, with $p(V_y) \subseteq U_x$, and complex charts $g_y : V_y \ra \C$ and $f_x : U_x \ra \C$ such that $f_x(x) = g_y(y)=0$ such that there is a commutative diagram
$$
\xymatrix{
V_y \ar@{->}[r]^f \ar@{->}[d]^{g_y} & U_x \ar@{->}[d]^{f_x}  \\
\C  \ar@{->}[r]^{z \mapsto z^{e_y}}  &   \C
}
$$
with an appropriate integer $e_y$, which does not depend on the choice of charts.
\end{pro}

\begin{dfn}
The integer $e_y$ above is called the \textit{ramification index} or \textit{multiplicity} of $p$ at $y$. The points $y \in Y$ with $e_y > 1$ are called \textit{ramification points} of $p$, and images of ramification points are called \textit{branch points} of $p$. The set of ramification points of $Y$ is called the \textit{ramification locus} of $p$, and the set of branch points of $X$ is called the \textit{branch locus}.
\end{dfn}

The above description (\ref{local}) of holomorphic maps between Riemann surfaces tells us that locally these maps look like power maps. In particular, they are open. Almost immediately we also have the following:

\begin{pro}\cite[3.2.4]{szamuely2009galois}
Let $p:Y \ra X$ be a holomorphic map of Riemann surfaces. Then the fibres of $p$, and the ramification locus of $p$ are discrete, closed subsets of $Y$.
\end{pro}
So if $Y$ is compact then the ramification locus and fibres are finite. 
\begin{dfn}Recall that a continuous map between locally compact topological spaces is said to be \textit{proper} if the preimages of compact sets are compact.
\end{dfn}
Proper maps between Riemann surfaces have finite fibres. Note also that a proper map between Hausdorff spaces is closed, and that a continuous map $p:Y \ra X$ between Hausdorff spaces is automatically proper if $Y$ is compact. \newline

Since proper holomorphic maps are open and closed we have the following:
\begin{pro}
Let $p:Y \ra X$ be a proper holomorphic map of Riemann surfaces, with $X$ connected. If $p$ is proper, then $p$ is surjective.
\end{pro}


\begin{dfn}
A surjective map of locally compact Hausdorff spaces which restricts to a covering map outside of a discrete closed subset is called a \textit{branched cover}. If it restricts to a finite cover, then it is called a \textit{finite branched cover}.
\end{dfn}

We can package up all of the above discussion into the following:

\begin{thm}
Let $X$ be a connected Riemann surface, and $p:Y \ra X$ a holomorphic map. Then $p$ is a branched covering map of its image. If $p$ is proper, then it is a finite branched covering map, and above a branch point $x$, the sum of the ramification indices $e_y$ of the points $y$ in the fibre $p^{-1}(x)$ is equal to the degree of the (restricted) cover.
\end{thm}
So in particular, holomorphic maps between compact Riemann surfaces are branched covering maps.

\begin{dfn}
A branched cover $p: Y \ra X$ is called \textit{Galois} if $\Aut_{\Cov(X)}(Y / X)$ acts transitively on the fibres outside of the ramification locus.
\end{dfn}

In fact, by continuity of the covering automorphisms we have:

\begin{pro}
If $p:Y \ra X$ is a Galois branched cover then $\Aut_{\Cov(X)}(Y / X)$ acts transitively on all fibres.
\end{pro}

By the above discussion it is also easy to see that the following holds:

\begin{pro}
Let $p:Y \ra X$ be a proper holomorphic map of connected Riemann surfaces, such that $p$ is a Galois branched cover. Then if $y \in p^{-1}(x)$ is a branch point with ramification index $e$, then so are all the points in the fibre $p^{-1}(x)$. The stabilisers of points in the fibre $p^{-1}(x)$ in $\Aut_{\Cov(X)}(Y / X)$ are conjugate cyclic subgroups of order $e$.
\end{pro}

We end this section with some topological facts which are useful when thinking about holomorphic maps between Riemann surfaces:
\begin{itemize}
\item A map is proper iff it is closed, and fibres of points are compact;
\item A finite covering map is an open map which is also a local homeomorphism. It is also closed and therefore proper;
\item A continuous map is a finite branched covering map iff it is both open and closed;
\item Let $p$ be a continuous map between locally compact Hausdorff spaces. Then $p$ is a proper local homeomorphism iff $p$ is a (finite) covering map.
\end{itemize}

\subsection{Complex algebraic curves}

In this section, all Riemann surfaces are assumed to be connected, and all varieties irreducible.

\subsubsection{The analytification functor}

Given a complex algebraic variety $X$, we can consider $X$ (or, in particular its $\C$-points $X(\C)$, see \ref{kvar}) with the complex topology (i.e. as an analytic set in some complex space). We denote this complex analytic set by $X(\C)^{an}$. Similarly for a map of complex algebraic varieties $f: Y \ra X$, we denote the corresponding map of analytic sets by $f^{an} : Y(\C)^{an} \ra X(C)^{an}$. 

\subsection{The Riemann existence theorem}

The following well known theorem provides an important link between the algebraic and complex analytic worlds. Clearly the interesting direction of the theorem is that every compact Riemann surface can be given a unique structure as an algebraic variety (i.e. can be seen as being cut out by polynomials).
\begin{thm}[The Riemann existence theorem]\label{thmrie}
The category of smooth projective curves over $\C$ with rational maps as morphisms, is equivalent to that of compact Riemann surfaces with holomorphic maps.
\end{thm}
Note that for a smooth curve, every rational map is regular. We also have the following:

\begin{thm}\cite[I, 4.4]{hartshorne1977algebraic}
The category of compact Riemann surfaces with holomorphic maps, is anti-equivalent to the category of transcendence degree one field extensions of $\C$.
\end{thm}
So we can pass between smooth complex projective curves, compact Riemann surfaces, and field extensions, which is extremely useful. We now want to relate the above discussion to affine curves. The correspondence is very simple:\newline

If $X$ is a smooth complex affine curve, then (being irreducible by definition) its coordinate ring $\cO(X)$ is an integral domain, which embeds uniquely into its fraction field. So we see that an affine algebraic curve sits inside a unique complex projective curve as an open set. Furthermore, if $f: Y \ra X$ is a map of smooth complex affine curves, then this induces an injection of coordinate rings
$$f^* : \cO(X) \hookrightarrow \cO(Y).$$
By the universal property of the field of fractions, this extends uniquely to an injection of function fields
$$f^*: \C(X) \hookrightarrow \C(Y)$$
and therefore extends uniquely to the corresponding complex projective curves. From now on, we will use the implications of the above discussion freely.

\subsection{Proper morphisms and finite morphisms}

In \S \ref{secrie}, we saw that proper maps had some good properties. We would like to find a similar notion for algebraic varieties, however the Zariski topology is good for some things, and bad for others. Since affine varieties are Noetherian in the Zariski topology and a topological space is Noetherian if and only if every subspace is compact, it is clearly no good to try and carry the definition of a proper map over word for word from general topology. There is a general algebro-geometric definition of a proper map, which corresponds to the general topological notion when varieties are viewed with their complex topology. However, we do not give this definition here, because we are interested in smooth curves, where properness is equivalent to the notion of finite. We focus on the notion of finite, because we are aiming to give a definition of an \'etale cover, and finiteness is the correct notion for generalisation to arbitrary varieties and schemes.

\begin{dfn}
A morphism $p:Y \ra X$ of affine curves is said to be \textit{finite} if the induced map on coordinate rings
$$f^*: \cO(X) \hookrightarrow \cO(Y)$$
gives $\cO(Y)$ the structure of a finitely generated $\cO(X)$-module. A morphism of projective curves is said to be finite if for every affine open set $U \subset X$, the inverse image $f^{-1}(U)$ is affine, and the induced map
 $$f^*(\cO(U)) \hookrightarrow \cO(f^{-1}(U))$$
gives $\cO(f^{-1}(U))$ the structure of a finitely generated $\cO(U)$-module.
\end{dfn}

For example, the inclusion of algebraic curves
$$\A^1(\C)\backslash \{0\} \hookrightarrow \A^1(\C),$$
corresponding to the inclusion of rings
$$\C[X] \hookrightarrow \C[X,X^{-1}]$$
is not finite. Finite morphisms (between any varieties) have finite fibres, and are proper (by the going up theorem of algebra), and therefore closed. Furthermore, it is well known that a map between varieties is finite if and only if it is proper, and has finite fibres (see \cite{ega1964ements} for example). In the case of smooth complex curves we have:

\begin{pro}\cite[4.4.7]{szamuely2009galois}
A surjective morphism of smooth, projective complex curves is finite.
\end{pro}
As a corollary, a surjective morphism of smooth complex affine curves is also finite.

\subsubsection{An algebraic curve minus a point is algebraic}

\begin{thm}
Given a smooth complex projective curve $X$, and $\{ x_1,...,x_n \} \subset X$,  $X \backslash \{x_1,...,x_n \}$ is an affine curve.
\end{thm}

\begin{proof}
By the Riemann-Roch theorem, we can construct a rational function
$$f: X(\C) \lora \P^1(\C) = \A^1(\C) \cup \infty$$
with poles exactly at the $x_i$. So $f$ is a non-constant holomorphic map between compact Riemann surfaces and is therefore surjective. The induced morphism on complex projective curves is therefore finite, so $f^{-1}(\A^1(\C)) = X \backslash \{x_1,...,x_n \}$ is an affine curve.
\end{proof}

By embedding a smooth complex affine curve as an open set in a projective curve, we get:

\begin{cor}
Given a smooth complex affine curve $X$, and $\{ x_1,...,x_n \} \subset X$,  $X \backslash \{x_1,...,x_n \}$ is an affine curve.
\end{cor}

\subsection{\'Etale morphisms}\label{appetalemor}

\begin{dfn}
A map $p: Y\ra X$ of smooth complex curves is said to be \textit{\'etale} if it is covering map in the complex toplology. Let $\Feg(X)$ be the category of finite, \'etale, Galois covers of $X$.
\end{dfn}

From our previous discussion, it is easy to see that we have:
\begin{pro}
A map $p: Y\ra X$ of smooth complex curves is \'etale iff it is surjective and unramified. A surjective map $p:Y \ra X$ of smooth complex algebraic curves is a finite Galois, \'etale cover iff the fibres of $p$ are  $\Aut_{\Feg(X)} (Y / X)$-torsors.
\end{pro}

\begin{rem}
Clearly covering maps are local homeomorphisms, however \'etale morphisms are not local homeomorphisms in the Zariski topology. To see this, consider the finite \'etale covering 
\begin{align*}
p: \A^1(\C)\backslash \{ 0\} & \ra \A^1(\C) \backslash \{0 \} \\
x & \mapsto x^n.
\end{align*}
Then if this were to be a local homeomorphism between open sets $U$ and $U'$, then it would induce an isomorphism of coordinate rings $\cO(U)$ and $\cO(U')$, and this would induce an automorphism of the function field
\begin{align*}
p^*: K(X) & \ra K(X) \\
X & \mapsto X^n.
\end{align*}
But clearly this is not an automorphism.
\end{rem}




\subsection{Affine $k$-varieties}\label{kvar}
I tried to avoid mentioning schemes at all, since model theorists generally tend to shy away from them, but in this section a scheme theoretic approach is taken because it is a neater way of doing things, and the scheme theoretic notions introduced are in direct correspondence with the model theoretic notions of the next section \ref{natlang}. For the model-theorist who knows `Weil style' algebraic geometry over an algebraically closed field, this section is supposed to provide an easy route through to scheme theory. For example, for a smooth, irreducible affine curve over $\C$, the points of the curve may be identified with the prime ideals of the coordinate ring, in a functorial manner. For the definition of the $\Spec$-functor, see Ravi Vakil's online algebraic geometry notes. Let $k$ be a field (not-necessarily algebraically closed).
\begin{dfn}
An \textit{affine $k$-variety} is an affine variety $X=\Spec(A)$, such that there is morphism of finite type in the category of schemes (called the structure morphism)
$$X \lora \Spec(k),$$
i.e. there is a ring homomorphism
$$k \hookrightarrow A$$
giving $A$ the structure of a finitely generated $k$-algebra. Let $X$ and $Y$ be $k$-varieties with coordinate rings $A$ and $B$ respectively. Then a morphism of $k$-varieties $f:X \ra Y$ is a regular map $f$ such that the induced map on the coordinate rings commutes with the structure morphisms i.e. there is a commutative diagram
$$
\xymatrix{
X \ar@{->}[rr]^{f} \ar@{->}[dr] & & Y \ar@{->}[dl]  \\
& \Spec(k) &
}
$$
i.e. we have a commutative diagram
$$
\xymatrix{
B \ar@{->}[rr]^{f_*}  & & A   \\
& \ar@{->}[ul] k \ar@{->}[ur]&
}
$$
We denote the category of $k$-varieties by $k$-var. If $X$ is a $k$-variety and $L / k$ is a field extension, we define the \textit{$L$-rational points} of $X$ (denoted $X(L)$) to be elements of $\Hom_{\kvar}(\Spec(L),X)$. If $k$ is a field and $\bar{k}$ is a fixed algebraic closure of $k$ then a $\bar{k}$-valued point is called a \textit{geometric point}. For a $k$-variety $X$ and a point $x \in X$, we define $k(X)$ to be the function field of $X$, and $\kappa(x)$ to be the residue field of $x$.
\end{dfn}
\begin{rem} The $L$-rational points of $X$ are in bijection with $\Hom_{\kalg}(A,L)$. Giving an $L$-rational point is equivalent to giving a point $x \in X$ and an embedding $\kappa(x) \hookrightarrow L$. 
\end{rem}
We can unwind the above definitions into something more familiar to model-theorists as follows: \newline

\begin{rem} The coordinate ring of a $k$-variety is isomorphic to $k[X_1,...,X_n] / I$ for a finitely generated ideal $I = (f_1,...,f_r)$ of $k[X_1,...,X_n]$.
\end{rem}

\begin{pro}
Let $X = \Spec(k[X_1,...,X_n] / (f_1,...,f_r))$ be a $k$-variety and $L / k$ a field extension. Then $X(L)$ can be identified with the set of elements of the vanishing set of $\{f_1,...,f_r \}$ in $\A^n(L)$, or similarly as the set defined by the equations $f_i(X_1,...,X_n)=0$ in the structure $\langle L , + ,  \cdot , \{c_i \}_{i \in k} \rangle$ (see \ref{natlang}).
\end{pro}

\begin{proof}
Fix an embedding $\i :k \hookrightarrow L$ (it is important to note that an $L$-valued point includes the data of this embedding). Let $X=\Spec(A)$ with $$A  = k[x_1,...,x_n] \cong k[X_1,...,X_n] / (f_1,...,f_r).$$ Then given a ring homomorphism
$$\bar{g}: A \ra L,$$
for a positive formula $\varphi$ in the language of rings, we have
$$A \models \varphi(x_1,...,x_n) \Ra L \models \varphi(\bar{g}(x_1),...,\bar{g}(x_n))$$
so that if $\bar{g}(x_j)=l_j \in L$, then 
$$L \models \bigwedge_{i=1}^r f_i(l_1,...,l_n) = 0. $$

 Conversely, by the universal property of the polynomial ring, any map of sets
$$g:\{X_1,...,X_n\} \ra L$$
taking $(X_1,...,X_n)$ to $(l_1,...,l_n) \in L^n$, extends uniquely to a $k$-algebra homomorphism
$$g:k[X_1,...,X_n] \ra L$$
preserving the embedding $\i$. So if $f_i(l_1,...,l_n)=0$, then $g$ will pass to a map $\bar{g}$ on the quotient $k[X_1,...,X_n] /  (f_1,...,f_r)$.
\end{proof}

\begin{rem}
For the model-theorist, the residue field $\kappa(x)$ is just the minimal field of definition of the coordinates of $x$ under the above embedding of $X(\bar{k})$ as a definable set in $\langle \bar{k} , + ,  \cdot , \{c_i \}_{i \in k} \rangle$.
\end{rem}

Given a $k$-variety $X$, and a field extension $L / k$, we denote by $X \times_k L :=\Spec(A \otimes_k L)$, the \textit{base change} of $X$ to $L$. Note that if $X \cong \Spec(k[X_1,...,X_n] / I)$ then $X \times_k L \cong \Spec(L[X_1,...,X_n] / I)$, and $X \times_k L $ is an $L$-variety, but not necessarily a $k$-variety.  Intuitively, here we have just forgotten that $X$ was defined by polynomials in $k$, and we view the polynomials as being over $L$. For a fixed algebraic closure $\bar{k} / k$, we will also use the notation $$\bar{X}:=X \times_k \bar{k}.$$

\subsection{A natural structure for $k$-varieties}\label{natlang}

Let $X$ be a $k$-variety. Given an embedding of fields $\i: k \hookrightarrow L$, consider the structure
$$L_k := \langle L , + , \cdot , \{ c_i\}_{i \in k} \rangle$$
i.e. we have a field $L$ in the ring language, enriched with a constant symbol for every element of $k$ (interpreted as in the embedding $\i$). Then we can identify the $L$-rational points of $X := \Spec(k[X_1,...,X_n] / (f_1,...,f_r))$ as the definable set cut out by the formula
$$\bigwedge_{i=1}^{r} f_i(X_1,...,X_n)=0$$
in $\langle L , + , \cdot , \{ c_i\}_{i \in k} \rangle$. We denote the language for structures of this form by $\cL_{\Ring,k}$.

\begin{rem}
Similarly to the way geometers consider the functor of points of a $k$-variety, we can consider a $k$-variety $X:=\Spec(k[X_1,...,X_n] / (f_1,...,f_r))=\Spec(A)$ as a functor 
$$\varphi_X : \cL_{\Ring,k}\textrm{-structures} \lora \Sets,$$
determined by the formula $$\varphi_X(X_1,...,X_n):=\wedge_{i=1}^r f_i(X_1,...,X_n)=0,$$ sending a structure $\cM:=\langle L , + , \cdot , \{ c_i\}_{i \in k} \rangle$, to the set defined by the formula $\varphi_X$ in $\cM$ which we denote by $\varphi_X(\cM)$. It is a functor because given a homomorphism $f:L_1 \ra L_2$ of $\cL_{\Ring,k}$-structures, and an $L_1$-rational point
$$\bar{g}: A \lora L_1$$
we get an $L_2$-rational point via the composition $f \circ \bar{g}$, so we obtain a map of sets from $\varphi(L_1)$ to $\varphi(L_2)$.
\end{rem}

\subsubsection{Types in algebraically closed fields}

Clearly the easiest case to consider with respect to the model theory, is where we take $k$-variety $X$, and consider $X(\bar{k})$ as a definable subset of 
$$\langle \bar{k} , + , \cdot , \{ c_i\}_{i \in k} \rangle.$$
This is because if $\bar{k}$ is an algebraically closed field of characteristic 0, then the theory of this structure has just about every nice model-theoretic property that you can think of. In particular it has quantifier elimination, and a unique model in every uncountable cardinality. For countable algebraically closed fields of characteristic 0, with transcendence degree $\kappa \leq \aleph_0$, the $\cL_{\omega_1, \omega}$ sentence (see \ref{app quasi}) fixing the transcendence degree gives a categorical theory. \newline

 In particular, we have the following, well known, simple description of the types:
\begin{thm}
Let $L$ be an algebraically closed field of characteristic 0, $k \subseteq K \subseteq L$ and $x \subset L^n$. Then $\tp(x / K)$ in the structure $L_k$, is determined by the formulae
$$x \in W \cup \{ x \notin W' \ | \ W' \subset W, \ \dim W'< \dim W\}$$
where $W$ is the minimal algebraic variety over $K$ containing $x$, and $W'$ are subvarieties of $W$ over $K$.
\end{thm}

\subsection{The weak Lefschetz principle}

In the previous sections, we focused on complex curves, but we would like to study curves over arbitrary algebracially closed fields of characteristic 0. Once we embed a variety in an algebraically closed field of characteristic 0 as a definable set, statements which are expressible as sentences in the language of rings apply to all algebraically closed fields of characteristic 0 because of the following basic fact:

\begin{thm}[The weak Lefschetz principle]\label{wlp}
The theory of algebraically closed fields of characteristic 0 is complete.
\end{thm}

I have called it the weak Lefschetz principle here, because apparently algebraic-geometers have stronger versions, but I am unfamiliar with these.

\subsection{The pro-\'etale cover}

Let $X$ be a $k$-curve (i.e. a 1-dimensional $k$-variety), fix an algebraic closure $k \hookrightarrow \bar{k}$, and let $\bar{X}:=X \times_k \bar{k}$. Let $\Fet(\bar{X})$ be the category of finite \'etale covers of $\bar{X}$, and $\Feg(\bar{X})$ the  category of finite Galois \'etale covers of $\bar{X}$. \newline

Fix a directed set $I$, indexing all finite \'etale covers of $\bar{X}$, and fix a directed system $(Y_i, \varphi_{j,i})$ where $\varphi_{j,i} : Y_j \ra Y_i$ for $i \leq j$. Define the \textit{pro-\'etale cover} of $\bar{X}$ to be
$$\hat{\bar{X}}:= \varprojlim_{I} Y_i.$$
and define the \textit{\'etale fundamental group} of $\bar{X}$ to be
$$\pi_1^{et}(\bar{X}):= \Aut(\hat{\bar{X}} / \bar{X}):= \varprojlim_i \Aut_{\Feg(\bar{X})}(Y / \bar{X}) ,$$
which comes equipped with its natural topology as a projective limit of finite discrete groups. 


\begin{rem}
Note that since we viewed $X$ as a $k$-curve and $k$ might not be algebraically closed, the \'etale fundamental group of $X$ (as in the general definition of \cite{szamuely2009galois} for example) may be much bigger than the \textit{`geometric etale fundamental group'} of $X$ (i.e. the \'etale fundamental group of $\bar{X}$) described here.
\end{rem}

Now we define another \textit{fibre functor} $\Fib^{et}_x$ from $\Fet(\bar{X})$ to $\Sets$
\begin{align*}
\Fib^{et}_x: \Fet(\bar{X}) & \lora \Sets \\
(p:Y \ra \bar{X}) & \longmapsto p^{-1}(x).
\end{align*}

This functor is pro-representable by a cover i.e. we have the analogue of \ref{fibrep}
\begin{thm}\cite[5.4.6]{szamuely2009galois}
Let $X$ be a $k$-curve, and $x \in X \times_k \bar{k} $. Then the functor $\Fib^{et}_x$ is pro-representable by a cover $\tilde{\bar{X}}_x$ i.e.
$$\Fib^{et}_x(Y) \cong \Hom(\hat{\bar{X}}_x , Y) := \varinjlim_{i}\Hom_{\Fet(\bar{X})}(Y_i, Y)$$
for every $Y\in \Fet(\bar{X})$.
\end{thm}

Note that the pro-representing object is not unique, since the maps in the system pro-representing $\Fib_x$ depend on a choice of basepoint lift. In particular, a point $(y_i)_I$ in the fibre above $x$, determines a unique system of maps between the system of covers i.e.
\begin{align*}
\varphi_{j,i} : Y_j & \ra Y_i \\
 y_{j} & \mapsto y_i.
\end{align*}

Finally we note that the geometric \'etale fundamental group does not change when considering extensions of algebraically closed fields.

\begin{pro}\cite[Remark 3.3]{milne1998lectures}
Let $K \subseteq L$ be an extension of algebraically closed fields of characteristic 0. Let $k \hookrightarrow K$, and $X$ be a $k$-curve. Then $\pi_1^{et}(X \times_k K) \cong \pi_1^{et}(X \times_{k} L)$.
\end{pro}
The above may be seen by embedding the finite \'etale covers in a suitable algebraically closed field of characteristic 0 (i.e. for a finite \'etale cover $p:Y \ra X \times_k K$ let $F$ be the minimal field of definition of the cover and use the structure $K_F$ of \ref{natlang}) and applying the weak Lefschetz principle \ref{wlp}, since finite \'etale morphisms of finite covers are definable.

\subsection{Galois action on the pro-\'etale cover}
Let $X$ be as above, let $K$ be a field containing all the fields of definition of finite \'etale covers $p:Y \ra \bar{X}$, and let $\hat{p}$ denote the projection map from the pro-\'etale cover to $\bar{X}$
$$\hat{p} : \hat{\bar{X}} \ra X.$$
If $x \in \bar{X}$, then the fibre $\hat{p}(x)$ is a left $\pi_1^{et}(\bar{X})$-torsor. Fix an algebraic closure $K \hookrightarrow \bar{K}$. Then there is also an action of the absolute Galois group of $K$, $G_K$ on the fibre $\hat{p}^{-1}(x)$. This can be seen model-theoretically as follows:

Consider the structure
$$\bar{K}_K := \langle \bar{K}, + , \cdot , \{ c_i\}_{i \in K} \rangle$$
of \S \ref{natlang}. For each finite \'etale cover $p:Y \ra \bar{X}$, fix an embedding of the $\bar{K}$-points of $Y$, $Y(\bar{K})$,
into $\bar{K}_K$ as a definable set (along with the covering maps). 
Let $L$ be a field containing $K$, and the coordinates of the point $x$. Then $\Aut(\bar{L}_L) \cong G_L$, and the fibre $\hat{p}(x)$ is a `pro-definable set'. \newline

Since the fibre $\hat{p}^{-1}(x)$ is a left $\pi_1^{et}(\bar{X})$-torsor, for $\sigma \in G_L$ and $y \in \hat{p}^{-1}(x)$ there is $g_{\sigma} \in \pi_1^{et}(\bar{X})$ such that $y \sigma = g_{\sigma}(y)$. Since the elements of $\Aut_{\Fet(\bar{X})}(Y / \bar{X})$ are defined over $L$, the actions commute and we get a continuous homomorphism
\begin{align*}
\rho_{(x,y)} : G_L & \ra \pi_1^{et}(\bar{X}) \\
\sigma & \mapsto g_{\sigma}.
\end{align*}

\begin{rem}
It is important to note that given a $k$-curve as above, to get an action of Galois on fibres in the pro-\'etale cover, it is unnecessary to extend the base field to include the fields of definition of all finite \'etale covers. However, this is possibly most naturally seen scheme-theoretically. If $L$ is a field containing $k$, and the coordinates of $x$, then in general (as shown in \ref{secgalpro}), you will get a continuous cocycle from $G_L$ to $\pi_1^{et}(\bar{X}) $. To see this model-theoretically, we can use the same setting as a above, but just consider the action of $G_L$ on the structure $\bar{K}_K$.
\end{rem}

\section{Quotients of the upper half plane}\label{appmod}

In this section, we define Riemann surfaces to be connected.

\begin{dfn}
Define the (complex) \textit{upper half plane}
$$\H:= \{ z \in \C \ : \  \Im(z) >0 \},$$
and the \textit{extended upper half plane} 
$$\H^*:= \H \cup \P^1(\Q)=\H \cup \Q \cup \infty.$$
The points of $\P^1(\Q)$ are called \textit{cusps}.
We define a (Hausdorff) topology on $\H^*$ as follows: For $z \in \H$ just use a usual complex open neighbourhood contained in $\H$. For the cusp $\infty$, take the sets
$$\{ z  \in \H \ : \ \Im(z) > \epsilon \} \cup \{ \infty \}$$
for every $\epsilon >0$ as a basis of open sets.  For a cusp $c \neq \infty$ take the interiors of circles in $\H$ tangent to the real axis at $c$, along with $c$ as a basis of open neighbourhoods. The extended upper half plane is connected in this topology.
\end{dfn}

$\psl2(\R)$ acts as the group of holomorphic automorphisms on the upper half plane, and also acts transitively on $\R\cup \{ \infty \}$. Define the \textit{modular group} $\Gamma:= \psl2(\Z)$. We now collect together some results regarding quotients of the extended upper half plane by discrete subgroups of $\psl2(\R)$. The standard reference is  \cite{shimura1971introduction}.

\begin{thm} \cite[\S 1.4]{shimura1971introduction}
$\Gamma \backslash \H^*$ is compact.
\end{thm}

\begin{thm} \cite[\S 1.5]{shimura1971introduction}
If $G$ is a discrete subgroup of $\psl2(\R)$ then $G \backslash \H^*$ is a Riemann surface.
\end{thm}

\begin{thm} \cite[1.36]{shimura1971introduction}
Let $G$ and $G'$ be discrete subgroups of $\psl2(\R)$ with $G'$ of finite index in $G$. Then the natural map
$$p: G' \backslash \H^* \ra G \backslash \H^*$$
is holomorphic.
\end{thm}

From the above two theorems we have the following:

\begin{cor}
If $\Gamma'$ is a finite index subgroup of $\Gamma$ then $\Gamma' \backslash \H^*$ is a compact Riemann surface.
\end{cor}

\begin{proof}
The natural map
$$p: \Gamma' \backslash \H^* \ra \Gamma \backslash \H^*$$
is a finite branched covering map and is therefore proper.
\end{proof}

So since $\Gamma \backslash \H^*$ is a compact Riemann surface, and is therefore algebraic, we may produce many more algebraic curves by quotienting $\H^*$ by finite index subgroups of $\Gamma$ to produce finite branched covers of $\Gamma \backslash \H^*$. Since these algebraic curves arose from geometry, they carry extra information. Of particular interest are those which arise by quotienting by certain arithmetic subgroups of the modular group $\Gamma$.

\subsection{Congruence subgroups of $\Gamma$}

Let  $\Gamma(N)$ be the principal congruence subgroup of level $N$. i.e.
$$\Gamma(N)=\left\{ \gamma =\begin{pmatrix} a&b\\ c&d \end{pmatrix} \in \Gamma \ \ | \ \ a \equiv d \equiv 1  \textrm{ and } b \equiv c \equiv 0 \mod N \right\},$$
and let
$$\Gamma_0(N)=\left\{ \gamma =\begin{pmatrix} a&b\\ c&d \end{pmatrix} \in \Gamma \ \ | \ \  c \equiv 0 \mod N \right\}.$$
Define the affine modular curves
$$Y_0(N):= \Gamma_0(N) \backslash \H \ \ , \ \ Y(N):= \Gamma(N) \backslash \H.$$

They can be compactified by adding in `cusps' i.e.
$$X_0(N):= \Gamma_0(N) \backslash \H \cup \P^1(\Q) \ \ , \ \ X_0(N):= \Gamma_0(N) \backslash \H \cup \P^1(\Q).$$
These compactified modular curves are compact Riemann surfaces and therefore are algebraic. However the good thing about quotienting by a congruence subgroup of $\Gamma$ is that you get a model over something nice. For example $X_0(N)$ has a model $X_0(N)_{\Q}$ over $\Q$ (via the modular polynomial $\Phi_N$ - see \cite[\S 7]{milnemodular}) and $X(N)$ has a model over $\Q(\mu_N)$. Modular curves parametrise elliptic curves with some extra torsion data, and in particular they are the solutions to various moduli problems. See \cite[\S 8]{milnemodular} for a discussion of moduli problems.



\section{Tate modules}\label{app tate}

Define $\hat{\Z}$ to be the limit inverse of additive groups
$$\hat{\Z}:=\varprojlim_n \Z / n \Z$$
(with the maps being reduction mod $n$). Being the inverse limit of finite groups, it is a profinite group i.e. a Hausdorff, compact and totally disconnected topological group (the topology is inherited from the product topology on $\prod \Z / n \Z$ where each $\Z / n \Z$ has the discrete topology). It can be seen as a `horizontal' product of `vertical' things i.e.
$$\hat{\Z} \cong \prod_l \Z_l$$
where the isomorphism is as topological rings and the `vertical' $\Z_l$ are the $l$-adic integers
$$\Z_l:= \varprojlim_n \Z / l^n \Z.$$

In profinite groups, being open is equivalent to being closed and of finite index. Sets of the form $\pi_n^{-1}(U_n)$ form a basis of open sets (where $U_n$ is open in $\Z / n \Z$) and where $\pi_n$ is the projection onto the $n$'th factor i.e. reduction mod $n$.
Closed subgroups of $\Z_l$ are $0$ and $l^n \Z_l$ and so the open subgroups of $\Z_l$ are all of the form $l^n \Z_l$ for some $n$.

For the multiplicative group $\C^{\times}$ we define the $l$-adic Tate module 
$$T_l(\C^{\times}) := \varprojlim_n  \mu_{l^n}$$
where $\mu_{l^n}$ are the $l^n$th roots of unity. It is a free $\Z_l$-module of rank 1. More concretely, as a set we have
$$T_l(\C^{\times})= \{ (a_1,a_2,...) \  | \ a_i^{l^i}=1 , \ a_{i+1}^{l}=a_i  \}.$$
We also define the Tate module
$$T(\C^{\times}):= \varprojlim_N \mu_N,$$
which is a free $\hat{\Z}$-module of rank 1. \newline

 Similarly, for an abelian variety $A$ we define the $l$-adic Tate module
$$T_l(A):= \varprojlim_n(A[l^n])$$
where $A[N]$ is the $N$-torsion, and the Tate module
$$T(A):= \varprojlim_N A[N].$$

\chapter{Quasiminimal excellence}\label{app quasi}

The aim of this chapter, is to give an exposition of some of the background regarding the concept of quasiminimal excellence, which is the main tool used in demonstrating the categoricity results of this thesis. We assume the reader has a basic knowledge of model theory, and is comfortable with the following notions: Formula, (many-sorted) structure, theory, model, type (denoted $\tp(x)$), quantifier free type ($\qftp(x)$), quantifier elimination, model theoretic definable closure ($\dcl(x)$), model-theoretic algebraic closure ($\acl(x)$), saturation,  homogeneity, the Lowenheim-Skolem theorem of first-order logic, and the compactness theorem of first-order logic. Marker's book \cite{marker2002model} is the standard reference for these notions, and is very accessible to algebraic geometers.

\section{The infintary logic $\cL_{\omega_1, \omega}$}

The logic $\cL_{\omega_1,\omega}$, is the most basic infinitary extension of first-order logic. In the language of first-order logic you are only allowed to form finite conjunctions and disjunctions, whereas in $\cL_{\omega_1,\omega}$ you can form countable ones. One of the main differences with first- order logic is the following:

\begin{pro}
The compactness and upward Lowenheim-Skolem theorems fail in $\cL_{\omega_1,\omega}$.
\end{pro}

\begin{proof}
For upward LS, consider a language with constants $c_0,c_1,c_2,...,c_{\omega}$. Then the sentence
$$\forall x \bigvee_{n< \omega} \left( x=c_n \right)$$
has a model of power $\omega$ but no uncountable model. For compactness, note that every finite subset of the set of sentences
$$\forall x \bigvee_{n< \omega} \left( x=c_n \right) \bigcup_{n< \omega} \{c_{\omega} \neq c_n \}$$
has a model but the whole thing does not.
\end{proof}

\subsection{Strong Minimality}

Quasiminimal excellence is a generalisation of strong minimality to infinitary logics, so we review this first. \newline

Let $T$ be a complete theory with infinite models, in a countable first-order language. Strong minimality is a way of showing that a first-order theory is categorical in uncountable cardinalities.

\begin{dfn}
Let $\cM$ and $\cN$ be structures.

\noindent A \textit{finite partial monomorphism} between $\cM$ and $\cN$ is a pair of (non-empty) tuples $(\bar{m},\bar{n})$ such that $\qftp_{\cM}(\bar{m})=\qftp_{\cN}(\bar{n})$, or equivalently if $\bar{m}$ and $\bar{n}$ generate isomorphic substructures.

\noindent A map $f:A \ra N$ is called a \textit{partial elementary monomorphism} if
$$\cM \models \phi(\bar{a}) \Lra \cN \models \phi(f(\bar{a}))$$
for all finite tuples $\bar{a} \subseteq A$. A partial elementary monomorphism from $A \subseteq M$ onto $B \subseteq N$ is called a \textit{partial elementary isomorphism} between $A$ and $B$.
\end{dfn}

\begin{pro}[Uniqueness of Closure]
Any elementary isomorphism between $A \subseteq M$ and $B \subseteq N$ can be extended to an elementary isomorphism between $\acl(A)$ and $\acl(B)$.
\end{pro}
The proof of the above rests on the fact that the type of an algebraic element is isolated, and is therefore realised in every model, so we can extend the elementary isomorphism via transfinite induction.

\begin{dfn}\label{def pregeom}
A \textit{pregeometry} is a set $X$ and $\cl:\cP(X)\ra \cP(X)$ (where $\cP$ denotes the powerset) such that
\begin{enumerate}
\item [PG1] $A \subseteq \cl(A)$
\item [PG2] $A \subseteq B  \Ra \cl(A) \subseteq \cl(B)$
\item [PG3] $\cl(\cl(A))=\cl(A)$
\item [PG4] $\cl(A)= \bigcup \cl(A')$ (for finite subsets $A'$ of $A$)
\item [PG5] (Exchange) $ a \in \cl(Ab)-\cl(A) \Ra b \in \cl(Aa)$.
\end{enumerate}
A pregeometry in which points are closed ($\cl\{a\}=\{a\}$), and $\cl(\emptyset)=\emptyset$) is called a \textit{geometry}.
\end{dfn}

\noindent Note that properties 1-4 hold for the usual $\acl$ operator (where $a \in \acl(A)$ if $a$ is one of only finitely many realisations of a formula with parameters from $A$).

\begin{dfn}
Given a pregeometry $\cl$, a subset $A \subseteq M$ is said to be $\cl$-\textit{independent} if $a \notin cl(A - a)$
for all $a \in A$. If $C \subseteq M$,  $A$ is $\cl$-\textit{independent over} $C$ if $a \notin
\cl(C \cup (A - a))$ for all $a \in A$.
A $\cl$-independent subset $A$ of $Y$ is a \textit{basis} for $Y\subseteq M$ if $\cl(A) = \cl(Y)$.
\end{dfn}
Where it is clear we are talking about a pregeometry $\cl$, we will just write independent instead of $\cl$-independent.

\begin{lem}
If $\cl$ is a pregeometry then any two $\cl$-bases B and C of a set A are of the same cardinality.
\end{lem}
\noindent The proof rests on the exchange principle holding. In light of the lemma, we can now define the dimension of $A$, $\dim(A)$ to be the cardinality of a basis of $A$. Recall that a subset $A$ of a structure $\cM$ is said to be \textit{indiscernible over}
$B$ if $\tp(\bar{a}/B)$ = $\tp(\bar{a}' / B)$ for any two $n$-tuples of distinct elements of A for
any finite $n$, and that a structure is called \textit{strongly minimal} if every definable set in every elementary extension is finite or cofinite. In a strongly minimal structure, $\acl$ is a pregeometry (i.e. the \textit{exchange property} (5) is satisfied) and therefore in strongly minimal theories $\acl$-bases have the same cardinality and we have a notion of dimension. Now an infinite independent set is a set of indiscernibles of the same type (where the type of a set of indiscernibles is the set of formulae satisfied by every finite tuple), so any bijection between bases is actually a partial elementary isomorphism and we have the following.

\begin{thm}
If $T$ is strongly minimal and $\cM,\cN \models T$,
then $\cM \cong \cN$ $\Lra$ $\dim(M) = \dim(N)$.
\end{thm}
So the models of strongly minimal theories are determined up to isomorphism by their dimension. If $X$ is uncountable then $\dim(X)=|X|$ since the language is countable and $\acl(Y)$ is countable for any $Y \subseteq X$ (i.e. $\acl$ has the \textit{countable closure property}) and we get:

\begin{cor}
If $T$ is strongly minimal then $T$ is $\kappa$-categorical for $\kappa \geq \aleph_1$.
\end{cor}

\subsection{Quasiminimality}\label{app quasi}
This section is just an amalgamation of Baldwin's and Kirby's expositions in \cite{baldwin2009categoricity} and \cite{kirby2010quasiminimal}.

\begin{dfn}
Given a structure $\cM$ with closure operator $\cl$, we say that $(\cM,\cl)$ satisfies the \textit{countable closure property} if $|\cl(X)|\leq \aleph_0$ for any finite $X \subseteq M$.
\end{dfn}

\begin{dfn}
If $\cK$ is a class satisfying Conditions I and II below, then $\cK$ is called \textit{weakly quasiminimal}.
\end{dfn}

\textbf{Condition I: (Pregeometry)}
\begin{enumerate}
\item For each $H \in \cK$ $\cl_H$ is a pregeometry on $H$ satisfying the countable
closure property.
\item For each $X \subseteq H$, $\cl_H(X)$ with the induced relations is in $\cK$.
\item If $\qftp_H(y/X)=\qftp_{H'}(y'/X')$ then $y \in \cl_H(X)$ iff $y' \in \cl_{H'}(X')$.
\end{enumerate}

\noindent \textbf{Condition II: ($\aleph_0$-homogeneity over countable models and $\emptyset$)}

\noindent Let $G \subseteq H,H'$ be empty, or a countable member of $\cK$ that is closed in $H,H'$.
\begin{enumerate}
\item If $f$ is a partial $G$-monomorphism from $H$ to $H'$ with finite domain $X$ then for any $y \in \cl_H(X)$ there is $y' \in H'$ such that $f\cup\{\langle y,y'\rangle \}$ is a partial $G$-monomorphism.
\item If $X \subseteq H$ and $X' \subseteq H'$ are $\cl$-independent over $G$ and $f:X \ra X'$ is a bijection, then $f$ is a partial $G$-monomorphism.
\end{enumerate}

\begin{lem}\label{lem_quasi one indep type}
If $\cK$ is weakly quasiminimal and $M \in \cK$ then for any finite set $X \subset M$ then if $a,b \in M-\cl_M(X)$ then $a$ and $b$ realise the same $\cL_{\omega_1,\omega}$-type over $X$.
\end{lem}
For a proof of the above see \cite{baldwin2009categoricity}. The next lemma justifies the name quasiminimal.

\begin{lem}
Let $\cK$ be a weakly quasiminimal class and $H \in \cK$. Then every $\cL_{\omega_1,\omega}$-definable subset of $H$ is countable or cocountable. As a result, $a \in \cl_H(X)$ iff $a$ satisfies some $\varphi$ over $X$ which has at most countably many solutions.
\end{lem}

\begin{proof}
Suppose that there was an $\cL_{\omega_1,\omega}$-formula $\varphi$ with parameters from $X$ such that $|\varphi(H)|, |\neg \varphi(H)| \geq \aleph_1$. Then since $\cl(X)$ is countable, there is $a \in \varphi(H)$ and $b \in \neg \varphi(H)$ such that $a,b \notin \cl(X)$ and this contradicts Lemma \ref{lem_quasi one indep type}.
\end{proof}

\begin{lem}\label{lem_increase G mono}
Suppose a class $\cK$. Let $G \in \cK$ be countable and suppose $G \subseteq H, H'$ with $G$ closed in $H,H'$ if $G \neq \emptyset$. Then if $X \subseteq H$ and $X' \subseteq H'$ are finite and $f$ is a partial $G$-monomorphism from $X$ onto $X'$ then $f$ extends to a partial $G$-monomorphism $f^*$ from $\cl_H(G \cup X)$ onto $\cl_{H'}(G \cup X')$.
\end{lem}

\begin{proof}
Since $\cl(X)$ and $\cl(X')$ are countable, we may choose an ordering of both of length $\omega$ (with $X$ and $X'$ at the start of each) and inductively construct the map $f^*$ by the back and forth method:

\noindent Let $h_0=f$. For odd $n$ let $a$ be the least element of $\cl(X)- dom(h_{n-1})$. By condition II.1 we can find $b \in H'$ such that $h_{n} := h_{n-1}\cup \{ (a, b)\}$ is a partial $G$-monomorphism, and by I.3 $b \in \cl_{H'}(X')$. At even stages go back and extend the range using the same argument. Let $f^* = \cup_{n \in \N} h_n$.
\end{proof}

\begin{lem}
Suppose the conditions of Lemma \ref{lem_increase G mono} hold. If $B$ is an independent set (over $G$) of cardinality no greater than $\aleph_1$, and $f$ is a partial $G$-monomorphism from $B$ onto $B'$ then $f$ extends to a partial $G$-monomorphism from $\cl_H(G \cup B)$ onto $\cl_{H'}(G \cup B')$
\end{lem}

\begin{proof}
Well order $B$ as $(b_{\lambda})_{\lambda < \mu}$ for an ordinal $\mu \leq \omega_1$, and for each ordinal $\nu \leq \mu$ let $G_{\nu}=\cl_H(G \cup \{ b_{\lambda} \  |  \ \lambda < \nu\})$ and let $G_0=G$.

\noindent For a each ordinal $\nu \leq \mu$, we construct a partial $G$-monomorphism $f_{\nu}$ extending $f$ from $G_{\nu}$ onto $G'_{\nu}:= \cl_{H'}(f_{\lambda}(G_{\lambda})\cup \{f(b_\lambda) \})$ such that $\nu_1 \leq \nu_2 \Ra f_{\nu_1} \subseteq f_{\nu_2}$:

\noindent For a successor ordinal $\nu = \lambda + 1$, apply Lemma \ref{lem_increase G mono} with $G=G_{\lambda}=G'_{\lambda}$ and $X=\{b_\lambda \}$. At limit ordinals take unions.
\end{proof}

\begin{cor}
Models of a weakly quasiminimal class of dimension less than or equal to $\aleph_1$ are determined up to
isomorphism by their dimension, and in particular there is at most one model of cardinality $\aleph_1$.
\end{cor}

Structures in a weakly quasiminimal class are homogeneous in the follwing sense:

\begin{lem}
Let $\cK$ be a weakly quasiminimal class, let $G \subseteq H$ be empty or countable and closed, and let $\bar{x}$ and $\bar{y}$ be tuples of the same length from $H$ such that $\qftp(\bar{x}/G)=\qftp(\bar{y}/G)$. Then there is an automorphism of $H$ fixing $G$ and sending $\bar{x}$ to $\bar{y}$.
\end{lem}

So we have seen that to get categoricity in cardinality $\aleph_1$, we just need to show that our theory is weakly quasiminimal. However, in our applications, we want to know that the standard model is the unique model of cardinality continuum (without assuming the continuum hypothesis). Previously, the standard way of doing this was to show that the class satisfied the conditions of excellence (Condition III of \cite{kirby2010quasiminimal} for example). However, we are lucky in that here has been a very recent big advance in the theory of quasiminimal excellent classes in a five author paper \cite{bays2012quasiminimal}, which basically says that if the structures in our class all have infinite dimension with respect to the pregeometry, then excellence follows automatically from the conditions of quasiminimality. Unlike the conditions for quasiminimality, the excellence conditions are not particularly natural for algebraic-geometers, so this is why I omitted them completely and just referred the reader to \cite{kirby2010quasiminimal}.

\subsection{Recent advances in the theory of quasiminimal excellence}\label{app bhhkk}

Most of the following section is lifted directly from \cite{bays2012quasiminimal}.

\begin{dfn} Let $\cM$ be an $\cL$-structure for a countable language $\cL$, equipped with a pregeometry $\cl$ (or $\cl_{\cM}$ if it is necessary to specify $\cM$). We say that $\cM$ is a \textit{quasiminimal pregeometry structure} if the following hold:\newline

\begin{itemize}
\item [QM1] The pregeometry is determined by the language. That is, if $\qftp(a,\bar{b}) = \qftp(a′, \bar{b}′)$ then $a \in \cl(\bar{b})$ iff $a' \in \cl(\bar{b}′)$.
\item [QM2] $\cM$ is infinite-dimensional with respect to $\cl$.
\item [QM3] (Countable closure property) If $A \subset M$ is finite then $\cl(A)$ is countable.
\item [QM4] (Uniqueness of the generic type) Suppose that $H, H' \subseteq M$ are countable closed subsets, enumerated such that $\qftp(H) = \qftp(H′)$. If $a \in M-H$ and $a′ \in M -H′$ then $\qftp(H, a) = \qftp(H′, a′)$ (with respect to the same enumerations for $H$ and $H′$).
\item [QM5] ($\aleph_0$-Homogeneity over closed sets and the empty set)
Let $H,H′ \subseteq M$ be countable closed subsets or empty, enumerated such that $\qftp(H) = \qftp(H′)$, and let $ \bar{b}, \bar{b}′$ be finite tuples from $M$ such that $\qftp(H, b) = \qftp(H′,b′)$, and let $a \in \cl(H, b)$. Then there is $a′ \in M$ such that $\qftp(H,\bar{b},a) = \qftp(H′,\bar{b}′,a′)$
\end{itemize}
\end{dfn}
We say $\cM$ is a weakly quasiminimal pregeometry structure if it satisfies all the axioms except possibly QM2. Given $\cM_1$ and $\cM_2$ both weakly quasiminimal pregeometry $\cL$-structures, we say that an $\cL$-embedding $\theta : \cM_1 \hookrightarrow \cM_2$ is a closed embedding if $\theta (\cM_1)$ is closed in $\cM_2$ with respect to $\cl_{\cM_2}$, and $\cl_{\cM_1}$ is the restriction of $\cl_{\cM_2}$ to $\cM_1$. We write $\cM_1 \preccurlyeq_{\cl} \cM_2$ for a closed embedding.\newline

Given a quasiminimal pregeometry structure $\cM$, let $\cK^-(\cM)$ be the smallest class of $\cL$-structures which contains $\cM$ and all its closed substructures and is closed under isomorphism, and let $\cK(\cM)$ be the smallest class containing $\cK^-(\cM)$ which is also closed under taking unions of chains of closed embeddings. Then both $\cK^-(\cM)$ and $\cK(\cM)$ satisfy axioms 0, I, and II of quasiminimal excellent classes from \cite{kirby2010quasiminimal}, and $\cK(\cM)$ also satisfies axiom IV and, together with closed embeddings, forms an abstract elementary class. We call any class of the form $\cK(\cM)$ a \textit{quasiminimal class}. Note that the homogeneity condition here is weaker than the one in the previous section since it works in one model only.
The main result of \cite{bays2012quasiminimal} is:
\begin{thm} If $\cK$ is a quasiminimal class then every structure $A\in \cK$ is a weakly quasiminimal pregeometry structure, and up to isomorphism there is exactly one structure in $\cK$ of each cardinal dimension. In particular, $\cK$ is uncountably categorical. Furthermore, $\cK$ is the class of models of an $\cL_{\omega_1,\omega}(Q)$ sentence.
\end{thm}

\bibliographystyle{alpha}
\bibliography{adambooks}

\end{document}